\documentclass[11pt]{article}
\usepackage[utf8]{inputenc}
\usepackage{array}
\usepackage{amsmath, amsthm}
\usepackage{amsfonts, dsfont}
\usepackage{xcolor}
\usepackage{appendix}
\usepackage[margin=2.5cm]{geometry} 
\usepackage{enumitem}
\usepackage{graphicx}
\usepackage{bbm}
\usepackage{mathtools}
\usepackage{float}
\usepackage[colorlinks=true, allcolors=blue]{hyperref}
\usepackage{doi} 

\mathtoolsset{showonlyrefs}

\title{Splitting methods for stochastic Hodgkin-Huxley type systems\\ and a localized fundamental mean-square convergence theorem}
\author{Pierre Étoré\footnote{pierre.etore@univ-grenoble-alpes.fr, Univ. Grenoble Alpes, CNRS, Grenoble INP (Institute of Engineering Univ. Grenoble Alpes), LJK, 38000 Grenoble, France}, Anna Melnykova\footnote{anna.melnykova@univ-avignon.fr, Université d’Avignon, Laboratoire de Mathématiques d’Avignon, UPR 2151, Avignon, 84000, France}, Irene Tubikanec\footnote{irene.tubikanec@jku.at,
       Institute of Applied Statistics,
       Johannes Kepler University Linz, 4040 Linz, Austria}}
\date{}

\newcommand{\balpha}{\boldsymbol\alpha}

\newcommand{\bbeta}{\boldsymbol\beta}

\newcommand{\bsigma}{\boldsymbol\sigma}

\newcommand\norm[1]{\left\lVert#1\right\rVert}

\newcommand{\1}{\mathds{1}}

\newcommand{\E}{\mathbb{E}}

\renewcommand{\P}{\mathbb{P}}

\newcommand{\R}{\mathbb{R}}

\newcommand{\cF}{\mathcal{F}}
\newcommand{\cG}{\mathcal{G}}

\newcommand{\ds}{\displaystyle}

\newtheorem{remark}{Remark}[section]
\newtheorem{lemma}{Lemma}[section]
\newtheorem{proposition}{Proposition}[section]
\newtheorem{theorem}{Theorem}[section]
\newtheorem{definition}{Definition}[section]

\usepackage[comma]{natbib}
\begin{document}

\maketitle

\begin{abstract}

Existing fundamental theorems for mean-square convergence of numerical methods for stochastic differential equations (SDEs) require globally or one-sided Lipschitz continuous coefficients, while strong convergence results under merely local Lipschitz conditions are largely restricted to Euler-Maruyama type methods. To address these limitations, we introduce a novel localized version of the fundamental mean-square convergence theorem for SDEs with locally Lipschitz coefficients, which naturally arise in a wide range of applications. Specifically, we show that if a numerical scheme is locally consistent in the mean-square sense of order $q>1$, then it is locally mean-square convergent with rate $q-1$. Building on this result, we further prove that global mean-square convergence follows, provided that both the exact solution and its numerical approximation admit bounded $2p$th moments for some $p>1$. 

These new convergence results are illustrated on a class of locally Lipschitz SDEs of Hodgkin-Huxley type, characterized by a conditionally linear drift structure. For these systems, we construct different Lie-Trotter and Strang splitting methods exploiting their conditional linearity. The proposed convergence framework is then applied to these schemes, requiring innovative proofs of local consistency and boundedness of moments. In addition, we establish key structure-preserving properties of the splitting methods, in particular state-space preservation and geometric ergodicity. Numerical experiments support the theoretical results and demonstrate that the proposed splitting schemes significantly outperform Euler-Maruyama type methods in preserving the qualitative features of the model.

\end{abstract}

\vspace{0.4cm}

\vspace{0.4cm}

\section{Introduction}\label{section:intro}

Stochastic differential equations (SDEs) with non-globally Lipschitz coefficients arise naturally as models in many applications, including mathematical biology, neuroscience, finance, and physics. In this paper, we study a class of $(d+1)$-dimensional SDEs of the form
\begin{align}\label{eq:HH_redefined}
    \begin{split}
        dV_t &=  \left(a\bigl(U_t\bigr)V_t + b\bigl( U_t \bigr)\right)dt + d\zeta_t, \\
        dU_t &=  \left(-\textbf{diag}\left( \balpha\bigl( V_t \bigr) + \bbeta\bigl( V_t \bigr) \right) U_t +\balpha \bigl( V_t \bigr)\right)dt, 
    \end{split}  
\end{align}
where $\textbf{diag}(x)$ denotes the diagonal matrix with diagonal entries given by the components of the vector $x$, $a:\R^d\to (-\infty,0)$, $b:\R^d\to\R$ and $\balpha,\bbeta:\R\to(0,\infty)^d$ are $C^1$ functions, and $d\zeta_t$ represents suitable stochastic perturbation terms. The drift coefficient of System \eqref{eq:HH_redefined} is neither globally nor one-sided Lipschitz continuous. However, it satisfies a \textit{local Lipschitz condition} and exhibits a characteristic \textit{conditional linear structure}, meaning that the drift is linear in each component when the respective other variables are held fixed.

A prominent motivating example for this class of SDEs is the Hodgkin–Huxley model, which is one of the best-known and most widely used mathematical models for describing neuronal activity. This model was first introduced by \cite{Hodgkin1952} as a $4$-dimensional ordinary differential equation (ODE), describing the dynamics of a giant squid neuron. In this framework, a membrane potential variable $V$ is coupled with $d=3$ gating variables $U=(n,m,h)^T$, which represent the ion channel subunits.

To account for channel fluctuations and intrinsic variability, several stochastic extensions of the Hodgkin-Huxley model have been proposed and investigated; see, e.g., \cite{Baladron2012,Bossy2015,Hoepfner2016, Hoepfner2016b}. In particular, conductance noise models, in which stochastic perturbations act only on the membrane potential variable $V$, have been shown to provide a more accurate reproduction of neuronal dynamics than models in which noise acts on the gating variables; see \cite{Goldwyn2011,Goldwyn2011b} for comparison studies, and \cite{Fox1994,Fox1997} for examples of the latter class. Following these findings, we consider two conductance noise variants of System \eqref{eq:HH_redefined}: (i) a formulation in which $\zeta_t$ is modelled as an Ornstein-Uhlenbeck process, leading to additive noise at the cost of an increased system dimension (see, e.g., \cite{Hoepfner2016}), and (ii) multiplicative noise extensions that retain the original system dimension.

Closed-form solutions of System \eqref{eq:HH_redefined} are generally unavailable, making numerical methods indispensable. In this work, we address the main challenges associated with the numerical approximation of \eqref{eq:HH_redefined}, namely the construction of suitable schemes, their structure preservation, and their strong (mean-square) convergence.

The numerical approach adopted in this paper is based on splitting methods, which decompose the SDE of interest into exactly solvable subsystems, derive their exact solutions, and combine them in an appropriate manner. We refer to \cite{Mclachlan2002} and \cite{Blanes2009} for comprehensive discussions of splitting methods for ODEs, and to \cite{Misawa2001,Shardlow2003,Leimkuhler2015,Ableidinger2016,Alamo2016,Brehier2018,Chevallier2020,Buckwar2022,KellyLord2023,Foster2024,Ditlevsen2025,JansLord2025} for selected extensions to~SDEs. 

Although the stochastic Hodgkin-Huxley system is extensively studied in the literature and widely used in computational neuroscience, to the best of our knowledge, structure preserving numerical methods for this setting have not yet been investigated. Their development is nontrivial due to both probabilistic and numerical challenges arising from the local Lipschitz framework. For the deterministic Hodgkin-Huxley system, structure-preserving numerical splitting methods that exploit the system's conditional linearity have been proposed by \cite{Chen2020}. Building on this approach, we extend the methodology to stochastic Hodgkin-Huxley type systems \eqref{eq:HH_redefined}. Specifically, for both stochastic variants considered here, we construct numerical splitting schemes based on Lie-Trotter and Strang compositions.

In contrast to Euler-Maruyama type methods, the proposed splitting schemes are inherently structure-preserving. In particular, we prove that they are geometrically ergodic and preserve the natural state space of the model, ensuring that the gating variables $U$ remain within $[0,1]^d$. Numerical experiments demonstrate that the splitting schemes are markedly superior in preserving the qualitative dynamics of neuronal spiking and allow for significantly larger time steps, resulting in substantial computational savings, a feature of particular  relevance in computational neuroscience. 

While structure preservation follows naturally from the construction of the splitting methods, establishing mean-square convergence is considerably more challenging. Classical strong convergence results for numerical methods rely on the fundamental theorem of mean-square convergence \citep{Milstein2004}, which assumes globally Lipschitz coefficients, or on its extension to the one-sided Lipschitz case \citep{Tretyakov2013}. A key advantage of the fundamental theorem is that it reduces the problem of strong convergence to controlling the one-step error, a property commonly referred to as \textit{consistency}. However, neither of these theorems applies to the class of locally Lipschitz SDEs considered here. An influential contribution to the locally Lipschitz setting was provided by \cite{Higham2002}, who established the convergence of the Euler-Maruyama method under the local Lipschitz condition, assuming that the $2p$th moments of both the exact solution and the numerical approximation are bounded for some $p>1$. This result, however, does not provide  information on the rate of convergence, and its proof cannot be directly extended to splitting methods, whose structure is fundamentally different from that of Euler-Maruyama type methods.

Motivated by these limitations, we propose a novel localized version of the fundamental theorem of mean-square convergence. Specifically, we show that if the coefficients of an SDE are locally Lipschitz and a numerical scheme for that SDE is \textit{locally consistent} of order $q>1$, then the scheme is locally mean-square convergent with rate $q-1$. Building on this result, we adapt the ideas of \cite{Higham2002} to show that, if in addition the exact solution and the numerical scheme admit bounded $2p$th moments for some $p>1$, global mean-square convergence follows. While this theorem does not provide a global convergence rate, the local rate $q-1$ gives an indication of the actual speed of convergence. Indeed, this rate matches the empirical rate observed in our numerical experiments on stochastic Hodgkin-Huxley systems, where we have $q-1=1/2$. 

Importantly, the proposed convergence result is not limited to conditionally linear stochastic systems of the form \eqref{eq:HH_redefined}. It also applies to general SDEs with locally Lipschitz coefficients and can be used for convergence analysis of a broad class of numerical methods beyond the setting considered~here. 

In the context of splitting methods, the application of splitting techniques to conditionally linear SDEs has recently begun to attract attention (see, e.g., \cite{Souli2023,Jovanovski2025}). Existing works, however, either rely on classical fundamental theorems or do not provide convergence analyses. Moreover, the splitting methods proposed for deterministic Hodgkin-Huxley type systems \citep{Chen2020} are analysed in terms of limit cycle preservation but not accompanied by any convergence investigations. Since our framework covers the deterministic setting, convergence in that case follows as a direct consequence of our analysis. Finally, classical fundamental convergence theorems have been extensively used recently in the literature to establish convergence of splitting methods, see, e.g. \cite{Ableidinger2017,Chevallier2020,Tubikanec2020} for globally Lipschitz systems and \cite{Buckwar2022,Souli2023,Pilipovic2024} for extensions to the one-sided Lipschitz setting. These aspects further highlight the relevance of the present work for the convergence analysis of splitting methods for locally Lipschitz systems and, more broadly, for general numerical methods applied to such problems.

The contributions of this work are presented as follows:
\begin{itemize}
\item In Section~\ref{section:model}, we introduce the two considered classes of stochastic conditionally linear systems of type \eqref{eq:HH_redefined} and recall the Hodgkin-Huxley model as a motivating example. We discuss existence and uniqueness of solutions, and prove that they have bounded moments. 
\item In Section~\ref{section:splitting_schemes}, we construct several Lie-Trotter and Strang splitting methods for the numerical approximation of these two classes of SDEs, exploiting their conditional linear structure.
\item In Section~\ref{sec:fund-thm}, we develop a novel localized fundamental theorem of mean-square convergence, together with a corresponding global mean-square convergence result, which are of general applicability.
\item In Section~\ref{sec:Splitting_properties}, we apply the new  results of Section~\ref{sec:fund-thm} to the proposed splitting methods, including innovative proofs of local consistency (of order $q=3/2$) and boundedness of moments, through continuous-time extensions of the schemes. In addition, we establish structure-preservation properties of the splitting methods, especially state-space preservation and geometric ergodicity. 
\item In Section~\ref{sec:num_exp}, we provide an extensive numerical study illustrating the theoretical findings, including a comparison with various state-of-the-art Euler-Maruyama type methods. 
\end{itemize}
The paper concludes with Section~\ref{sec:conclusion}, where we discuss possible applications and extensions. 

\paragraph*{Notation.}

Throughout, the following notations are used. By $0_d$ we denote the $d$-dimensional vector of zeros, and by $\mathbb{I}_d$ the $d\times d$-dimensional identity matrix. For a vector $x \in \mathbb{R}^d$, \textbf{diag}$(x)$ denotes the diagonal matrix whose diagonal entries are given by the corresponding elements of~$x$. For $s,t \in \mathbb{R}$, $s \wedge t$ and $s \vee t$ denote the minimum and maximum of $s$ and $t$, respectively. Finally,  the symbol $|\cdot|$ denotes the Euclidean vector norm or associated matrix norm, depending on the~context.

\section{Model}\label{section:model}

In this section, we present the class of SDEs under consideration. In Section \ref{sec:model_description}, two variants of System \eqref{eq:HH_redefined}, with different stochastic perturbation terms $d\zeta_t$, are introduced. In Section \ref{ssec:existence-moments}, we prove that they have a unique solution with bounded moments. 

\subsection{Description, first assumptions and notations}\label{sec:model_description}

Throughout, we suppose that $a:\R^d\to (-\infty,0)$, 
$b:\R^d\to\R$ and $\balpha,\bbeta:\R\to(0,\infty)^d$ are $C^1$ functions.
It will be convenient to use a component-wise notation for $\balpha$ and $\bbeta$, defining
\begin{align*}
    \balpha(v) &:=\left( \alpha_{1}(v),\ldots, \alpha_{d}(v) \right)^T, \quad \bbeta(v):=\left( \beta_{1}(v),\ldots, \beta_{d}(v) \right)^T,\quad v\in\R. 
\end{align*}
Besides a time horizon $0<T<\infty$ is given, alongside with a probability space $(\Omega,\mathcal{F},\mathbb{P})$ 
on which a standard Brownian motion  $(W_t)_{t \in [0,T]}$ is defined. 

We now introduce two different variants of System \eqref{eq:HH_redefined}. 

\vspace{0.3cm}
\noindent
\textbf{[BM]} Brownian motion-driven SDE: We set
$d\zeta_t:=\Sigma(U_t)dW_t$,
where it is assumed that $\Sigma:\R^d\to\R$ is of class $C^1$. Specifically, we examine the $(d+1)$-dimensional SDE
\begin{align}
\tag*{\eqref{eq:HH_redefined}\textbf{[BM]}}
\begin{split}
        dV_t &=  \left(a\bigl(U_t\bigr)V_t + b\bigl( U_t \bigr)\right)dt + \Sigma(U_t)dW_t, \quad 0\leq t\leq T, \\
        dU_t &=  \left(-\textbf{diag}\left( \balpha\bigl( V_t \bigr) + \bbeta\bigl( V_t \bigr) \right) U_t +\balpha \bigl( V_t \bigr)\right)dt, \quad 0\leq t\leq T,\\
        \end{split}
\end{align}
where the processes $V$ and $U:=(U^{1},\ldots,U^{d})^T$ are respectively one and  $d$-dimensional. We denote by $X=(V,U^T)^T$ the possible $(d+1)$-dimensional solution of \eqref{eq:HH_redefined}\textbf{[BM]}.

A real-valued random variable $V_0$ with $\mathbb{E}\big[ \left| V_0 \right|^{2p} \big]<\infty$ for all $p\geq 1$ and a $[0,1]^d$-valued random variable~$U_0$ are given. They are supposed to be independent of $W$ and will represent the starting points at time $t=0$ of $V$ and $U$ respectively.

We consider the natural filtration $(\cF^W_t)_{t\in[0,T]}$ of $W$, the filtration
$\cG_t=\cF^W_t\vee\sigma(V_0)\vee\sigma(U_0)$, for~$0\leq t\leq T$, and finally augment $(\cG_t)_{t \in [0,T]}$ with the null sets for $\P$. This gives the filtration $(\cF_t)_{t\in[0,T]}$, that satisfies the usual conditions (and to which $W$ is still adapted). 

\vspace{0.3cm}
\noindent
\textbf{[OU]} Ornstein-Uhlenbeck-driven SDE: We set $d\zeta_t:=dZ_t$, where $(Z_t)_{t\in[0,T]}$ is an Ornstein-Uhlenbeck process following the dynamics
\begin{equation}\label{eq:OU}
    dZ_t = \theta(\mu - Z_t)dt + \sigma dW_t,
\end{equation}
with $\theta,\sigma>0$, $\mu \in \mathbb{R}$. 

Specifically, plugging \eqref{eq:OU} into \eqref{eq:HH_redefined} gives the $(d+2)$-dimensional SDE 
\begin{align}
    \tag*{\eqref{eq:HH_redefined}\textbf{[OU]}}
    \begin{split}
        dV_t &=  \left(a\bigl(U_t\bigr)V_t + b\bigl( U_t \bigr) + \theta(\mu-Z_t) \right)dt + \sigma dW_t, \quad 0\leq t\leq T,\\
        dU_t &=  \left(-\textbf{diag}\left( \balpha\bigl( V_t \bigr) + \bbeta\bigl( V_t \bigr) \right) U_t +\balpha \bigl( V_t \bigr)\right)dt, \quad 0\leq t\leq T,\\
        dZ_t &=\theta(\mu-Z_t) dt + \sigma dW_t,\quad 0\leq t\leq T.
    \end{split}  
\end{align}
We denote by $X=(V,U^T,Z)^T$ the possible $(d+2)$-dimensional solution of \eqref{eq:HH_redefined}\textbf{[OU]}. 

Real-valued random variables $V_0$ and $Z_0$ with 
$\mathbb{E}\big[ \left| V_0 \right|^{2p} \big]<\infty$ and 
$\mathbb{E}\big[ \left| Z_0 \right|^{2p} \big]<\infty$
for all $p\geq 1$ and a $[0,1]^d$-valued random variable~$U_0$ are given. They are supposed to be independent of $W$ and will represent the starting points at time $t=0$ of $V$, $Z$ and $U$ respectively.

Similar to the previous case (using this time $\sigma(V_0)\vee\sigma(U_0)\vee \sigma(Z_0)$
instead of $\sigma(V_0)\vee\sigma(U_0)$), we define the filtration
 $(\cF_t)_{t\in[0,T]}$  satisfying the usual conditions. 

\begin{remark}
Note that, for example, \cite{Hoepfner2012} investigate a setting similar to System~\eqref{eq:HH_redefined}\textbf{[OU]}. To extend the system and its analysis to cases with multiplicative noise, while retaining a convenient representation and notation, we additionally introduce System~\eqref{eq:HH_redefined}\textbf{[BM]}.
\end{remark}

\subsection{Existence and uniqueness of solutions and moment bounds}\label{ssec:existence-moments}

The aim of this section is to prove that Systems \eqref{eq:HH_redefined}\textbf{[BM]} and
\eqref{eq:HH_redefined}\textbf{[OU]} each admit a unique strong solution with bounded moments. To the best of our knowledge, this result is new in the \textbf{[BM]}-case. In the \textbf{[OU]}-case, the existence result was given in \cite{Hoepfner2012}, but the boundedness of moments is a novel contribution. We first detail the proof in the \textbf{[BM]}-case, and then explain how to adapt the arguments to the \textbf{[OU]}-case.

\paragraph*{The \textbf{[BM]}-case.}

Note that System \eqref{eq:HH_redefined}$\textbf{[BM]}$ may be rewritten in the equivalent form
\begin{equation}
\label{eq:HH_Xfashion}
dX_t=\mathbf{b}(X_t)dt+ \bsigma(X_t)dW_t, \quad 0\leq t\leq T, 
\end{equation}
where $X_t=(V_t,U_t^T)^T$ and the coefficients $\mathbf{b},\bsigma$ are defined by 
\begin{equation}
\label{eq:def-coeffXfashion}
\mathbf{b}(x)=\left(
 \begin{array}{c}
 a(u)v+b(u)\\
 -\textbf{diag}\left( \balpha\bigl( v \bigr) + \bbeta\bigl( v \bigr) \right) u +\balpha \bigl( v \bigr)\\
 \end{array}
 \right)
,\quad
 \bsigma(x)=\left(
 \begin{array}{c}
 \Sigma(u)\\
 0\\
 \end{array}
 \right),
\end{equation}
for all $x=(v,u^T)^T\in\R\times[0,1]^d$. Notice that the coefficients defined in \eqref{eq:def-coeffXfashion} are locally Lipschitz thanks to our assumptions on $\Sigma$, $a$, $b$, and $\alpha$, $\beta$ (cf. Section \ref{sec:model_description}).
Therefore, it is known (see, e.g., \cite{revuz1999} Exercise IX.2.10-1)) that a strong solution~$X$ to \eqref{eq:HH_Xfashion} (equiv. to~\eqref{eq:HH_redefined}\textbf{[BM]}) exists up to the explosion~time
$$
\tau=\inf\{0\leq t\leq T: \,|X_t|=\infty\}.
$$
We thus aim at proving that $\P(\tau=\infty)=1$, and start with the following lemma.

\begin{lemma}[\cite{Hoepfner2012}, Subsection 4.4]\label{lem:Uin01}
For any $\omega\in\Omega$ and any $s<\tau(\omega)$ one has that~$U_t(\omega)\in[0,1]^d$, $0\leq t\leq s$.
\end{lemma}

\begin{proof}
The proof of \cite{Hoepfner2012} is provided in the Appendix for completeness.
\end{proof}

Once we know that the $U$-component of $X$ stays in the hypercube $[0,1]^d$ the situation regarding the $V$-component essentially boils down to the one with bounded diffusion coefficient and affine drift term. One can then expect that the moments of $V$ are bounded.

However, one has to be cautious because the fact that $U$ stays in $[0,1]^d$ is only guaranteed before explosion time $\tau$. In order to establish non explosion we thus introduce the stopping times $S_k=\inf\{ t\geq 0:\,|X_t|>k \}$  (note that $S_k<\tau$ for any $k$) and prove the following~lemma. 

\vspace{0.2cm}

\begin{lemma}\label{lem:bounded_moments_true_BM}
    Let $X_t=(V_t,U_t^T)^T$, $t \in [0,T]$, be the solution of System \eqref{eq:HH_redefined}\textbf{[BM]} up to explosion time $\tau$. Then, for any $p\geq 1$, it holds that 
    \begin{equation}
    \label{eq:control-psik}
        \E\left[\,\max_{0\leq s\leq t\wedge S_k}|V_s|^{2p} \,\right] \leq C\left( 1+ \mathbb{E}\left[\left|V_0\right|^{2p}\right] \right), \quad 0\leq t\leq T,
    \end{equation}
     where the constant $C>0$ depends on $p$, $T$, $\mathbf{b}$ and $\bsigma$. 
\end{lemma}

\begin{proof}
The proof follows standard arguments, such as Grönwall’s inequality, and is postponed to the Appendix.
\end{proof}

We are now in position to prove the following theorem.

\begin{theorem}
\label{thm:bounded_moments_true_BM} It holds that $\P(\tau=\infty)=1$, so that a unique strong solution $X=(V,U^T)^T$ to System  \eqref{eq:HH_redefined}\textbf{[BM]} exists on the time interval $[0,T]$.
Besides, for this strong solution one has
\begin{equation}
\label{eq:Uin01-BM}
U_t\in[0,1]^d,\quad 0\leq t\leq T, \;\;\P\text{-a.s.}
\end{equation}
Moreover, for any $p\geq 1$, it holds that
 \begin{equation}
 \label{eq:V-pmoments-BM}
        \mathbb{E}\left[ \max_{0\leq t\leq T}\left|V_t\right|^{2p} \right] \leq C\left( 1+ \mathbb{E}\left[\left|V_0\right|^{2p}\right] \right), 
    \end{equation}
    where the constant $C>0$ depends on  $p$, $T$, $\mathbf{b}$ and $\bsigma$.
\end{theorem}

\begin{proof}
Let $\varphi^p_k(t)=\E[\max_{0\leq s\leq t\wedge S_k}|X_s|^{2p} ]$. Using Lemmas \ref{lem:Uin01} and \ref{lem:bounded_moments_true_BM} one has
$$
\varphi_k^p(t)\leq C(p)\big\{  \mathbb{E}\big[ \max_{0 \leq s\leq t\wedge S_k}\left|V_t\right|^{2p} \big]   +\E[ \max_{ 0\leq s\leq t\wedge S_k}|U_s|^{2p} ] \big\}\leq 
C(p)\big\{  C\left( 1+ \mathbb{E}\big[\left|V_0\right|^{2p}\big] \right)+C(p,d) \big\},
$$
where the constant $C$ is the one in \eqref{eq:control-psik}. 

By using, for instance, Fatou's lemma, one obtains 
\begin{align*}
    \E\left[ \max_{0\leq t\leq T}|X_t|^{2p} \right]=\E\left[\liminf_k\max_{0\leq t\leq T\wedge S_k}|X_t|^{2p} \right] 
    &\leq \liminf_k\varphi^p_k(T) \\ &\leq C(p)\big\{  C\left( 1+ \mathbb{E}\big[\left|V_0\right|^{2p}\big] \right)+C(p,d) \big\}.
\end{align*}
This implies that at any time $0\leq t\leq T$ one has $|X_t|<\infty$, $\P$-a.s., and therefore $\P(\tau=\infty)=1$.

Thus, a strong solution exists, and its uniqueness follows from the fact that $\mathbf{b}$ and $\bsigma$ are locally Lipschitz (see, e.g., \cite{kara}, Theorem 5.2.5).
As $\tau=\infty$ a.s., one can take~$s=T$ in Lemma \ref{lem:Uin01} and obtain \eqref{eq:Uin01-BM}.
Applying again Fatou's lemma, but this time for inequality~\eqref{eq:control-psik}, yields~\eqref{eq:V-pmoments-BM}.
\end{proof}

\paragraph*{The \textbf{[OU]}-case.}
To get an analogue of Theorem \ref{thm:bounded_moments_true_BM} for System 
\eqref{eq:HH_redefined}$\textbf{[OU]}$ one can proceed similarly as in the $\textbf{[BM]}$-case. First, note that \eqref{eq:HH_redefined}$\textbf{[OU]}$ can be rewritten in the form
\begin{equation}
\label{eq:HHOU_Xfashion}
dX_t=\mathbf{b}_{OU}(X_t)dt+ \bsigma_{OU}(X_t)dW_t, \quad 0\leq t\leq T, 
\end{equation}
where $X_t=(V_t,U_t^T,Z_t)^T$ and the coefficients $\mathbf{b}_{OU},\bsigma_{OU}$ are defined by 
\begin{equation}
\label{eq:def-coeffXfashion-OU}
\mathbf{b}_{OU}(x)=\left(
 \begin{array}{c}
 a(u)v+b(u)+\theta(\mu-z)\\
 -\textbf{diag}\left( \balpha\bigl( v \bigr) + \bbeta\bigl( v \bigr) \right) u +\balpha \bigl( v \bigr)\\
 \theta(\mu-z)
 \end{array}
 \right)
,\;
 \bsigma_{OU}(x)=\left(
 \begin{array}{c}
 \sigma\\
 0\\
 \sigma
 \end{array}
 \right),
\end{equation}
for all $x=(v,u^T,z)^T\in\R\times[0,1]^d\times\R$. The above coefficients are locally Lipschitz and a strong solution to \eqref{eq:HHOU_Xfashion} (equiv. to \eqref{eq:HH_redefined} $\textbf{[OU]}$) exists up to explosion time 
$\tau=\inf\{0\leq t\leq T: \,|X_t|=\infty\}$. Proceeding as for Lemma
\ref{lem:Uin01} leads to the following result.
\begin{lemma}
\label{lem:Uin01-OU}
Let $X_t=(V_t,U_t^T,Z_t)^T$, $t \in [0,T]$, be the solution to~\eqref{eq:HH_redefined}\textbf{[OU]} up to explosion time~$\tau$. Then, one has
\begin{equation}
\label{eq:Uin01OU}
    \forall \omega\in\Omega,\;\;\forall s<\tau(\omega),\quad
    U_t(\omega)\in[0,1]^d,\quad 0\leq t\leq s.
\end{equation}
\end{lemma}

We then set  again  $S_k=\inf\{ t\geq 0:\,|X_t|>k \}$   and  get an analogue of Lemma \ref{lem:bounded_moments_true_BM} for the \textbf{[OU]}-case:
\begin{lemma}\label{lem:bounded_moments_OU}
    Let $X_t=(V_t,U_t^T,Z_t)^T$, $t \in [0,T]$, be the solution of System \eqref{eq:HH_redefined}\textbf{[OU]} up to explosion time $\tau$. Then, for any $p\geq 1$, it holds that  
    \begin{equation}
    \label{eq:control-psik-OU}
        \E\left[ \max_{0\leq s\leq t\wedge S_k}|V_s|^{2p}
     + \max_{0\leq t\wedge S_k}|Z_s|^{2p}\right]
         \leq C(1+\E\left[|V_0|^{2p}\right] + \E\left[|Z_0|^{2p}\right]), \quad 0\leq t\leq T,
    \end{equation}
     where the constant $C>0$  depends on $p$, $T$, $\mathbf{b}_{OU}$ and $\bsigma_{OU}$. 
\end{lemma}

\begin{proof}
    In the Appendix, we explain how the arguments in the proof of Lemma~\ref{lem:bounded_moments_true_BM} can be adapted to the \textbf{[OU]}-case.
\end{proof}

One then obtains the following result.
\begin{theorem}
         \label{thm:existence-and-bounds-OU}
         It holds that $\P(\tau=\infty)=1$, so that a unique strong solution $X=(V,U^T,Z)^T$  to~\eqref{eq:HH_redefined}\textbf{[OU]} exists on the time interval $[0,T]$. Besides, for this strong solution one has
\begin{equation}
\label{eq:Uin01-OU}
U_t\in[0,1]^d,\quad 0\leq t\leq T, \;\;\P\text{-a.s.}
\end{equation}
Moreover, for any $p\geq 1$, it holds that
 \begin{equation}
 \label{eq:V-pmoments-OU}
        \mathbb{E}\left[ \max_{0\leq t\leq T}\left|V_t\right|^{2p} \right] \leq C\left( 1+ \mathbb{E}\left[\left|V_0\right|^{2p}\right] \right), 
    \end{equation}
    where the constant $C>0$ depends on $p$, $T$, $\mathbf{b}_{OU}$ and $\bsigma_{OU}$.
     \end{theorem}

     \begin{proof}
         We can proceed as in the proof of Theorem~\ref{thm:bounded_moments_true_BM} to obtain that
     $$
     \E\left[ \max_{0\leq t\leq T}|X_t|^{2p} \right]\leq C(p)\big\{  C\left( 1+ \mathbb{E}\left[\left|V_0\right|^{2p}\right]+\mathbb{E}\left[\left|Z_0\right|^{2p}\right]\right)+C(p,d) \big\},$$
     where the constant $C$ is as in \eqref{eq:cont-Psi-OU} (cf. Appendix \ref{appA}). 
     Then minor last steps, similar to those in the proof of Theorem~\ref{thm:bounded_moments_true_BM}, lead to the result.
     \end{proof}
     
     \begin{remark}
     Note that the above proof leads to the estimate
    \begin{equation}
    \label{eq:Z-pmoments-crude}
        \mathbb{E}\left[ \max_{0\leq t\leq T}\left|Z_t\right|^{2p} \right] \leq C\left( 1+ \mathbb{E}\left[\left|Z_0\right|^{2p}\right] \right), 
    \end{equation}
    where the constant $C$ depends on $T$, which is not surprising as the coefficients of the SDE for $Z$ satisfy a linear growth condition (see, e.g., Problem 5.3.15 in \cite{kara}). However, the control of moments of an O.U. process can be refined, achieving a bound $C$ which is independent of $T$. This will be performed in Lemma \ref{lemma:boundedMomentsZ} (cf. Appendix \ref{appA}) and is particularly relevant when proving the properties of the splitting schemes for System \eqref{eq:HH_redefined}\textbf{[OU]} in Section \ref{sec:Splitting_properties}. 
     \end{remark}

\subsection{Motivating example: Hodgkin-Huxley neuronal model}
\label{sec:HHmodel}

The work presented in this paper is motivated by the Hodgkin-Huxley neuronal model (\cite{Hodgkin1952}), which is arguably the most famous and widely used mathematical model to describe how action potentials are generated and propagated in a neuron. In its original form, it is given by the following 4-dimensional ODE
\begin{equation}\label{eq:original_HH}
\begin{aligned}
    dV_t &= \frac{1}{C} \left( I - {g}_K n_t^4 (V_t - E_K) 
        - {g}_{Na} m_t^3 h_t (V_t - E_{Na}) 
        - {g}_L (V_t - E_L) \right) dt, \\
    dn_t &= \alpha_n(V_t)(1-n_t) - \beta_n(V_t) n_t \ dt, \\
    dm_t &= \alpha_m(V_t)(1-m_t) - \beta_m(V_t) m_t \ dt, \\
    dh_t &= \alpha_h(V_t)(1-h_t) - \beta_h(V_t) h_t \ dt,
\end{aligned}
\end{equation}
where the variable $V_t$ describes the membrane potential dynamics of a neuron with capacitance $C$ (in particular, its response to an input current $I$), and the variables $n_t$, $m_t$, and $h_t$ are auxiliary components between $0$ and $1$, representing the potassium channel subunit activation, sodium channel subunit activation, and sodium channel subunit inactivation, respectively. The constants ${g}_K, {g}_{Na}, {g}_L \in (0,\infty)$ represent the conductances, and the constants $E_K, E_{Na}, E_L \in \mathbb{R}$ the reversal potentials, for the potassium (K), sodium (Na) and leak (L) channels, respectively. Moreover, in the original paper by \cite{Hodgkin1952}, the following rate functions, which take values in $(0,\infty)$, are considered
\[
\begin{array}{ll}
 \alpha_n(V_t) = \frac{0.01(10+V_{rest} - V_t)}{\exp\big(\frac{10+V_{rest} - V_t}{10}\big)-1}, & \beta_n(V_t) = 0.125\exp\bigg(\frac{V_{rest} - V_t}{80}\bigg),  \\ 
 \alpha_m(V_t) = \frac{0.1(25+V_{rest}-V_t)}{\exp\big(\frac{25+V_{rest} - V_t}{10}\big)-1}, &  \beta_m(V_t) = 4\exp\bigg(\frac{V_{rest} - V_t}{18}\bigg), 
\\ \alpha_h(V_t) = 0.07\exp\bigg(\frac{V_{rest} - V_t}{20}\bigg), & \beta_h(V_t) = \frac{1}{\exp\big(\frac{30+V_{rest} - V_t}{10}\big) + 1},
\end{array} 
\]
where $(V_{rest} - V_t)$ denotes the negative depolarization in the membrane potential $V_t$. 

Introducing the 3-dimensional variable $U_t:=(n_t,m_t,h_t)^T \in [0,1]^3$, defining the functions
\begin{equation}\label{eq:a_b_functions}
\begin{aligned}
    a\bigl(U_t\bigr) & := \frac{1}{C}\left(-{g}_Kn^4_t-{g}_{Na}m^3_th_t-{g}_L \right),\\ 
    b\bigl(U_t\bigr) & := \frac{1}{C}\left(I+{g}_K E_K n^4_t +{g}_{Na}E_{Na}m^3_th_t+{g}_L E_L\right),
\end{aligned}
\end{equation}
and adding a suitable stochastic term (cf. Section \ref{sec:model_description}), the original Hodgkin-Huxley neuronal model \eqref{eq:original_HH} fits into the class of conditionally linear systems of type \eqref{eq:HH_redefined} investigated in this manuscript.

\section{Splitting methods}\label{section:splitting_schemes}

The numerical splitting approach consists of the following three steps:
\begin{itemize}
    \item[i.] Split the (unsolvable) system into exactly solvable subsystems.
    \item[ii.] Derive the exact solution for each of these subsystems.
    \item[iii.] Compose the derived exact solutions in a suitable way.
\end{itemize}

Let $t_0=0, \ldots, t_n=T$ be a discretization of the time interval $[0,T]$, where $n\in\mathbb{N}\backslash \{0\}$, $t_i=i\Delta$, for $i=0,\ldots,n$, with $\Delta=T/n$. Throughout, we assume that $\Delta \in (0,1)$.  

In Section \ref{subsec:Splitting_BM_HH} and Section \ref{subsec:Splitting_OU_HH}, we  construct  numerical solutions $(\widetilde{X}_{t_i})_{i=0,\ldots,n}$ based on the splitting approach, for System~\eqref{eq:HH_redefined}\textbf{[BM]} and \eqref{eq:HH_redefined}\textbf{[OU]}, respectively, which approximate the true (unknown) solution $(X_t)_{t\in [0,T]}$ at the discrete time points $t_i$, where $\widetilde{X}_{t_0=0}=X_0$. Moreover, in Section~\ref{subsec:Splitting_contVersions}, we introduce continuous-time extensions $(\overline{X}_t)_{t\in[0,T]}$ of the constructed splitting methods, which coincide with the discrete-time versions at the grid points $t=t_i$, i.e. $\overline{X}_{t_i}=\widetilde{X}_{t_i}$, for $i=0,\ldots,n$. These  continuous-time extensions are introduced to facilitate the convergence analysis of the splitting schemes (see Section \ref{sec:Splitting_properties}) via the framework introduced in Section \ref{sec:fund-thm}.

\begin{remark}
    It is worth noting that the It\^o and Stratonovich formulations of both System \eqref{eq:HH_redefined}\textbf{[BM]} and System \eqref{eq:HH_redefined}\textbf{[OU]} coincide. For System \eqref{eq:HH_redefined}\textbf{[BM]}, this is because  the diffusion coefficient $\Sigma$ only depends on the $U$-component, while acting solely on the $V$-component. For System \eqref{eq:HH_redefined}\textbf{[OU]}, it is due to additive noise. This is beneficial because the literature suggests potential advantages of constructing splitting schemes from the Stratonovich form (see, e.g., \cite{Misawa2001, Alamo2016,Foster2024}).
\end{remark}

\subsection{Splitting methods for the Brownian motion-driven model}\label{subsec:Splitting_BM_HH}

In this subsection, we consider the Brownian motion-driven SDE, i.e. System  \eqref{eq:HH_redefined}\textbf{[BM]}.  

\paragraph*{Step i (definition of subsystems).}

Exploiting the conditional linearity of the SDE, we propose to split System  \eqref{eq:HH_redefined}\textbf{[BM]} into the following two exactly solvable subsystems
\begin{align}
    \label{eq:sub1_BM}d \begin{pmatrix}
        V_t^{[1]} \\
        U_t^{[1]} 
    \end{pmatrix} & = \begin{pmatrix}
        a(U_t^{[1]})V_t^{[1]}+b(U_t^{[1]}) \\
        0_d
    \end{pmatrix}dt + 
    \begin{pmatrix}
        \Sigma(U_t^{[1]}) \\
        0_d
    \end{pmatrix}dW_t, \\
    \label{eq:sub2_BM}d \begin{pmatrix}
        V_t^{[2]} \\
        U_t^{[2]} 
    \end{pmatrix} & = \begin{pmatrix}
        0 \\
        -\textbf{diag}\left( \balpha\bigl( V_t^{[2]} \bigr) + \bbeta\bigl( V_t^{[2]} \bigr) \right) U_t^{[2]} +\balpha \bigl( V_t^{[2]} \bigr)
    \end{pmatrix}dt.
\end{align}
Note that, in the first subsystem the $U^{[1]}$-component is held constant so that the $V^{[1]}$-component corresponds to the solution of a linear SDE, while in the second subsystem the $V^{[2]}$-component is held constant so that the $U^{[2]}$-component solves a linear ODE.

\paragraph*{Step ii (solutions of subsystems).}

The integral equation, which defines the (stochastic) flow of Subsystem \eqref{eq:sub1_BM} with initial condition $X_0^{[1;\textrm{BM}]}=x=(v,u^T)^T$ is given by
\begin{equation}\label{eq:sub1_BM_flow_intEq}
    \phi_{t}^{[1;\textrm{BM}]}(x) = 
    \begin{pmatrix}
        \phi_{t}^{[1;\textrm{BM}],1}(x) \\ 
        \phi_{t}^{[1;\textrm{BM}],2}(x)
    \end{pmatrix} = 
    \begin{pmatrix}
         v+ \int\limits_0^t a(u) \phi_{s}^{[1;\textrm{BM}],1}(x) + b(u) ds + \int\limits_0^t \Sigma(u) dW_s \\
         u
    \end{pmatrix}.
\end{equation}
The flow $\phi_{t}^{[1;\textrm{BM}]}$ admits the explicit expression 
\begin{equation}\label{eq:sub1_BM_flow}
    \phi_{t}^{[1;\textrm{BM}]}(x) = \left(\begin{matrix} \psi_{\textrm{BM}}(t;v,u)  \\ u\end{matrix} \right), 
\end{equation}
where
\begin{align*}
    \psi_{\textrm{BM}}(t;v,u) &:= e^{a(u)t}v + b(u) \int\limits_{0}^{t} e^{a(u)(t-s)} ds + \Sigma(u) \int\limits_{0}^{t} e^{a(u)(t-s)} dW_s \\
    &= e^{a(u)t}v + \frac{b(u)}{a(u)}\left( e^{a(u)t} -1 \right) + \Sigma(u) \int\limits_{0}^{t} e^{a(u)(t-s)} dW_s.
\end{align*}

The integral equation, which defines the (deterministic) flow of Subsystem \eqref{eq:sub2_BM} with initial condition $X_0^{[2;\textrm{D}]}=x=(v,u^T)^T$ is given by
\begin{equation}\label{eq:sub2_BM_flow_intEq}
    \phi_t^{[2;\textrm{D}]}(x) = 
    \begin{pmatrix}
        \phi_t^{[2;\textrm{D}],1}(x) \\
        \phi_t^{[2;\textrm{D}],2}(x) 
    \end{pmatrix}
    =
    \begin{pmatrix}
        v \\
        u + \int\limits_0^t \left(-\textbf{diag}\left( \balpha\bigl( v \bigr) + \bbeta\bigl(v \bigr) \right) \phi_s^{[2;\textrm{D}],2}(x) \right) + \balpha\bigl( v \bigr) ds
    \end{pmatrix}.
\end{equation}
The flow $\phi_{t}^{[2;\textrm{D}]}$ admits the explicit expression
\begin{equation}\label{eq:sub2_BM_flow}
    \phi_t^{[2;\textrm{D}]}(x) = \left(\begin{matrix} v  \\  \psi_{\textrm{D}}(t;v,u) \end{matrix} \right), 
\end{equation}
where
\begin{align*}
    \psi_{\textrm{D}}(t;v,u)&:=e^{-\textbf{diag} \bigl( \balpha(v)+\bbeta(v) \bigr) t } u  + \int\limits_{0}^{t} e^{- \textbf{diag} \bigl( \balpha(v)+\bbeta(v) \bigr) (t-s)  } \balpha(v) ds \\
    &=e^{-\textbf{diag} \bigl( \balpha(v)+\bbeta(v) \bigr) t } u + \left( \mathbb{I}_d - e^{-\textbf{diag} \bigl( \balpha(v)+\bbeta(v) \bigr) t } \right) U^\infty(v),
\end{align*}
and
\begin{equation*}
    U^\infty(v) :=\left( U^{1,\infty}(v),\ldots, U^{d,\infty}(v) \right)^T, \quad \text{with} \quad U^{l,\infty}(v)=\frac{\alpha_l(v)}{\alpha_l(v)+\beta_l(v)}, \quad \text{for } l\in\{1,\ldots,d\}.
\end{equation*}

\paragraph*{Step iii (composition of solutions).}

To obtain a numerical approximation of the solution of SDE \eqref{eq:HH_redefined}\textbf{[BM]}, the flows of the two subequations have to be properly composed in every iteration step. We consider two Lie-Trotter compositions and one Strang composition, which are given recursively by 
\begin{equation}\label{eq:splittings_BM}
    \begin{aligned}
        \widetilde{X}^{LT1}_{t_{i}} &= \left(\phi_{\Delta}^{[1;\textrm{BM}]}\circ \phi_\Delta^{[2;\textrm{D}]} \right)(\widetilde{X}^{LT1}_{t_{i-1}}), \\
        \widetilde{X}^{LT2}_{t_{i}} &= \left(\phi_\Delta^{[2;\textrm{D}]}\circ \phi_{\Delta}^{[ 1;\textrm{BM}]} \right)(\widetilde{X}^{LT2}_{t_{i-1}}), \\ 
        \widetilde{X}^{S}_{t_{i}} &= \left(\phi_{\Delta/2}^{[2;\textrm{D}]}\circ \phi_{\Delta}^{[1;\textrm{BM}]}\circ \phi_{\Delta/2}^{[2;\textrm{D}]} \right)(\widetilde{X}^{S}_{t_{i-1}}),
    \end{aligned}
\end{equation}
respectively, for $i=1,\ldots,n$, with initial condition $\widetilde{X}^{LT1}_{t_{0}}=\widetilde{X}^{LT2}_{t_{0}}=\widetilde{X}^{S}_{t_{0}}=X_0$.

Specifically, using the explicit formulas \eqref{eq:sub1_BM_flow} and \eqref{eq:sub2_BM_flow} for the flows of the subequations, the first Lie-Trotter scheme is given by the iteration 
\begin{equation}\label{eq:BM_HH_LT1}
    \begin{aligned}
        {U}^{LT1}_i &:=\psi_\textrm{D}\bigl(\Delta;\widetilde{V}^{LT1}_{t_{i-1}},\widetilde{U}^{LT1}_{t_{i-1}}\bigr), \\
        \widetilde{V}^{LT1}_{t_{i}}& = \psi_{\textrm{BM}}\bigl(\Delta;\widetilde{V}^{LT1}_{t_{i-1}},{U}^{LT1}_i\bigr), \\
        \widetilde{U}^{LT1}_{t_{i}}& = {U}^{LT1}_i,
    \end{aligned}
\end{equation}
the second Lie-Trotter scheme by the iteration
\begin{equation}\label{eq:BM_HH_LT2}
    \begin{aligned}
         {U}_i^{LT2}&:=\widetilde{U}^{LT2}_{t_{i-1}}, \\
        \widetilde{V}^{LT2}_{t_{i}}& = \psi_{\textrm{BM}}\bigl(\Delta;\widetilde{V}^{LT2}_{t_{i-1}},{U}_i^{LT2}\bigr), \\
        \widetilde{U}^{LT2}_{t_{i}}& = \psi_{\textrm{D}} (\Delta, \widetilde{V}^{LT2}_{t_{i}}, {U}_i^{LT2}),
    \end{aligned}
\end{equation}
and the Strang scheme by the iteration
\begin{equation}\label{eq:BM_HH_S}
    \begin{aligned}
        {U}^{S}_i&:=\psi_\textrm{D}\bigl(\Delta/2;\widetilde{V}^{S}_{t_{i-1}},\widetilde{U}^{S}_{t_{i-1}}\bigr), \\
        \widetilde{V}^{S}_{t_{i}}& = \psi_{\textrm{BM}}\bigl(\Delta;\widetilde{V}^{S}_{t_{i-1}},{U}_i^{S}\bigr), \\
        \widetilde{U}^{S}_{t_{i}}& = \psi_{\textrm{D}} (\Delta/2; \widetilde{V}^{S}_{t_{i}}, {U}_i^{S}).
    \end{aligned}
\end{equation}

Note in particular that the $V$-components of the schemes are all of the type
\begin{align}\label{eq:explicitV_BM}
    \widetilde{V}_{t_i} &= \psi_{\textrm{BM}}\bigl(\Delta;\widetilde{V}_{t_{i-1}},{U}_i\bigr) = e^{a\left({U}_i\right)\Delta} \widetilde{V}_{t_{i-1}} + \frac{b\bigl({U}_i\bigr) }{ a\bigl({U}_i\bigr) } \left(  e^{a\left({U}_i\right)\Delta} -1 \right) + \chi_{i-1}^{\textrm{BM}}(U_i), \\
    \chi_{i-1}^{\textrm{BM}} \bigl({U}_i\bigr) &:= \Sigma\bigl({U}_i\bigr) \int\limits_{t_{i-1}}^{t_i} e^{a\bigl({U}_i\bigr)(t_i-s)} dW_s,
\end{align}
where the It\^o-integrals $\chi_{i-1}^{\textrm{BM}} \bigl({U}_i\bigr)$, for $i=1,\ldots,n$, are independent. This is because, for each scheme, the corresponding ${U}_i$ as in \eqref{eq:BM_HH_LT1}-\eqref{eq:BM_HH_S} are $\mathcal{F}_{t_{i-1}}$-measurable for $i=1,\ldots,n$.

\begin{remark}
Note that when $\Sigma(U)\equiv 0$, which includes the classical deterministic Hodgkin-Huxley model, the derivations remain valid and the constructed splitting schemes still apply; cf. \cite{Chen2020}. 
\end{remark}

\paragraph*{Implementation of the splitting schemes.}

The schemes are constructed from the iterative expressions \eqref{eq:BM_HH_LT1}-\eqref{eq:BM_HH_S}.
To implement them, we use the fact that the It\^o-integrals $\chi^{\textrm{BM}}_{i-1} ({u}_i)$ defined in \eqref{eq:explicitV_BM}, for given realizations ${u}_i$ of ${U}_i$, $i=1,\ldots,n$, are independent and normally distributed with mean zero and variance 
\begin{equation}\label{eq:nu_variance_BM}
    \nu^2_{\textrm{BM}}(\Delta;{u}_i):=\frac{\Sigma^2({u}_i)}{2a({u}_i)}\left( e^{2a({u}_i) \Delta} - 1\right),
\end{equation}
as follows from It\^o's isometry.

Let $(W_{t_i})_{i=0,\ldots,n}$ be a realization of the underlying Brownian motion. Then, the increments
\begin{equation}\label{eq:xi_increments}
    \xi_{i-1}:=\frac{1}{\sqrt{\Delta}}\left(W_{t_i}-W_{t_{i-1}} \right), \quad i=1,\ldots,n,
\end{equation}
are independent realizations of the standard normal distribution. Consequently, realizations of the integrals $\chi^{\textrm{BM}}_{i-1}({u}_i)$ can be generated as
\begin{equation}\label{eq:chi_realizations}
    \nu_{\textrm{BM}}(\Delta;{u}_i)\xi_{i-1}, \quad i=1,\ldots,n,
\end{equation}
yielding a pathwise (strong) approximation of the solution corresponding to the given Brownian trajectory.
\subsection{Splitting methods for the Ornstein-Uhlenbeck-driven model}\label{subsec:Splitting_OU_HH}

In this subsection, we consider the Ornstein-Uhlenbeck-driven SDE, i.e. System \eqref{eq:HH_redefined}\textbf{[OU]}. Note that Equation~\eqref{eq:OU} adds another dimension to that system, where the Brownian motion is the same in the first and the third equation.

\paragraph*{Step i (definition of subsystems).}

Similar to before, we propose to split this system into the following two exactly solvable subsystems
\begin{align}
    \label{eq:sub1_OU}d \begin{pmatrix}
        V_t^{[1]} \\
        U_t^{[1]} \\
        Z_t^{[1]}
    \end{pmatrix} & = \begin{pmatrix}
        a(U_t^{[1]})V_t^{[1]}+b(U_t^{[1]}) + \theta(\mu-Z_t^{[1]}) \\
        0_d \\
        0
    \end{pmatrix}dt + 
    \begin{pmatrix}
        \sigma \\
        0_d \\
        0
    \end{pmatrix}dW_t, \\
    \label{eq:sub2_OU}d \begin{pmatrix}
        V_t^{[2]} \\
        U_t^{[2]} \\
        Z_t^{[2]}
    \end{pmatrix} & = \begin{pmatrix}
        0 \\
        -\textbf{diag}\left( \balpha\bigl( V_t^{[2]} \bigr) + \bbeta\bigl( V_t^{[2]} \bigr) \right) U_t^{[2]} +\balpha \bigl( V_t^{[2]} \bigr) \\
        \theta(\mu-Z_t^{[2]})
    \end{pmatrix}dt + 
    \begin{pmatrix}
        0 \\
        0_d \\
        \sigma
    \end{pmatrix}dW_t.
\end{align}
Since $U^{[1]}$ and $Z^{[1]}$ are held constant in the first subsystem, the $V^{[1]}$-component solves again a linear SDE. Moreover, the dynamics for $Z^{[2]}$ in the second subsystem are independent of the other components and, since $V^{[2]}$ is held constant and $Z^{[2]}$ does not appear in the second equation, the $U^{[2]}$-component corresponds again to the solution of a linear ODE.  

\paragraph*{Step ii (solutions of subsystems).}

The integral equation, which defines the (stochastic) flow of Subsystem \eqref{eq:sub1_OU} with initial condition $X_0^{[1;\textrm{OU}]}=x=(v,u^T,z)^T$ is given by
\begin{equation}\label{eq:sub1_OU_flow_intEq}
    \phi_t^{[1;\textrm{OU}]}(x) = 
    \begin{pmatrix}
        \phi_t^{[1;\textrm{OU}],1}(x) \\
        \phi_t^{[1;\textrm{OU}],2}(x) \\
        \phi_t^{[1;\textrm{OU}],3}(x)
    \end{pmatrix}
    =
    \begin{pmatrix}
        v + \int\limits_0^t a(u)  \phi_s^{[1;\textrm{OU}],1}(x) + b(u) +\theta(\mu-z) ds +\int\limits_0^t \sigma dW_s \\
        u \\
        z
    \end{pmatrix}.
\end{equation}
The flow $\phi_{t}^{[1;\textrm{OU}]}$ admits the explicit expression
\begin{align}\label{eq:sub1_OU_flow}
        \phi_{t}^{[1;\textrm{OU}]}(x)&=
        \begin{pmatrix}
        \psi_\textrm{OU}(t;v,u,z) \\
        u \\ 
        z
        \end{pmatrix}, 
\end{align}
where
\begin{align*}
    \psi_\textrm{OU}(t;v,u,z)&=e^{a(u)t}v + \left(b(u)+\theta(\mu-z)\right) \int\limits_{0}^{t} e^{a(u)(t-s)}ds +  \sigma \int\limits_{0}^{t} e^{a(u)(t-s)}dW_s \\
    &= e^{a(u)t}v+\frac{b(u)+\theta(\mu-z)}{a(u)}\left( e^{a(u)t} - 1 \right) + \sigma \int\limits_{0}^{t} e^{a(u)(t-s)}dW_s.
\end{align*}

The integral equation, which defines the flow (stochastic in the third component) of Subsystem~\eqref{eq:sub2_OU} with initial condition $X_0^{[2;\textrm{Z}]}=x=(v,u^T,z)^T$ is given by
\begin{equation}\label{eq:sub2_OU_flow_intEq}
    \phi_t^{[2;\textrm{Z}]}(x) = 
    \begin{pmatrix}
        \phi_t^{[2;\textrm{Z}],1}(x) \\
        \phi_t^{[2;\textrm{Z}],2}(x) \\
        \phi_t^{[2;\textrm{Z}],3}(x)
    \end{pmatrix}
    =
    \begin{pmatrix}
        v \\
        u + \int\limits_0^t \left( -\textbf{diag}\left( \balpha\bigl( v \bigr) + \bbeta\bigl(v \bigr) \right) \phi_s^{[2;\textrm{Z}],2}(v,u^T) \right) + \balpha\bigl( v \bigr) ds \\
        z + \int\limits_0^t \theta \left( \mu - \phi_s^{[2;\textrm{Z}],3} (z) \right) ds + \int\limits_0^t \sigma dW_s.
    \end{pmatrix}.
\end{equation}
The flow $\phi_{t}^{[2;\textrm{Z}]}$ admits the explicit expression
\begin{align}\label{eq:sub2_OU_flow}
        \phi_{t}^{[2;\textrm{Z}]}(x)&=
        \begin{pmatrix}
        v \\
        \psi_\textrm{D}(t;v,u) \\ 
        \psi_\textrm{Z}(t;z) \end{pmatrix},
\end{align}
where $\psi_\textrm{D}(t;v,u)$ is as in \eqref{eq:sub2_BM_flow}, and
\begin{align*}
    \psi_\textrm{Z}(t;z)&= e^{-\theta t} z +\theta\mu \int\limits_{0}^{t} e^{-\theta(t-s)} ds + \sigma \int\limits_{0}^{t} e^{-\theta(t-s)}dW_s \\
    &= e^{-\theta t} z + \mu(1-e^{-\theta t}) +  \sigma \int\limits_{0}^{t} e^{-\theta(t-s)}dW_s.
\end{align*}

\paragraph*{Step iii (composition of solutions).}

Similar to before, the flows of the two subequations have to be composed in every iteration step. We consider again two Lie-Trotter compositions and one Strang composition, which are recursively given by
\begin{equation}\label{eq:splittings_OU}
    \begin{aligned}
        \widetilde{X}^{LT1}_{t_{i}} &= \left(\phi_{\Delta}^{[1;\textrm{OU}]}\circ \phi_{\Delta}^{[2;\textrm{Z}]} \right)(\widetilde{X}^{LT1}_{t_{i-1}}), \\
        \widetilde{X}^{LT2}_{t_{i}} &= \left(\phi_{\Delta}^{[2;\textrm{Z}]}\circ \phi_{\Delta}^{[1;\textrm{OU}]} \right)(\widetilde{X}^{LT2}_{t_{i-1}}), \\ 
        \widetilde{X}^{S}_{t_{i}} &= \left(\phi_{\Delta/2}^{[2;\textrm{Z}]}\circ \phi_{\Delta}^{[1;\textrm{OU}]}\circ \phi_{\Delta/2}^{[2;\textrm{Z}]} \right)(\widetilde{X}^{S}_{t_{i-1}}),
    \end{aligned}
\end{equation}
respectively, for $i=1,\ldots,n$, with initial condition $\widetilde{X}^{LT1}_{t_0}=\widetilde{X}^{LT2}_{t_0}=\widetilde{X}^{S}_{t_0}=X_0$.

Specifically, using the explicit formulas \eqref{eq:sub1_OU_flow} and \eqref{eq:sub2_OU_flow} for the flows of the subequations, the first Lie-Trotter scheme is given by the iteration 
\begin{equation}\label{eq:OU_HH_LT1}
    \begin{aligned}
        {U}^{LT1}_i &:=\psi_\textrm{D}\bigl(\Delta;\widetilde{V}^{LT1}_{t_{i-1}},\widetilde{U}^{LT1}_{t_{i-1}}\bigr), \quad 
        Z_i^{LT1} := \psi_\textrm{Z}(\Delta;Z_{t_{i-1}}), \\
        \widetilde{V}^{LT1}_{t_{i}}& = \psi_{\textrm{OU}}\bigl(\Delta;\widetilde{V}^{LT1}_{t_{i-1}},{U}^{LT1}_i,Z_i^{LT1}\bigr), \\
        \widetilde{U}^{LT1}_{t_{i}} &= {U}^{LT1}_i, \quad
        Z_{t_i}=Z_i^{LT1},
    \end{aligned}
\end{equation}
the second Lie-Trotter scheme by the iteration
\begin{equation}\label{eq:OU_HH_LT2}
    \begin{aligned}
         {U}_i^{LT2}&:=\widetilde{U}^{LT2}_{t_{i-1}}, \quad 
         Z_i^{LT2} := Z_{t_{i-1}}, \\
        \widetilde{V}^{LT2}_{t_{i}}& = \psi_{\textrm{OU}}\bigl(\Delta;\widetilde{V}^{LT2}_{t_{i-1}},{U}_i^{LT2},Z_i^{LT2}\bigr), \\
        \widetilde{U}^{LT2}_{t_{i}}& = \psi_{\textrm{D}} (\Delta, \widetilde{V}^{LT2}_{t_{i}}, {U}_i^{LT2}), \quad
        Z_{t_i}=\psi_\textrm{Z}(\Delta;Z_i^{LT2}),
    \end{aligned}
\end{equation}
and the Strang scheme by the iteration
\begin{equation}\label{eq:OU_HH_S}
    \begin{aligned}
        {U}^{S}_i&:=\psi_\textrm{D}\bigl(\Delta/2;\widetilde{V}^{S}_{t_{i-1}},\widetilde{U}^{S}_{t_{i-1}}\bigr), \quad
        Z_i^{S} := \psi_\textrm{Z}(\Delta/2;Z_{t_{i-1}}), \\
        \widetilde{V}^{S}_{t_{i}}& = \psi_{\textrm{OU}}\bigl(\Delta;\widetilde{V}^{S}_{t_{i-1}},{U}_i^{S},Z_i^S\bigr), \\
        \widetilde{U}^{S}_{t_{i}}& = \psi_{\textrm{D}} (\Delta/2; \widetilde{V}^{S}_{t_{i}}, {U}_i^S), \quad
        Z_{t_i}=\psi_\textrm{Z}(\Delta/2;Z_i^{S}).
    \end{aligned}
\end{equation}
Note in particular that the $V$-components of the schemes are all of the type
\begin{align}\label{eq:explicitV_OU}
    \widetilde{V}_{t_i} &= \psi_{\textrm{OU}}\bigl(\Delta;\widetilde{V}_{t_{i-1}},{U}_i,Z_i\bigr) = e^{a\left({U}_i\right)\Delta} \widetilde{V}_{t_{i-1}} + \frac{b\bigl({U}_i\bigr) +\theta(\mu-Z_i)}{ a\bigl({U}_i\bigr) } \left(  e^{a\left({U}_i\right)\Delta} -1 \right) 
    + \chi_{i-1}^{\textrm{OU}}(U_i), \\
    & \chi_{i-1}^{\textrm{OU}} \bigl({U}_i\bigr) := \sigma \int\limits_{t_{i-1}}^{t_i} e^{a\bigl({U}_i\bigr)(t_i-s)} dW_s.
\end{align}
Moreover, the $Z$-components of the Lie-Trotter schemes are both given by
\begin{equation}\label{eq:explicitZ_OU_LT}
    Z_{t_i} = \psi_\textrm{Z}(\Delta;Z_{t_{i-1}}) = e^{-\theta \Delta}Z_{t_{i-1}} + \mu(1-e^{-\theta \Delta}) + \chi_{i-1}^{\textrm{Z}}, \quad \chi_{i-1}^{\textrm{Z}}:= \sigma \int\limits_{t_{i-1}}^{t_i} e^{-\theta(t_i-s)} dW_s,
\end{equation}
while the Strang scheme requires to sample from the Z-component twice:
\begin{align}\label{eq:explicitZ_OU_S}
    Z_i^S &= \psi_\textrm{Z}(\Delta/2;Z_{t_{i-1}}) = e^{-\theta/2 \Delta}Z_{t_{i-1}} + \mu(1-e^{-\theta \Delta/2}) + \chi_{i-1}^{\textrm{Z},1}, \\
    Z_{t_i} &= \psi_\textrm{Z}(\Delta/2;Z_i^S) = e^{-\theta \Delta/2}Z_{i}^S + \mu(1-e^{-\theta \Delta/2}) + \chi_{i-1}^{\textrm{Z},2}, \\
    \chi_{i-1}^{\textrm{Z},1} &:= \sigma \int\limits_{t_{i-1}}^{t_{i-1}+\Delta/2} e^{-\theta(t_{i-1}+\Delta/2-s)} dW_s, \quad \chi_{i-1}^{\textrm{Z},2} := \sigma \int\limits_{t_{i-1}+\Delta/2}^{t_{i}} e^{-\theta(t_{i}-s)} dW_s.
\end{align}

\begin{remark}
    The proposed splitting technique for System \eqref{eq:HH_redefined}\textbf{[OU]} is not the only possible choice, and the number of subsystems is not restricted to two. In general, however, it is desirable to work with as few subsystems as possible, provided that they are chosen in a suitable manner, in order to keep the resulting schemes tractable and reliable. 
\end{remark}

\paragraph*{Implementation of the splitting schemes.}

The schemes are constructed via the iterative expressions \eqref{eq:OU_HH_LT1}-\eqref{eq:OU_HH_S}. To implement the Lie-Trotter schemes, we use that the It\^o-integrals $\chi_{i-1}^{\textrm{OU}}(u_i)$ defined in \eqref{eq:explicitV_OU}, for given realizations ${u}_i$ of $U_i$, $i=1,\ldots,n$, are independent and normally distributed with mean zero and variance 
\begin{equation}\label{eq:nu_variance_OU}
    \nu^2_{\textrm{OU}}(\Delta;{u}_i):=\frac{\sigma^2}{2 a({u}_i)}\left(e^{2a({u}_i) \Delta} -1 \right).
\end{equation}
Similarly, the $\chi_{i-1}^{\textrm{Z}}$ defined in \eqref{eq:explicitZ_OU_LT} for $i=1,\ldots,n$, are independent and normally distributed with mean zero and variance 
\begin{equation*}
    \nu_\textrm{Z}^2(\Delta):=\frac{\sigma^2}{2\theta}\left(1 - e^{-2\theta \Delta} \right).
\end{equation*}

Let $(W_{t_i})_{i=0,\ldots,n}$ be a realization of the underlying Brownian motion and consider independent samples of the standard normal distribution $\xi_{i-1}$, $i=1,\ldots,n$, as in \eqref{eq:xi_increments}. Then, sampling from the $\chi_{i-1}^{\textrm{OU}}(u_i)$ and $\chi_{i-1}^{\textrm{Z}}$ via
\begin{equation*}
    \nu_{\textrm{OU}}(\Delta;{u}_i)\xi_{i-1}, \quad \nu_\textrm{Z}(\Delta)\xi_{i-1}, \quad i=1,\ldots,n,
\end{equation*}
yields an approximation of the solution path, which corresponds to the given Brownian path.
Note that, in each iteration $i \in \{1,\ldots,n\}$ the same $\xi_{i-1}$ is used in both terms to account for the fact that the two It\^o-integrals $\chi_{i-1}^{\textrm{OU}}(u_i)$ and $\chi_{i-1}^{\textrm{Z}}$ are driven by the same Brownian motion.

To implement the Strang scheme, consider a realization of the underlying Brownian motion on an equidistant grid with time step size $\Delta/2$, and define
\begin{align*}
    \xi_{i-1}^1 &:= \frac{1}{\sqrt{\Delta/2}}\left( W_{t_{i-1}+\Delta/2} - W_{t_{i-1}} \right), \quad  \xi_{i-1}^2:= \frac{1}{\sqrt{\Delta/2}}\left( W_{t_{i}} - W_{t_{i-1}+\Delta/2} \right), \\ 
    \xi_{i-1} &:=\frac{1}{\sqrt{2}}\left( \xi_{i-1}^1+\xi_{i-1}^2 \right)=\frac{1}{\sqrt{\Delta}}\left( W_{t_i} -W_{t_{i-1}} \right) , \quad i=1,\ldots,n.
\end{align*}
Then, for each $i \in \{1,\ldots,n\}$, $\xi_{i-1}^1$ and $\xi_{i-1}^2$ are independent realizations of the standard normal distribution. Moreover, the $\xi_{i-1}$ (resp. the  $\xi_{i-1}^1$ or the $\xi_{i-1}^2$), for $i=1,\ldots,n$, are also independent standard normal samples. Therefore, sampling from the $\chi_{i-1}^{\textrm{Z,1}}$, $\chi_{i-1}^{\textrm{Z,2}}$, and $\chi_{i-1}^{\textrm{OU}}(u_i)$ via
\begin{equation*}
   \nu_\textrm{Z}(\Delta/2)\xi_{i-1}^1, \quad  \nu_\textrm{Z}(\Delta/2)\xi_{i-1}^2,  \quad \nu_{\textrm{OU}}(\Delta;{u}_i)\xi_{i-1},  \quad i=1,\ldots,n,
\end{equation*}
respectively, yields also an approximation of the solution path corresponding to the given Brownian path.
Note that, each iteration step $i \in \{ 1,\ldots,n \}$ of the Strang scheme involves the use of two independent standard normal samples (associated with $\Delta/2$-steps of the underlying Brownian path), making this scheme more costly to implement compared to the others.

\subsection{Continuous-time extensions of the splitting methods}\label{subsec:Splitting_contVersions}

The composition methods \eqref{eq:splittings_BM} and \eqref{eq:splittings_OU}, introduced in the previous two subsections for System~\eqref{eq:HH_redefined}\textbf{[BM]} and \eqref{eq:HH_redefined}\textbf{[OU]}, respectively, are used to construct continuous-time extensions $(\overline{X}_t)_{t\in[0,T]}$ of the proposed splitting schemes. These continuous-time extensions agree with the corresponding discrete-time approximations $(\widetilde{X}_{t_i})_{i=0,\ldots,n}$ at the grid points $t=t_{i}$, that is,  $\overline{X}_{t_i}=\widetilde{X}_{t_i}$ for all $i=0,\ldots,n$. They are subsequently employed in Section~\ref{sec:Splitting_properties} to establish boundedness of moments and mean-square convergence. 

For ease of readability, we restrict our attention to the second Lie–Trotter composition, noting that analogous representations can be derived for the other methods. 

\paragraph*{The \textbf{[BM]}-case.}

Let us denote $\overline{X}^{LT2}=\overline{X}=(\overline{V},\overline{U}^T)^T$ and set $\overline{X}_{t_0}=X_0$. Then, the continuous-time extension of the second Lie-Trotter splitting method for System \eqref{eq:HH_redefined}\textbf{[BM]} is defined by the recursive rule
\begin{equation}\label{eq:splittings_BM_cont}
    \begin{aligned}
        \overline{X}_{t} &= \left(\phi_{t-t_{i-1}}^{[2;\textrm{D}]}\circ \phi_{t-t_{i-1}}^{[ 1;\textrm{BM}]} \right)(\overline{X}_{t_{i-1}}),  \quad \forall t \in(t_{i-1},t_i] , \ 1\leq i\leq n.  
    \end{aligned}
\end{equation}
By definition of the flows $\phi^{[1;\textrm{BM}]}$ and $\phi^{[2;\textrm{D}]}$, and taking into account \eqref{eq:sub1_BM_flow_intEq} and \eqref{eq:sub2_BM_flow_intEq}, the 
second Lie-Trotter scheme is, $\forall t \in(t_{i-1},t_i] , \ 1\leq i\leq n$, given by
\begin{equation}\label{eq:LT2_BM_cont}
    \begin{aligned}
    \overline{V}_{t}&=\overline{V}_{t_{i-1}} + \int\limits_{t_{i-1}}^{t} 
   a(\overline{U}_{t_{i-1}}) \overline{V}_s + b(\overline{U}_{t_{i-1}}) \ ds + \int\limits_{t_{i-1}}^{t} \Sigma(\overline{U}_{t_{i-1}}) \ dW_s, \\
  \overline{U}_{t}&=\overline{U}_{t_{i-1}} + \int\limits_{t_{i-1}}^{t} 
    \left(-\textbf{diag}\left( \balpha\bigl( \overline{V}_{t} \bigr) + 
  \bbeta\bigl( \overline{V}_{t} \bigr) \right) \overline{U}_s \right)
  +\balpha \bigl( \overline{V}_{t} \bigr) \ ds.
\end{aligned}
\end{equation}

Note that, by construction, $\overline{V}_{t_i}=\widetilde{V}_{t_i}$ and $\overline{U}_{t_i}=\widetilde{U}_{t_i}$, for all $i=0,\ldots,n$. Moreover, $\overline{X}$ is an  adapted process, that is $\overline{X}_t$ is $\mathcal{F}_t$-measurable for every $t \in [0,T]$. 

\paragraph*{The \textbf{[OU]}-case.}

Let us denote $\overline{X}^{LT2}=\overline{X}=(\overline{V},\overline{U}^T,Z)^T$ and set $\overline{X}_{t_0}=X_0$. Then, the continuous-time extension of the second Lie-Trotter splitting method for System \eqref{eq:HH_redefined}\textbf{[OU]} is defined by the recursive rule
\begin{equation}\label{eq:splittings_OU_cont}
    \begin{aligned}
        \overline{X}_{t} &= \left(\phi_{t-t_{i-1}}^{[2;\textrm{Z}]}\circ \phi_{t-t_{i-1}}^{[ 1;\textrm{OU}]} \right)(\overline{X}_{t_{i-1}}),  \quad \forall t \in(t_{i-1},t_i] , \ 1\leq i\leq n.  
    \end{aligned}
\end{equation}
By definition of the flows $\phi^{[1;\textrm{OU}]}$ and $\phi^{[2;\textrm{Z}]}$, and taking into account \eqref{eq:sub1_OU_flow_intEq} and \eqref{eq:sub2_OU_flow_intEq}, the second Lie-Trotter scheme is, $\forall t \in(t_{i-1},t_i] , \ 1\leq i\leq n$, given by

\begin{equation}\label{eq:LT2_OU_cont}
    \begin{aligned}
    \overline{V}_{t}&=\overline{V}_{t_{i-1}} + \int\limits_{t_{i-1}}^{t} 
    a(\overline{U}_{t_{i-1}}) \overline{V}_s + b(\overline{U}_{t_{i-1}}) + \theta (\mu- Z_{t_{i-1}}) \ ds + \int\limits_{t_{i-1}}^{t} \sigma \ dW_s, \\
  \overline{U}_{t}&=\overline{U}_{t_{i-1}} + \int\limits_{t_{i-1}}^{t} 
    \left(-\textbf{diag}\left( \balpha\bigl( \overline{V}_{t} \bigr) + 
  \bbeta\bigl( \overline{V}_{t} \bigr) \right) \overline{U}_s \right)
  +\balpha \bigl( \overline{V}_{t} \bigr) \ ds, \\
  Z_t &= Z_{t_{i-1}} + \int\limits_{t_{i-1}}^t \theta(\mu-Z_s) \ ds + \int\limits_{t_{i-1}}^{t} \sigma \ dW_s.
\end{aligned}
\end{equation}

Again, by construction, we have that $\overline{V}_{t_i}=\widetilde{V}_{t_i}$ and $\overline{U}_{t_i}=\widetilde{U}_{t_i}$, for all $i=0,\ldots,n$, and that $\overline{X}_t$ is $\mathcal{F}_t$-measurable for every $t \in [0,T]$.

\section{A general localized version of the fundamental theorem of mean-square convergence}
\label{sec:fund-thm}

The purpose of this section is to establish a localized fundamental mean-square convergence theorem, together with a global mean-square convergence result, for general SDEs satisfying a local Lipschitz condition. In Section \ref{ssec:def_statements}, we introduce the relevant definitions and state the main results. The proofs of the local and global convergence results are provided in Sections \ref{ssec:proof_conv-speed-loc} and \ref{ssec:proof_conv-nospeed}, respectively. The convergence framework developed here covers the settings of Sections \ref{section:model} and \ref{section:splitting_schemes} and will be used in Section \ref{sec:Splitting_properties} to establish mean-square convergence of the proposed splitting~schemes.

\subsection{Some definitions and statement of the results}\label{ssec:def_statements}

\paragraph*{Considered class of SDEs.}

Let $D,Q\in\mathbb{N}\backslash \{0\}$, $T>0$ and consider the SDE
\begin{equation}
\label{eq:SDE-gene-2}
    dX_t=\mathbf{b}(t,X_t)dt+\bsigma(t,X_t)dW_t,\quad 0\leq t\leq T,
\end{equation}
where $\mathbf{b}:[0,T]\times\R^D\to\R^D$, $\bsigma:[0,T]\times\R^D\to\R^{D\times Q}$ and $W$ is $Q$-dimensional Brownian motion. Throughout, we assume that the coefficients $\mathbf{b}$ and $\bsigma$ are locally Lipschitz w.r.t. the space variable ({\it locally Lipschitz} in short in the sequel), that is, for any $R>0$, 
\begin{equation}
\label{eq:def-local-lip}
|\mathbf{b}(t,x)-\mathbf{b}(t,y)|+|\bsigma(t,x)-\bsigma(t,y)|\leq C_R|x-y|,\quad
\forall t\in [0,T],\forall
x,y\in B_R(0),
\end{equation}
where the constant $C_R>0$ depends on the radius $R$ of the centered ball $B_R(0)$.

Let $(\Omega,\cF,\P)$ be a probability space on which the Brownian motion $W$ is defined. We use the Brownian filtration $(\cF^W_t)$ and $\sigma(X_0)$ in order to define $(\cF_t)_{t\in[0,T]}$, the filtration satisfying the usual conditions. We further assume that there exists a unique strong solution $X$ to \eqref{eq:SDE-gene-2} (note that this is guaranteed, for example, if the condition \eqref{eq:def-local-lip} is fulfilled and the explosion time is infinite; cf. 
Subsection~\ref{ssec:existence-moments}). The solution $X$ is said to have bounded moments of order $2p$, $p\geq 1$, if it holds~that
\begin{equation}
    \label{eq:bounded-2-moments-sol}
    \E\left[\max_{0\leq t\leq T}|{X}_{t}|^{2p}\right]\leq C \left(1+\E \left[ |X_0|^{2p} \right]\right),
\end{equation}
where the constant $C>0$ depends on $p$, $T$, $\mathbf{b}$ and $\bsigma$.

\begin{remark}
    This setting is related to what has been described in Section~\ref{section:model}, but is formulated more generally. In particular, Systems \eqref{eq:HH_redefined}\textbf{[BM]} and \eqref{eq:HH_redefined}\textbf{[OU]} (see also Equations \eqref{eq:HH_Xfashion} and \eqref{eq:HHOU_Xfashion}) correspond to time-homogeneous cases of \eqref{eq:SDE-gene-2} with~$D=d+1$ and $Q=1$, and they satisfy  conditions \eqref{eq:def-local-lip} and~\eqref{eq:bounded-2-moments-sol} (cf. Theorems \ref{thm:bounded_moments_true_BM} and \ref{thm:existence-and-bounds-OU}).
\end{remark}

\paragraph*{Numerical schemes.} 

Let $t_0=0, \ldots, t_n=T$ be a discretization of the time interval $[0,T]$, where $n\in\mathbb{N}\backslash \{0\}$, $t_i=i\Delta$, for $i=0,\ldots,n$, with $\Delta=T/n$. Throughout, let $\Delta \in (0,1)$. 

Consider a stochastic numerical scheme $\widetilde{X}=(\widetilde{X}_{t_i})_{i=0,\ldots,n}$ such that  $\widetilde{X}_{t_0}=X_0$, and where $\widetilde{X}_{t_i}$ is thought to be a strong approximation of $X_{t_i}$, for any $i=1,\ldots,n$. Moreover, we assume that~$\widetilde{X}_{t_i}$ is~$\cF_{t_i}$-measurable for any $i=0,\ldots,n$. 

The study of convergence of a numerical scheme is carried out via a corresponding continuous-time extension (in the sequel also referred to as a numerical scheme, even though it is not implementable). In particular, by a continuous-time extension of a scheme $\widetilde{X}=(\widetilde{X}_{t_i})_{i=0,\ldots,n}$ we mean a process $\overline{X}=(\overline{X}_t)_{t\in[0,T]}$ such that 
\begin{equation}
\label{eq:tX-bX}
    \overline{X}_{t_i}=\widetilde{X}_{t_i},\quad \forall i=0,\ldots,n.
\end{equation}
Thus, the random variable $\overline{X}_{t_i}$ is also thought to be a strong approximation of $X_{t_i}$, for any 
\mbox{$i=0,\ldots,n$}, with the following structure
\begin{equation}
\label{eq:def-scheme-gene}
\overline{X}_{t}=F(t_{i-1},t-t_{i-1},\overline{X}_{t_{i-1}},\chi^t_{i-1}),\quad \forall 1\leq i\leq n,\;\;\forall t_{i-1}< t\leq t_i,
\end{equation}
where for any $i$ the function $F$ is Borel measurable and $\chi^t_{i-1}$  is a random variable depending on
$\{W_{t_{i-1}+r}-W_{t_{i-1}},\,0\leq r\leq t-t_{i-1}\}$  and $\overline{X}_{t_{i-1}}$.
Note that for any $0\leq t\leq T$ the random variable $\overline{X}_{t}$ is $\cF_{t}$-measurable (i.e., $\overline{X}$ is an adapted process), and for any $i=1,\ldots,n$ the random variable~$\overline{X}_{t_{i-1}}$ is independent of 
$\{W_{t_{i-1}+r}-W_{t_{i-1}}\,,0\leq r\leq \Delta\}$.  The scheme $\overline{X}$ is said to have bounded moments of order $2p$, $p\geq 1$, if it holds that
\begin{equation}
    \label{eq:bounded-2-moments}
    \E\left[\max_{0\leq t\leq T}|\overline{X}_{t}|^{2p}\right]\leq C \left(1+\E \left[ |X_0|^{2p} \right] \right),
\end{equation}
where the constant $C>0$ depends on $p$, $T$, $\mathbf{b}$ and $\bsigma$. Note that, compared to \cite{Milstein2004} we have a maximum over $t$ inside the expectation of \eqref{eq:bounded-2-moments}, as this will be needed in the proof of Theorem \ref{thm:conv-nospeed} below.

\begin{remark}
   The splitting schemes introduced in Section \ref{section:splitting_schemes} enter this general scope. For example, the second Lie-Trotter scheme \eqref{eq:splittings_BM_cont} of System \eqref{eq:HH_redefined}\textbf{[BM]} corresponds to 
\begin{equation*}
    \overline{X}_{t}=F(t_{i-1},t-t_{i-1},\overline{X}_{t_{i-1}},\chi^t_{i-1}) = \left( \phi_{t-t_{i-1}}^{[2;\textrm{D}]}\circ\phi_{t-t_{i-1}}^{[1;\textrm{BM}]} \right) (\overline{X}_{t_{i-1}}),
\end{equation*}
where $\chi^t_{i-1}=\Sigma(\overline{U}_{t_{i-1}})\int_{t_{i-1}}^te^{a(\overline{U}_{t_{i-1}})(t-s)}dW_s$, similarly to \eqref{eq:explicitV_BM}. Their moment boundedness \eqref{eq:bounded-2-moments} will be established in Section \ref{sec:Splitting_properties}.
\end{remark}

\paragraph*{Localized notion of consistency.}  

Consider the following stopping times. For any $R>0$, 
\begin{equation}\label{eq:stopping_times}
    \rho_R^{X}:= \inf \{ 0\leq t\leq T\; \colon |{X}_t| \geq R \}, \quad 
    \rho_R^{\overline{X}}:= \inf \{ 0\leq t\leq T\; \colon |\overline{X}_{t}| \geq R \}, \quad \rho_R:=\rho_R^{\overline{X}} \wedge \rho_R^X.
\end{equation}
In the same spirit we introduce another family of stopping times,
defined for any $R>0$ and $0\leq s\leq T$ (the case $s=0$ corresponds to \eqref{eq:stopping_times}):
\begin{equation}\label{eq:stopping_times-s}
\begin{gathered}
    \rho_R^{s,X}:= \inf \{ s\leq t\leq T \colon |{X}_t| \geq R \}, \quad 
    \rho_R^{s,\overline{X}}:= \inf \{ s\leq t\leq T \colon |\overline{X}_{t}| \geq R \}, \quad  \rho^s_R:=\rho_R^{s,\overline{X}} \wedge \rho_R^{s,X}.
\end{gathered}
\end{equation}

Moreover, the following notation is used: $X^{s,x}_t$ 
(resp. $\overline{X}^{s,x}_t$)
denotes the position of the process~$X$ (resp. the scheme $\overline{X}$) at time $t>s$, 
knowing that it is at point $x$ at time~$s$. 

Based on these definitions, we now introduce the following new notion of \textit{local consistency} for a numerical scheme.

\begin{definition}[Local consistency]\label{def:one_step_consistency}
    A scheme $\overline{X}$ for $X$ is said to be  {\it consistent of order} $q>1$ for the mean-square deviation {\it locally on the ball} $B_R(0)$, if for any $i=1,\ldots,n$  and any $x\in B_R(0)$, it holds that
\begin{equation}
    \label{eq:def-consist-q2}
    \left( \E\left[\Big|X^{t_{i-1},x}_{t_i\wedge\rho^{t_{i-1}}_R}-\overline{X}^{t_{i-1},x}_{t_i\wedge\rho^{t_{i-1}}_R}\Big|^2 \right] \right)^{1/2}\leq 
    C_R\Delta^{q},
\end{equation}
where the constant $C_R>0$ depends on $R$, $T$, $\mathbf{b}$ and $\bsigma$.
\end{definition}

\paragraph*{Statement of results.}

We are now in position to state the main results of the section. In~Theorem~\ref{thm:conv-speed-loc}, we state a localized version of the fundamental theorem of mean-square convergence (in the spirit of Theorem 1.1.1 by \cite{Milstein2004}, who considered the globally Lipschitz case). Then, we adapt the arguments of \cite{Higham2002} and prove that if Theorem~\ref{thm:conv-speed-loc} holds and the moments of both the process $X$ and the scheme $\overline{X}$ are bounded, then the scheme is globally mean-square convergent (without a guarantee of the speed of convergence). This is formalized in the statement of Theorem~\ref{thm:conv-nospeed}.

\begin{theorem}[Localized version of the fundamental mean-square convergence theorem]
\label{thm:conv-speed-loc}
   Assume that the coefficients $\mathbf{b}$ and $\bsigma$ of \eqref{eq:SDE-gene-2} are locally Lipschitz and there exists a unique strong solution~$X$ to \eqref{eq:SDE-gene-2}.
 Let $R>0$  and assume the scheme is consistent of order $q>1$  locally on the ball $B_R(0)$, in the sense of Definition~\ref{def:one_step_consistency}.
   Then, for any $n$ and any $i=0,\ldots,n$ the following inequality holds:
   \begin{equation*}
       \left( \E\left[ \Big|X_{t_i}-\overline{X}_{t_i}\Big|^2 \1_{t_i<\rho_R}  \right]  \right)^{1/2}\leq 
       C_R\,\Delta^{q-1},
   \end{equation*}
where the constant  $C_R>0$ depends on $R$, $T$, $\mathbf{b}$ and $\bsigma$. 
\end{theorem}

\begin{theorem}[Global mean-square convergence theorem]
\label{thm:conv-nospeed}
    Assume that the coefficients $\mathbf{b}$ and $\bsigma$ of \eqref{eq:SDE-gene-2} are locally Lipschitz and there exists a unique strong solution $X$ to \eqref{eq:SDE-gene-2}, which has bounded moments of order $2p$ for a certain $p>1$. 
   Let $\overline{X}$ be a scheme for $X$ that also has bounded moments of order $2p$ for a certain $p>1$. Further, assume that for any $R>0$ the scheme is consistent of order $q>1$ locally on the ball $B_R(0)$, in the sense of Definition~\ref{def:one_step_consistency}. Then, it holds that 
$$
\sup_{i=0,\ldots,n}\E\left[ |X_{t_i}-\overline{X}_{t_i}|^2 \right] \xrightarrow[\Delta\downarrow 0]{}0.
$$
    \end{theorem}

Since the continuous-time scheme $\overline{X}$ is constructed in a way to coincide with the corresponding discrete-time scheme $\widetilde{X}$ at the grid points $t=t_i$, that is $\overline{X}_{t_i}=\widetilde{X}_{t_i}$ for all $i=0,\ldots,n$, Theorem \ref{thm:conv-nospeed} readily implies  
$$
    \sup_{i=0,\ldots,n}\E\left[ |X_{t_i}-\widetilde{X}_{t_i}|^2 \right] \xrightarrow[\Delta\downarrow 0]{}0.
$$

\subsection{Proof of Theorem \ref{thm:conv-speed-loc}}\label{ssec:proof_conv-speed-loc}

The proof of Theorem \ref{thm:conv-speed-loc} is based on the following two technical lemmas. They are an adaptation of Lemma 1.1.3 in \cite{Milstein2004}, tailored to the considered local context. 

Before stating them, note that for any $i=1,\ldots,n$, any $R>0$, and any 
$x,y\in B_R(0)$, the following representation holds
\begin{equation}
\label{eq:pres-Z}
    X^{t_{i-1},x}_{t_i\wedge\rho_R^{t_{i-1}}}-X^{t_{i-1},y}_{t_i\wedge\rho_R^{t_{i-1}}}=x-y+Z,
\end{equation}
where
$$
Z=\int_{t_{i-1}}^{t_i\wedge\rho_R^{t_{i-1}}}\big( \mathbf{b}(s,X^{t_{i-1},x}_s)-\mathbf{b}(s,X^{{t_{i-1}},y}_s) \big)ds
+\int_{t_{i-1}}^{t_i\wedge\rho_R^{t_{i-1}}}\big( \bsigma(s,X^{{t_{i-1}},x}_s)-\bsigma(s,X^{t_{i-1},y}_s) \big)dW_s.
$$

\begin{lemma}
\label{lem:diffXxy}
Assume that the coefficients $\mathbf{b}$ and $\bsigma$ of \eqref{eq:SDE-gene-2} are locally Lipschitz and there exists a unique strong solution $X$ to \eqref{eq:SDE-gene-2}. Then, for any $i=1,\ldots,n$, any $R>0$, and any 
$x,y\in B_R(0)$, it holds that
\begin{equation}
\label{eq:diff-X-startingpoints}
    \E \left[ \big|
    X^{t_{i-1},x}_{t_i\wedge\rho_R^{t_{i-1}}}-X^{t_{i-1},y}_{t_i\wedge\rho_R^{t_{i-1}}}\big|^2 \right] \leq |x-y|^2(1+C_R\Delta), 
\end{equation}
where the constant $C_R>0$ depends on $R$, $\mathbf{b}$ and $\bsigma$.
\end{lemma}

\begin{proof}
Using Itô's formula for the function $x\mapsto|x|^2$ and the 
 process
 $\big(X^{{t_{i-1}},x}_{t\wedge\rho_R^{t_{i-1}}}-X^{t_{i-1},y}_{t\wedge\rho_R^{t_{i-1}}}\big)_{t_{i-1}\leq t\leq t_i}$, for any $t \in (t_{i-1},t_i]$, one has 
\begin{align*}
    \E \left[ \big|
    X^{t_{i-1},x}_{t\wedge\rho_R^{t_{i-1}}}-X^{t_{i-1},y}_{t\wedge\rho_R^{t_{i-1}}}\big|^2 \right]  &= \ds \E\Big[ |x-y|^2 \\
    & \qquad \ds +2\int_{t_{i-1}}^{t\wedge\rho_R^{t_{i-1}}}(X^{t_{i-1},x}_s-X^{t_{i-1},y}_s)^T\big( \mathbf{b}(s,X^{t_{i-1},x}_s)-\mathbf{b}(t_{i-1},X^{s,y}_s) \big)ds
    \\
    & \qquad  \ds + 2\int_{t_{i-1}}^{t\wedge\rho_R^{t_{i-1}}}(X^{t_{i-1},x}_s-X^{t_{i-1},y}_s)^T\big( \bsigma(s,X^{t_{i-1},x}_s)-\bsigma(s,X^{t_{i-1},y}_s) \big)dW_s
    \\
    & \qquad  \ds +\int_{t_{i-1}}^{t\wedge\rho_R^{t_{i-1}}}\big| \bsigma(s,X^{t_{i-1},x}_s)-\bsigma(s,X^{t_{i-1},y}_s) \big|^2ds\Big].
\end{align*}
Note that in the above expression we have used the matrix type notation for the stochastic integral terms (namely 
$\int Y_s^T\Sigma_sdW_s=\sum_{kj}\int Y^k_s\Sigma^{kj}_sdW^j_s$). For any 
$t_{i-1}\leq s\leq \rho_R^{t_{i-1}}$, the integrands in those terms are bounded, as 
$X_s$ is in $B_R(0)$ and $\bsigma$ is Lipschitz continuous on this ball w.r.t. the space variable. Thus, the second right-hand side term inside the brackets in the above equation may be seen  as a square integrable martingale, stopped at stopping time $\rho_R^{t_{i-1}}$, and taken at time $t$: by the stopping theorem (see Corollary II.3.6 in \cite{revuz1999}) this is a square-integrable martingale, taken at time $t$, starting at zero at time $t_{i-1}$. Thus, the expectation of  the second right-hand side term inside the brackets vanishes.

Then, from the local Lipschitz condition \eqref{eq:def-local-lip} we get 
\begin{align*}
\E \left[ \big|
    X^{t_{i-1},x}_{t\wedge\rho_R^{t_{i-1}}}-X^{t_{i-1},y}_{t\wedge\rho_R^{t_{i-1}}}\big|^2 \right]
    &\leq \ds |x-y|^2+2C_R\E \left[ \int_{t_{i-1}}^{t\wedge\rho_R^{t_{i-1}}}\big|X^{t_{i-1},x}_s-X^{t_{i-1},y}_s\big|^2ds \right] 
    \\
    &\leq \ds|x-y|^2+2C_R\E \left[ \int_{t_{i-1}}^{t}\big|X^{t_{i-1},x}_{s\wedge\rho_R^{t_{i-1}}}-X^{t_{i-1},y}_{s\wedge\rho_R^{t_{i-1}}}\big|^2ds \right] 
    \\
    &\leq \ds|x-y|^2+2C_R\int_{t_{i-1}}^{t}\E \left[ \big|X^{t_{i-1},x}_{s\wedge\rho_R^{t_{i-1}}}-X^{t_{i-1},y}_{s\wedge\rho_R^{t_{i-1}}}\big|^2 \right] ds,
\end{align*}
where the last inequality follows from Fubini's theorem.

Using now Gr\"onwall's inequality, we get
\begin{equation}\label{eq:groenwall_fund}
     \E \left[ \big|X^{t_{i-1},x}_{t\wedge\rho_R^{t_{i-1}}}-X^{t_{i-1},y}_{t\wedge\rho_R^{t_{i-1}}}\big|^2 \right] \leq |x-y|^2e^{2C_R\Delta}, \qquad \forall t \in (t_{i-1},t_i],
\end{equation}
which leads to \eqref{eq:diff-X-startingpoints}, since $\Delta\in (0,1)$.
\end{proof}

\begin{lemma}
\label{lem:Zxy}
Assume that the coefficients $\mathbf{b}$ and $\bsigma$ of \eqref{eq:SDE-gene-2} are locally Lipschitz and there exists a unique strong solution $X$ to \eqref{eq:SDE-gene-2}. Then, for any $i=1,\ldots,n$, any $R>0$, and any 
$x,y\in B_R(0)$, it holds that
$$
X^{t_{i-1},x}_{t_i\wedge\rho_R^{t_{i-1}}}-X^{t_{i-1},y}_{t_i\wedge\rho_R^{t_{i-1}}}=x-y+Z,  
$$
with
$$
\E \left[\big|Z\big|^2\right] \leq C_R|x-y|^2\Delta,
$$
where the constant $C_R>0$ depends on $R$, $\mathbf{b}$ and $\bsigma$.
\end{lemma}

\begin{proof}
Recall the representation \eqref{eq:pres-Z}. Then, we have 
\begin{align*}
&\E \left[ \big|Z\big|^2 \right] \\
& \quad \leq \ds c\E \left[ \int_{t_{i-1}}^{t_i\wedge\rho_R^{t_{i-1}}}\big| \mathbf{b}(s,X^{t_{i-1},x}_s)-\mathbf{b}(s,X^{t_{i-1},y}_s) \big|^2ds
+ \int_{t_{i-1}}^{t_i\wedge\rho_R^{t_{i-1}}}\big| \bsigma(s,X^{t_{i-1},x}_s)-\bsigma(s,X^{t_{i-1},y}_s) \big|^2ds \right]
\\
& \quad \leq \ds c'\E \left[ \int_{t_{i-1}}^{t_i}\big| \mathbf{b}(s,X^{t_{i-1},x}_{s\wedge\rho_R^{t_{i-1}}})-\mathbf{b}(s,X^{t_{i-1},y}_{s\wedge\rho_R^{t_{i-1}}}) \big|^2ds
+ \int_{t_{i-1}}^{t_i}\big| \bsigma(s,X^{t_{i-1},x}_{s\wedge\rho_R^{t_{i-1}}})-\bsigma(s,X^{t_{i-1},y}_{s\wedge\rho_R^{t_{i-1}}}) \big|^2ds \right] 
\\
& \quad \leq  \ds C_R\, \int_{t_{i-1}}^{t_i} \E \left[ \big|X^{t_{i-1},x}_{s\wedge\rho_R^{t_{i-1}}}-X^{t_{i-1},y}_{s\wedge\rho_R^{t_{i-1}}}\big|^2 \right] ds.
\end{align*}

At the first inequality we have used the Hölder and Burkholder-Davis-Gundy inequalities, and at the third inequality we have used the local Lipschitz condition \eqref{eq:def-local-lip}. Using \eqref{eq:groenwall_fund}, we obtain the~result.
\end{proof}


Towards the end of the proof of Theorem \ref{thm:conv-speed-loc}
 we will use the same discrete Gr\"onwall lemma as in \cite{Milstein2004}. We recast this lemma below, for the sake of completeness.

 \begin{lemma}[Discrete Gr\"onwall lemma]
 \label{lem:discr-gronwall}
Let $(\epsilon_i)_{i=0}^n$ be a sequence of non-negative real numbers satisfying
\begin{equation}
    \label{eq:hyp-disgronwall-lemma}
\epsilon_i\leq (1+A\Delta)\epsilon_{i-1}+B\Delta^p,\quad \forall i=1,\ldots,n,
\end{equation}
where $A,B\geq 0$ and $p\geq 1$ (we recall that $\Delta=T/n$). Then,
$$
\epsilon_i\leq e^{AT}\epsilon_0+\frac B A(e^{AT}-1)\Delta^{p-1}, \quad\forall i=0,\ldots,n,
$$
where for $A=0$, $(e^{AT}-1)/A$ is equal to $0$ by convention.
 \end{lemma}

\paragraph*{Proof of the theorem.}

We are now in position to proceed, in several steps, to the proof of Theorem \ref{thm:conv-speed-loc} itself. The idea is to set $\varepsilon_i^2= \E\big[ |X_{t_i}-\overline{X}_{t_i}|^2 \1_{t_i<\rho_R}  \big]$ and to get a relation of type
 \eqref{eq:hyp-disgronwall-lemma} for the sequence $(\varepsilon_i^2)$. The result will then follow from Lemma \ref{lem:discr-gronwall}.  

\vspace{0.3cm}

{\it Step 1: Conditioning.} Note that, for any Borel bounded function $f$ the quantity $\E [f (X^{s,x}_t)]$ may be seen as the conditional expectation $\E[f(X_t)\,|\,X_s=x]$. Thus, we denote here the random variable $\E[f(X_t)\,|\,X_s]$ as $\E [f (X^{s,x}_t)]_{x=X_s}$.
Let $1\leq i\leq n$. Then, we have
\begin{equation}\label{eq:cond-step1}
\begin{aligned} 
 \E\left[\,\big|
    X_{t_i}-\overline{X}_{t_i}\big|^2\1_{t_i<\rho_R}\,\right]&=
    \E\left[\E\big[\,\big|
    X_{t_i}-\overline{X}_{t_i}\big|^2\1_{t_{i-1}<\rho_R}\1_{t_{i}<\rho^{t_{i-1}}_R}\,\big|\,\cF_{t_{i-1}}\big]\right]\\[-4mm]
    \\
    &=\E\left[\1_{t_{i-1}<\rho_R}\E\big[\,\big|
    X_{t_i\wedge\rho^{t_{i-1}}_R}-\overline{X}_{t_i\wedge\rho^{t_{i-1}}_R}\big|^2\1_{t_i<\rho^{t_{i-1}}_R}\,\big|\,\cF_{t_{i-1}}\big]\right]\\
    \\[-4mm]
    &=  \E\left[\1_{t_{i-1}<\rho_R}\E\big[\,\big|
    X_{t_i\wedge\rho^{t_{i-1}}_R}-\overline{X}_{t_i\wedge\rho^{t_{i-1}}_R}\big|^2\1_{t_i<\rho^{t_{i-1}}_R}\,\big|\,X_{t_{i-1}},
    \overline{X}_{t_{i-1}}\big]\right]\\[-4mm]
    \\
    &=  \E\left[\1_{t_{i-1}<\rho_R}\E\Big[\,\big|
    X^{t_{i-1},x}_{t_i\wedge\rho^{t_{i-1}}_R}-\overline{X}^{t_{i-1},\bar{x}}_{t_i\wedge\rho^{t_{i-1}}_R}\big|^2\1_{t_i<\rho^{t_{i-1}}_R}\,\Big]_{x=X_{t_{i-1}},
    \bar{x}=\overline{X}_{t_{i-1}}}\right].\\
\end{aligned}
\end{equation}
At the first equality we have used the fact that
${\{t_i<\rho_R\}}={\{t_{i-1}<\rho_R\}}\cap{\{t_{i}<\rho^{t_{i-1}}_R\}}$. At the second inequality we have used $\{ \rho_R > t_{i-1} \} \in \mathcal{F}_{t_{i-1}}$. At the third equality we have used the Markov property for $X$ and the structure of $\overline{X}$ (see \eqref{eq:def-scheme-gene} and the comments hereafter). Note that on the event $\{t_{i-1}<\rho_R\}$ both $X_{t_{i-1}}$ and $\overline{X}_{t_{i-1}}$ are in $B_R(0)$.

\vspace{0.1cm}
{\it Step 2: Decomposition of the error.} We thus aim at controlling,
for $x,\bar{x}\in B_R(0)$, the error

\begin{equation}
\begin{aligned}
& \E\Big[\,\big|
    X^{t_{i-1},x}_{t_i\wedge\rho^{t_{i-1}}_R}-\overline{X}^{t_{i-1},\bar{x}}_{t_i\wedge\rho^{t_{i-1}}_R}\big|^2\1_{t_i<\rho^{t_{i-1}}_R}\,\Big] \\[-4mm] \\
& \qquad =\E\Big[ \,\big|
X^{t_{i-1},x}_{t_i\wedge\rho^{t_{i-1}}_R}
-X^{t_{i-1},\bar{x}}_{t_i\wedge\rho^{t_{i-1}}_R}
+X^{t_{i-1},\bar{x}}_{t_i\wedge\rho^{t_{i-1}}_R}
-\overline{X}^{t_{i-1},\bar{x}}_{t_i\wedge\rho^{t_{i-1}}_R}
\big|^2 \1_{t_i<\rho^{t_{i-1}}_R}\, \Big]\\[-4mm] \\
& \qquad =\E\Big[ \,\big| 
X^{t_{i-1},x}_{t_i\wedge\rho^{t_{i-1}}_R}
-X^{t_{i-1},\bar{x}}_{t_i\wedge\rho^{t_{i-1}}_R}
\big|^2 \1_{t_i<\rho^{t_{i-1}}_R}\, \Big]
+\,\E\Big[ \,
\big|X^{t_{i-1},\bar{x}}_{t_i\wedge\rho^{t_{i-1}}_R}
-\overline{X}^{t_{i-1},\bar{x}}_{t_i\wedge\rho^{t_{i-1}}_R}\big|^2
\1_{t_i<\rho^{t_{i-1}}_R}\, 
\Big]\\[-4mm] \\
&\qquad \qquad +\,2\E\Big[ \,(
X^{t_{i-1},x}_{t_i\wedge\rho^{t_{i-1}}_R}
-X^{t_{i-1},\bar{x}}_{t_i\wedge\rho^{t_{i-1}}_R}
)^T\,
(  X^{t_{i-1},\bar{x}}_{t_i\wedge\rho^{t_{i-1}}_R}
-\overline{X}^{t_{i-1},\bar{x}}_{t_i\wedge\rho^{t_{i-1}}_R}) \1_{t_i<\rho^{t_{i-1}}_R}\, \Big].\\
\end{aligned}
\end{equation}

\vspace{0.1cm}
{\it Step 3: Treating each term.}
Using Lemma \ref{lem:diffXxy}, for the first term we get directly
\begin{equation}
\label{eq:step2a}
    \E\Big[ \,\big| 
X^{t_{i-1},x}_{t_i\wedge\rho^{t_{i-1}}_R}
-X^{t_{i-1},\bar{x}}_{t_i\wedge\rho^{t_{i-1}}_R}
\big|^2 \1_{t_i<\rho^{t_{i-1}}_R}\, \Big]\leq 
\E\Big[ \,\big| 
X^{t_{i-1},x}_{t_i\wedge\rho^{t_{i-1}}_R}
-X^{t_{i-1},\bar{x}}_{t_i\wedge\rho^{t_{i-1}}_R}
\big|^2 \, \Big]
\leq 
(1+C_R\Delta) |x-\bar{x}|^2.
\end{equation}

\vspace{0.3cm}
The second term is also treated directly, thanks to the assumption that the scheme is consistent of order~$q$ locally on the ball $B_R(0)$ (in the mean-square sense). Specifically, one has
\begin{equation}
\label{eq:step2b}
\E\Big[ \,
\big|X^{t_{i-1},\bar{x}}_{t_i\wedge\rho^{t_{i-1}}_R}
-\overline{X}^{t_{i-1},\bar{x}}_{t_i\wedge\rho^{t_{i-1}}_R}\big|^2
\1_{t_i<\rho^{t_{i-1}}_R}\, 
\Big]\leq C_R\Delta^{2q}.
\end{equation}

\vspace{0.2cm}
Recall that $X^{t_{i-1},x}_{t_i\wedge\rho^{t_{i-1}}_R}
-X^{t_{i-1},\bar{x}}_{t_i\wedge\rho^{t_{i-1}}_R}=x-\bar{x}+Z$, where $Z$ is as in~\eqref{eq:pres-Z}. Thus, for the third term we have
\begin{equation}
\begin{array}{l}
\E\Big[ \,(
X^{t_{i-1},x}_{t_i\wedge\rho^{t_{i-1}}_R}
-X^{t_{i-1},\bar{x}}_{t_i\wedge\rho^{t_{i-1}}_R}
)^T\,
(  X^{t_{i-1},\bar{x}}_{t_i\wedge\rho^{t_{i-1}}_R}
-\overline{X}^{t_{i-1},\bar{x}}_{t_i\wedge\rho^{t_{i-1}}_R}) \1_{t_i<\rho^{t_{i-1}}_R}\, \Big]\\
\\
\hspace{1cm}=
\E\Big[ \,(
x-\bar{x}
)^T\,
(  X^{t_{i-1},\bar{x}}_{t_i\wedge\rho^{t_{i-1}}_R}
-\overline{X}^{t_{i-1},\bar{x}}_{t_i\wedge\rho^{t_{i-1}}_R}) \1_{t_i<\rho^{t_{i-1}}_R}\, \Big]
+\E\Big[ \,Z^T\,
(  X^{t_{i-1},\bar{x}}_{t_i\wedge\rho^{t_{i-1}}_R}
-\overline{X}^{t_{i-1},\bar{x}}_{t_i\wedge\rho^{t_{i-1}}_R}) \1_{t_i<\rho^{t_{i-1}}_R}\, \Big].
\end{array}
\end{equation}
Using the Minkowski and Jensen inequalities, we obtain
$$
   \Big| \,\E\Big[ \,(
x-\bar{x}
)^T\,
(  X^{t_{i-1},\bar{x}}_{t_i\wedge\rho^{t_{i-1}}_R}
-\overline{X}^{t_{i-1},\bar{x}}_{t_i\wedge\rho^{t_{i-1}}_R}) \1_{t_i<\rho^{t_{i-1}}_R}\, \Big]  
\,\Big|
\leq  |x-\bar{x}|\Big(  \E\Big[ \,
\big|X^{t_{i-1},\bar{x}}_{t_i\wedge\rho^{t_{i-1}}_R}
-\overline{X}^{t_{i-1},\bar{x}}_{t_i\wedge\rho^{t_{i-1}}_R}\big|^2
\, 
\Big] \Big)^{1/2}.
$$
Thus, local consistency yields 
\begin{equation}
\label{eq:stepc1}
    \Big| \,\E\Big[ \,(
x-\bar{x}
)^T\,
(  X^{t_{i-1},\bar{x}}_{t_i\wedge\rho^{t_{i-1}}_R}
-\overline{X}^{t_{i-1},\bar{x}}_{t_i\wedge\rho^{t_{i-1}}_R}) \1_{t_i<\rho^{t_{i-1}}_R}\, \Big]  
\,\Big|\leq |x-\bar{x}|C_R\Delta^q.
\end{equation}
Moreover, using the Cauchy-Schwarz inequality, one has
$$
\Big|
\E\Big[ \,Z^T\,
(  X^{t_{i-1},\bar{x}}_{t_i\wedge\rho^{t_{i-1}}_R}
-\overline{X}^{t_{i-1},\bar{x}}_{t_i\wedge\rho^{t_{i-1}}_R}) \1_{t_i<\rho^{t_{i-1}}_R}\, \Big]
\Big|
\leq \Big(\E \left[|Z|^2\right]\Big)^{1/2}\\
\,\Big(
\E\Big[ \,
\big|X^{t_{i-1},\bar{x}}_{t_i\wedge\rho^{t_{i-1}}_R}
-\overline{X}^{t_{i-1},\bar{x}}_{t_i\wedge\rho^{t_{i-1}}_R}\big|^2
\, 
\Big]
\Big)^{1/2}.
$$
Using then Lemma \ref{lem:Zxy} and local consistency, we get
 \begin{equation}
\label{eq:stepc2}
     \Big|
\E\Big[ \,Z^T\,
(  X^{t_{i-1},\bar{x}}_{t_i\wedge\rho^{t_{i-1}}_R}
-\overline{X}^{t_{i-1},\bar{x}}_{t_i\wedge\rho^{t_{i-1}}_R}) \1_{t_i<\rho^{t_{i-1}}_R}\, \Big]
\Big|
\leq 
C_R\,\Delta^{q+\frac12}|x-\bar{x}|.
 \end{equation}
Taking into account \eqref{eq:stepc1} and \eqref{eq:stepc2} we finally get
\begin{equation}
    \label{eq:step2c}
    \E\Big[ \,(
X^{t_{i-1},x}_{t_i\wedge\rho^{t_{i-1}}_R}
-X^{t_{i-1},\bar{x}}_{t_i\wedge\rho^{t_{i-1}}_R}
)^T\,
(  X^{t_{i-1},\bar{x}}_{t_i\wedge\rho^{t_{i-1}}_R}
-\overline{X}^{t_{i-1},\bar{x}}_{t_i\wedge\rho^{t_{i-1}}_R}) \1_{t_i<\rho^{t_{i-1}}_R}\, \Big]
\leq 
C_R(1+\Delta^{1/2})\Delta^q|x-\bar{x}|.
\end{equation}

\vspace{0.2cm}
{\it Step 4: Putting all the pieces together and getting a relation of type \eqref{eq:hyp-disgronwall-lemma}.}

\vspace{0.2cm}
We set $\varepsilon_i^2=\E\big[ |X_{t_i}-\overline{X}_{t_i}|^2 \1_{t_i<\rho_R}  \big]$ for $i=0,\ldots,n$. Taking into account the decomposition in \textit{Step 2}, plugging the estimates
\eqref{eq:step2a}, \eqref{eq:step2b} and \eqref{eq:step2c} 
into \eqref{eq:cond-step1}, and using the Cauchy-Schwarz inequality (for the second term), 
we get
$$
\varepsilon_i^2\leq (1+C_R\Delta)\varepsilon_{i-1}^2+2C_R(1+\Delta^{1/2})\Delta^q\varepsilon_{i-1}+C_R\Delta^{2q}, \quad i=1,\ldots,n.
$$
Let $q'=q-\frac12$. We have, replacing the constant $C_R$ by another constant $C_R'$ depending also on~$R,T,\mathbf{b},\bsigma$ (for the relevant terms),
$$
\varepsilon_i^2\leq (1+C_R\Delta)\varepsilon_{i-1}^2+C_R'\Delta^{q'+\frac12}\varepsilon_{i-1}+C_R'\Delta^{2q'}, \quad i=1,\ldots,n.
$$
Using Young's inequality, we have 
$\Delta^{q'+\frac12}\varepsilon_{i-1}\leq \dfrac{\Delta^{2q'}}{2}+\dfrac{\Delta\varepsilon_{i-1}^2}{2}$, and thus (changing again the constant~$C_R$)
$$
\varepsilon_i^2\leq (1+C_R\Delta)\varepsilon_{i-1}^2+C_R\Delta^{2q'}, \quad i=1,\ldots,n.
$$

\vspace{0.2cm}
{\it Step 5: Concluding.} Using Lemma \ref{lem:discr-gronwall}, we then have
$$
\varepsilon_i^2\leq e^{C_RT}\varepsilon^2_0+(e^{C_RT}-1)\Delta^{2q'-1}, \quad i=0,\ldots,n.
$$
But as $\varepsilon_0=0$ (the process $X$ and the scheme $\overline{X}$ start from the same point), we get
$$
\varepsilon_i=\left( \E\big[ |X_{t_i}-\overline{X}_{t_i}|^2 \1_{t_i<\rho_R}  \big]  \right)^{1/2}\leq C_R\Delta^{q'-1/2},
$$
for any $i=0,\ldots,n$. Recalling that $q'=q-\frac12$ we conclude the proof of the theorem.

\begin{remark}
    In the original work of \cite{Milstein2004} (see their Theorem 1.1.1, p5) two orders of consistency are involved. An order of mean-square consistency, say $q_2$, which plays the same role as our order $q$, but also an order of "linear" consistency, say $q_1$. Indeed, in their study of the global Lipschitz case no indicator function $\1_{t_i<\rho_R^{t_{i-1}}}$ appears in a term analogous to~\eqref{eq:stepc1}. Thus, if there is an order of linear consistency, it allows treating this term in a specific fashion (mainly because expectations of Itô integrals vanish in this case). But in our computations, in the presence of the indicator function, one finally has to resort to the Cauchy-Schwarz inequality, and is led to use the mean-square order of accuracy $q$.
\end{remark}

\subsection{Proof of Theorem \ref{thm:conv-nospeed}}\label{ssec:proof_conv-nospeed}

To prove Theorem \ref{thm:conv-nospeed}, we proceed in a similar way as in Theorem 2.2 of \cite{Higham2002}. First, we observe that
\[
\sup_{i=0,\ldots,n}\E\big[ |X_{t_i}-\overline{X}_{t_i}|^2 \big] \leq \sup_{i=0,\ldots,n}\E\big[ |X_{t_i}-\overline{X}_{t_i}|^2\mathbbm{1}_{\rho_R>t_i} \big] + \sup_{i=0,\ldots,n}\E\big[ |X_{t_i}-\overline{X}_{t_i}|^2 \mathbbm{1}_{\rho_R\leq t_i}\big].
\]
Thanks to Theorem \ref{thm:conv-speed-loc}, the first term is controlled and satisfies
\[
\sup_{i=0,\ldots,n}\E\big[ |X_{t_i}-\overline{X}_{t_i}|^2\mathbbm{1}_{\rho_R>t_i} \big] \leq C_R\Delta^{2(q-1)}.
\]
Let us focus on the second term. We start by recalling Young's inequality, which states that for any $r,k>1$ with $r^{-1}+k^{-1}=1$, it holds that
\[
ab \leq \frac{\delta}{r}a^r+\frac{1}{k\delta^{k/r}}b^k, \qquad \forall a,b,\delta>0.
\]
Thus, by taking $r=p$ and $k= \frac{p}{p-1}$, one has
\begin{align*}
\sup_{i=0,\ldots,n}\E\big[|X_{t_i}-\overline{X}_{t_i}|^2 \mathbbm{1}_{\rho_R\leq t_i}\big] &\leq \frac{\delta}{p}\sup_{i=0,\ldots,n}\E\left[|X_{t_i}-\overline{X}_{t_i}|^{2p}  \right] + \frac{p-1}{p\delta^{1/(p-1)}}\sup_{i=0,\ldots,n}\mathbb{E}[\mathbbm{1}_{\rho_R\leq t_i}].
\end{align*}
Then, the elementary inequality $(x+y)^{2p}\leq 2^{2p-1}(x^{2p}+y^{2p})$, for all $x,y\geq 0$, together with the fact that $\mathbb{E}[\mathbbm{1}_{\rho_R\leq t_i}]=\P(\rho_R\leq t_i)\leq\P(\rho_R^X\leq t_i)+\P(\rho_R^{\overline{X}}\leq t_i)$, further gives that
\begin{align*}
\sup_{i=0,\ldots,n}\E\big[|X_{t_i}-\overline{X}_{t_i}|^2 \mathbbm{1}_{\rho_R\leq t_i}\big] &\leq \frac{2^{2p-1}\delta}{p} \sup_{i=0,\ldots,n}\E\left[|X_{t_i}|^{2p} + |\overline{X}_{t_i}|^{2p}  \right] \\ & \qquad + \frac{p-1}{p\delta^{1/(p-1)}}\sup_{i=0,\ldots,n}\big(\P\big( \rho^X_R\leq t_i\big)+\P\big( \rho^{\overline{X}}_R\leq t_i\big)\big).
\end{align*}
Now, since $X_{\rho_R^X} = R$ on the event $\{ \rho_R^X \leq t_i \}$, one has
\begin{equation*}
   \sup_{i=0,\ldots,n} \P(\rho_R^X\leq t_i) = \sup_{i=0,\ldots,n} \mathbb{E}\left[ \mathbbm{1}_{\rho_R^X\leq t_i} \frac{|{X}_{\rho^{{X}}_R}|^{2p} }{R^{2p}} \right] \leq \frac{1}{R^{2p}} \E\left[\max_{0\leq t\leq T}|{X}_{t}|^{2p}  \right],
\end{equation*}
and a similar result holds for the term $\sup_{i=0,\ldots,n}\P(\rho_R^{\overline{X}}\leq t_i)$. Therefore, we further obtain that
\begin{align*}
\sup_{i=0,\ldots,n}\E\big[|X_{t_i}-\overline{X}_{t_i}|^2 \mathbbm{1}_{\rho_R\leq t_i}\big] 
& \leq \left(\frac{2^{2p-1}\delta}{p} + \frac{p-1}{p\delta^{1/(p-1)}R^{2p}}\right) \left(\E\left[\max_{0\leq t\leq T}|X_{t}|^{2p}\right] + \E\left[\max_{0\leq t\leq T}|\overline{X}_{t}|^{2p}  \right]\right).
\end{align*}
Using the fact that the moments of both the true process and the numerical scheme are bounded by assumption, we can define the following constant: 
\[
A := \max\left\{ \E\left[\max_{0\leq t\leq T}|X_{t}|^{2p}\right], \E\left[\max_{0\leq t\leq T}|\overline{X}_{t}|^{2p}  \right]\right\}.
\]
Then, the upper bound for the accumulated error can be written as 
\begin{align*}
\sup_{i = 0,\dots, n}\E\big[|X_{t_i}-\overline{X}_{t_i}|^2 \big] \leq C_R\Delta^{2(q-1)} + \frac{2^{2p}A\delta}{p}+ \frac{(p-1)2A}{p\delta^{1/(p-1)}R^{2p}}.
\end{align*}
Given any $\epsilon>0$, we can choose $\delta$ such that $({2^{2p}A\delta})/{p} < {\epsilon}/{3}$. Then, $R$ can be  chosen such that 
\[
 \frac{(p-1)2A}{p\delta^{1/(p-1)}R^{2p}} < \frac{\epsilon}{3}.
\]
Finally, choose the time step $\Delta$ sufficiently small, such that $C_R\Delta^{2(q-1)}  < {\epsilon}/{3}.$
Therefore, for any $\epsilon > 0$ it is possible to choose $\Delta$ small enough, such that
\[
\sup_{i=0,\ldots,n}\E\big[ |X_{t_i}-\overline{X}_{t_i}|^2 \big] \leq \epsilon,
\]
which completes the proof of the theorem.

\section{Properties of the splitting methods}\label{sec:Splitting_properties}

In this section, we analyse the properties of the splitting methods introduced in Section \ref{section:splitting_schemes}. In~particular, we prove that they preserve the state space of the process (Section~\ref{sec:stateSpacePreservation_splitting}), have bounded moments (Section~\ref{sec:boundedMoments_splitting}), are locally consistent and globally mean-square convergent (Section~\ref{sec:mean_square_convergence_splitting}), and are geometrically ergodic (Section~\ref{sec:ergodicity_splitting}). The strong convergence result of Section~\ref{sec:mean_square_convergence_splitting} is based on the findings of Section \ref{sec:fund-thm}.

\subsection{State space preservation}\label{sec:stateSpacePreservation_splitting}

We start with the following lemma.

\begin{lemma}\label{lemma:subunits_stay_in_cube}
Let the function $\psi_{\textrm{D}}$ be as in \eqref{eq:sub2_BM_flow}. Then, for any $v \in \mathbb{R}$, $u \in [0,1]^d$, and $t>0$, it holds that $\psi_{\textrm{D}}(t;v,u) \in [0,1]^d$.
\end{lemma}

\begin{proof} 
First, note that the function $\psi_{\textrm{D}}$ can be expressed as a vector of functions 
\begin{equation*}
   \psi_{\textrm{D}}(t;v,u) = \left(\psi_1(t;v,u^1),\ldots \psi_d(t;v,u^d)\right)^T,
\end{equation*}
where, for $l \in \{1,\ldots,d\}$, we have that
\begin{equation}\label{eq:psi_gamma}
    \psi_l(t;v,u^l) = e^{-(\alpha_l(v)+\beta_l(v))t} u^l+ \left(1- e^{-(\alpha_l(v)+\beta_l(v))t}\right)\frac{\alpha_l(v)}{\alpha_l(v)+\beta_l(v)}.
\end{equation}
Since $u^l \in [0,1]$ and the functions $\alpha_l(v)$ and $\beta_l(v)$ are strictly positive, $\textrm{$\psi_l(t;v,u^l)$}$ is also strictly positive.
Moreover, it holds that
\begin{equation*}
    \psi_l(t;v,u^l) \leq  \frac{\alpha_l(v)+ e^{-t(\alpha_l(v)+\beta_l(v))} \beta_l(v)}{\alpha_l(v)+\beta_l(v)}<1,
\end{equation*}
which concludes the statement.
\end{proof}

The following \textit{state space preservation} property of the splitting schemes will serve as an essential building block for the proofs of the subsequent subsections.

\begin{proposition}[State space preservation]\label{rem:stateSpacePreservation}
    Consider the splitting schemes proposed for System~\eqref{eq:HH_redefined}\textbf{[BM]} and \eqref{eq:HH_redefined}\textbf{[OU]}. Then, for both their discrete- and continuous-time formulations, the $V$-component is non-explosive and the $U$-component remains in $[0,1]^d$.
\end{proposition}

\begin{proof}
    We focus on the $[\textbf{BM}]$-case and note that the $[\textbf{OU}]$-case can be treated analogously. 
    
    Recall the flows $\phi^{[1;BM]}$ \eqref{eq:sub1_BM_flow} and $\phi^{[2;D]}$ \eqref{eq:sub2_BM_flow}, and the compositions \eqref{eq:splittings_BM} of the splitting schemes. By Lemma \ref{lemma:subunits_stay_in_cube}, an application of the flow $\phi^{[2;D]}$ to 
    $(v,u)\in\R\times[0,1]^d$ produces a new value $(v,u')\in\R\times[0,1]^d$. In addition, an application of the flow $\phi^{[1;BM]}$ to $(v,u)\in\R\times[0,1]^d$ produces a random value $(V,u)\in\R\times[0,1]^d$. Note in particular that this procedure is by nature non-explosive, as the random value $V$ is real-valued.
    
    Therefore, alternating the flows $\phi^{[2;D]}$ and $\phi^{[1;BM]}$ produces a non-explosive dynamic with the~$U$-component remaining in the hypercube $[0,1]^d$, for all considered splitting schemes. Moreover, by the definition of the continuous-time versions of the splitting schemes (see, e.g., \eqref{eq:splittings_BM_cont}), this result extends to this case as well. 
\end{proof}

\begin{remark}\label{remark:euler-maruyama}
Note that a result as in Proposition~\ref{rem:stateSpacePreservation} does not hold for the  Euler-Maruyama~method. For that scheme, the $l$-th component of $U$ is approximated recursively by $\widetilde{U}^l_{t_i} = \varphi_l(\widetilde{V}_{t_{i-1}},\widetilde{U}^l_{t_{i-1}})$, where $\varphi_l$ is defined as
\begin{equation*}
    \varphi_l(v,u^l) := u^l + \Delta \left(\alpha_l(v)(1-u^l) - \beta_l(v) u^l \right), \quad \forall u^l\in[0,1], \: |v| < \infty,
\end{equation*}
with $\alpha_l(v),\beta_l(v)>0$. This function, however, is not guaranteed to stay in the interval $[0,1]$. Indeed, $\varphi_l(v,u^l)>1$ whenever $u^l>(1-\Delta \alpha_l(v))/(1- \Delta \alpha_l(v)-\Delta \beta_l(v))$. If $\Delta \alpha_l(v)>1$, then the right-hand side of the inequality takes values between $0$ and $1$, and hence there exist $u\in[0,1]$ such that the trajectory $\widetilde{U}^l$ escapes the interval $[0,1]$. For reasonably small~$\Delta$, this condition implies that the trajectory of $\widetilde{U}$ is more likely to escape the set $[0,1]^d$ when one of its components is close to $1$. The larger is $\Delta$, the higher is the risk that the process escapes. 
\end{remark}

\subsection{Boundedness of moments}\label{sec:boundedMoments_splitting}

As  mentioned in Proposition \ref{rem:stateSpacePreservation} the scheme is non-explosive, more precisely the $V$-component is real-valued, and the $U$-component stays in the hypercube $[0,1]^d$. Still, this does not imply boundedness of moments for $\overline{V}$, in the sense of \eqref{eq:bounded-2-moments}. Thus, in this subsection, we work at establishing the boundedness of moments of the $V$-component for the proposed splitting schemes. For clarity of presentation, we focus on the second Lie-Trotter splitting method and note that the results for the other splitting schemes can be derived similarly.

\begin{proposition}[Boundedness of moments for $\overline{V}$]\label{thm:boundedMomentsTildeV}
    Consider the continuous-time version \eqref{eq:LT2_BM_cont} or \eqref{eq:LT2_OU_cont} of the second Lie-Trotter splitting scheme for System~\eqref{eq:HH_redefined}\textbf{[BM]} and  \eqref{eq:HH_redefined}\textbf{[OU]}, respectively. Then, $\forall p\geq 1$, it holds that 
    \begin{equation*}
        \mathbb{E}\left[ \max_{0\leq t\leq T} \left|\overline{V}_{t}\right|^{2p} \right] \leq C\left( 1+ \mathbb{E}\left[\left|V_0\right|^{2p}\right] \right), 
    \end{equation*}
    where the constant $C>0$ depends on $p$, $T$ and the coefficients. 
\end{proposition}

\begin{proof}
    We first treat the  \textbf{[BM]}-case and then the \textbf{[OU]}-case.
    \vspace{-0.3cm}
    \paragraph*{The \textbf{[BM]}-case.} We start with System \eqref{eq:HH_redefined}\textbf{[BM]} and recall the continuous-time version \eqref{eq:LT2_BM_cont} of the second Lie-Trotter splitting method in \eqref{eq:splittings_BM}. 
    Performing back iteration gives for any $0\leq t\leq T$,
    \begin{equation*}
        \overline{V}_{t} = {V}_{0} + \sum_{k=1}^{n} \int\limits_{t_{k-1}\wedge t}^{t_{k}\wedge t} a( \overline{U}_{t_{k-1}} ) \overline{V}_s ds + \sum_{k=1}^{n}  \int\limits_{t_{k-1}\wedge t}^{t_{k}\wedge t} b( \overline{U}_{t_{k-1}} ) ds + \sum_{k=1}^{n} \int\limits_{t_{k-1}\wedge t}^{t_{k}\wedge t} \Sigma( \overline{U}_{t_{k-1}} ) dW_s,
    \end{equation*}
    or in other words
    \begin{equation}
    \label{eq:dyn-tildeV}
        \overline{V}_{t} = {V}_{0} + \int_0^t\bar{a}_s\overline{V}_sds+\int_0^t\bar{b}_sds+
        \int_0^t \bar{\sigma}_sdW_s,
    \end{equation}
    if we define for any $0\leq s\leq T$
    $$\bar{a}_s:=a(\overline{U}_0)\mathbbm{1}_{\{0\}}(s)+\sum_{k=1}^{n} a(\overline{U}_{t_{k-1}}) \mathbbm{1}_{(t_{k-1},t_k]}(s),$$
    $$\bar{b}_s:=b(\overline{U}_0)\mathbbm{1}_{\{0\}}(s)+\sum_{k=1}^{n} b(\overline{U}_{t_{k-1}}) \mathbbm{1}_{(t_{k-1},t_k]}(s),$$
    and
    $$\bar{\sigma}_s:=\Sigma(\overline{U}_0)\mathbbm{1}_{\{0\}}(s)+\sum_{k=1}^{n} \Sigma(\overline{U}_{t_{k-1}}) \mathbbm{1}_{(t_{k-1},t_k]}(s).$$
    Note in particular that the stochastic integral involved in \eqref{eq:dyn-tildeV} is correctly defined as 
    $\overline{U}_{t_{i-1}}$ is $\mathcal{F}_{t_{i-1}}$-measurable for any $1\leq i\leq n$ (see e.g. \cite{kara} Section 3.2.A).
    
    Considering the $2p$-th power, we further obtain
    \begin{equation*}
        |\overline{V}_{t}|^{2p} \leq C(p) \left\{ \left|V_0\right|^{2p} + \left| \int\limits_{0}^{t} \bar{a}_s\overline{V}_s ds \right|^{2p} +\left| \int\limits_{0}^{t} \bar{b}_s ds \right|^{2p} + \left| \int\limits_{0}^{t} \bar{\sigma}_s dW_s \right|^{2p}  \right\}.
    \end{equation*}
    Using Hölder's inequality, and the fact that $|\bar{a}_s|\leq C$ and 
    $|\bar{b}_s|\leq C$ (since $a,b \in C^1$ and the component $\overline{U}$ is bounded, cf. Proposition \ref{rem:stateSpacePreservation}) 
    we get that for any time $0\leq t\leq T$
    $$
    \left| \int\limits_{0}^{t} \bar{a}_s \overline{V}_s ds \right|^{2p} +\left| \int\limits_{0}^{t} \bar{b}_s ds \right|^{2p}
    \leq C(p,T)\left[ 1 + \int_0^{t} \left|\overline{V}_s \right|^{2p} ds \right].
    $$
    Then it is possible to use  the Burkholder-Davis-Gundy inequality and the fact that
    $\left| \bar{\sigma}_s \right| \leq C$ (since $\Sigma \in C^1$ and the $\overline{U}$-component is bounded, cf. Proposition \ref{rem:stateSpacePreservation}) 
    in order to get for any stopping time $0\leq S\leq T$,
    \begin{align*}
        \mathbb{E}\left[  \max_{0\leq r\leq S}\left| \int\limits_{0}^{r} \bar{\sigma}_s dW_s \right|^{2p}   \right] \leq C(p)\,\E[S^p].
    \end{align*}
Gathering all the pieces, 
introducing the stopping times $S_k=\inf\{ t\geq 0:\,|\overline{V}_{t}|>k \}$ and $t\wedge S_k\leq T$, and using Fubini's theorem, 
we get for any $0\leq t\leq T$,
\begin{equation}
\begin{array}{lll}
\E\left[\max_{0\leq r\leq t\wedge S_k}|\overline{V}_r|^{2p}\right]
&\leq & \ds K\Big[ \E[|\overline{V}_0|^{2p}]+1  + \E\int_0^{t\wedge S_k}\big[\max_{0\leq r\leq s}|\overline{V}_r|^{2p}\big]ds  \Big]\\
\\
&\leq& \ds K\Big[ \E[|\overline{V}_0|^{2p}]+1  + \E\int_0^{t}\big[\max_{0\leq r\leq s\wedge S_k}|\overline{V}_r|^{2p}\big]ds  \Big]\\
\\
&\leq& \ds K\Big[ \E[|\overline{V}_0|^{2p}]+1  + \int_0^{t}\E\big[\max_{0\leq r\leq s\wedge S_k}|\overline{V}_r|^{2p}\big]ds  \Big],\\
\end{array}
\end{equation}
where the constant $K$ depends on $p$, $T$ and the coefficients of the SDE. We now use similar arguments as in Lemma \ref{lem:bounded_moments_true_BM} in the Appendix. First, an application of Grönwall's lemma to the real-valued function 
$t\mapsto\E\big[\max_{0\leq r\leq t\wedge S_k}|\overline{V}_r|^{2p}\big]$
gives, for any $0\leq t\leq T$ and any $k$,
$$
\E\big[\max_{0\leq r\leq t\wedge S_k}|\overline{V}_r|^{2p}\big]
\leq   K\left( \mathbb{E}\left[ |V_0|^{2p} \right] + 1 \right)  e^{KT} \leq
        C \left(  1+ \mathbb{E}\left[ \left| V_0 \right|^{2p} \right] \right).
$$
Then, by Fatou's lemma one obtains
$$
    \E\left[ \max_{0\leq t\leq T}|\overline{V}_t|^{2p} \right]=\E\left[\liminf_k\max_{0\leq t\leq T\wedge S_k}|\overline{V}_t|^{2p} \right] 
    \leq \liminf_k \E\left[\max_{0\leq t\leq T\wedge S_k}|\overline{V}_t|^{2p} \right]   \leq   C\left( 1+ \mathbb{E}\left[\left|V_0\right|^{2p}\right]\right),
$$
which is the desired result.

\vspace{-0.3cm}
\paragraph*{The \textbf{[OU]}-case.}
Now, consider System \eqref{eq:HH_redefined}\textbf{[OU]} and recall the corresponding continuous-time version \eqref{eq:LT2_OU_cont} of the splitting schemes in \eqref{eq:splittings_OU}. 
Using Lemma~\ref{lemma:boundedMomentsZ} (cf. Appendix \ref{appA}), 
the result can be shown analogously as in the \textbf{[BM]}-case. Note that the proof even simplifies, due to the case of additive noise.
\end{proof}

\subsection{Local consistency and mean-square convergence}\label{sec:mean_square_convergence_splitting}

In this section, we establish the local consistency of the proposed splitting schemes and subsequently use this result to prove their global mean-square convergence. These findings are based on the theoretical framework developed in Section \ref{sec:fund-thm}. For clarity of presentation, we focus again on the second Lie-Trotter splitting method and note that the results for the other splitting schemes can be derived in a similar manner.

\paragraph*{Local consistency of the splitting methods.}

In the following, we denote
$$
\rho^V_R:= \inf \{ 0\leq t\leq T \colon |{V}_t| \geq R \},
\quad \rho^{\overline{V}}_R:= \inf \{ 0\leq t\leq T \colon |\overline{V}_t| \geq R \},
\quad \rho_R:=\rho_R^{\overline{V}} \wedge \rho_R^V,
$$
and for any $0\leq s\leq T$,
\begin{equation*}
\begin{gathered}
    \rho_R^{s,V}:= \inf \{ s\leq t\leq T \colon |{V}_t| \geq R \}, \quad 
    \rho_R^{s,\overline{V}}:= \inf \{ s\leq t\leq T \colon |\overline{V}_{t}| \geq R \}, \quad \rho^s_R:=\rho_R^{s,\overline{V}} \wedge \rho_R^{s,V}.
\end{gathered}
\end{equation*}
As it is proven that $U$ and $\overline{U}$ remain in the hypercube (see Theorem \ref{thm:bounded_moments_true_BM} and Proposition~\ref{rem:stateSpacePreservation}, respectively), it follows that both $X=(V,U^T)^T$ and 
$\overline{X}=(\overline{V},\overline{U}^T)^T$ stay in the compact region 
$B_R(0)\times [0,1]^d$ before time $\rho_R$. 
Thus, local consistency for~$\overline{X}$ (in the sense of Definition \ref{def:one_step_consistency}) will be proven using the stopping times $\rho_R,\rho_R^s$, which share the common notation with the ones in Section~\ref{sec:fund-thm}, as they play the same role.

We also introduce the notation
\begin{equation}\label{eq:f_v}
    f(v) :=  -\textbf{diag}\left( \balpha\bigl( v \bigr) + \bbeta\bigl( v \bigr) \right),
\end{equation}
which allows lightening the computations related to the $U$-component. We start with two technical lemmas.


\begin{lemma}\label{lemma:aux_result_new2}
    Let  $X_t=(V_t,U_t^T)^T$, $t\in[0,T]$, be the solution of System \eqref{eq:HH_redefined}\textbf{[BM]} or \eqref{eq:HH_redefined}\textbf{[OU]}, respectively, and let $R>0$. For any $s,t\in[0,T]$ and any $\omega\in\Omega$ with 
    $s\leq t\leq \rho_R(\omega)$, it holds that
    \begin{equation*}
        \left| U_t(\omega)- U_s(\omega) \right|  \leq C_R(t-s),
    \end{equation*}
    where the constant $C_R>0$ depends on $R$ and the coefficients.
\end{lemma}

\begin{proof}
    Let $R>0$ fixed, pick $\omega\in\Omega$ and consider $s\leq t\leq \rho_R(\omega)$. From System \eqref{eq:HH_redefined} it immediately follows that
    \begin{equation*}
        U_t -U_s=   \int\limits_{s}^{t} \big[f(V_r) U_r +\balpha \bigl( V_r \bigr) \big] dr
    \end{equation*}
    (we have dropped the reference to $\omega$).
    Now, for any $r\in [s,t]$, we have $V_r\in B_R(0)$ and $U_r\in [0,1]^d$. As in addition the functions $f$ and $\balpha$ are of class $C^1$ the quantity 
    $|f(V_r) U_r +\balpha \bigl( V_r \bigr)|$ remains bounded, by a constant $C_R>0$ depending on $R$ and the coefficients. The result then follows from Minkowski's inequality.
\end{proof}


\begin{lemma}\label{lemma:boundedMomentsDifferencesV}
    Let  $X_t=(V_t,U_t^T)^T$, $t\in[0,T]$, be the solution of System \eqref{eq:HH_redefined}\textbf{[BM]} or \eqref{eq:HH_redefined}\textbf{[OU]}, respectively. Then, 
    for any $x\in\R\times[0,1]^d$, any $s',s,t \in [0,T]$ with $s'\leq s<t$ and any 
    stopping time~$\rho$ with 
    $s\leq \rho\leq t$, it holds that
    \begin{equation*}
        \mathbb{E}\left[ \left| V^{s',x}_\rho-V^{s',x}_s \right|^{2} \right] \leq C \left[(t-s)+(t-s)^{2}\right],
    \end{equation*}
    where the constant $C>0$ depends on $T$, $\E[|V_0|^2]$ and the coefficients. 
\end{lemma}

\begin{proof}
    In what follows the constant $C$ 
    may change from one inequality to the other.

    We first treat the  \textbf{[BM]}-case and then the \textbf{[OU]}-case.
    \vspace{-0.3cm}
    \paragraph*{The $\textbf{[BM]}$-case.}
    We start with the solution of System \eqref{eq:HH_redefined}\textbf{[BM]}, for which we have for any $s\leq \theta\leq t$ (we drop the reference to $s',x$ in the computation)
    \begin{equation*}
        V_\theta-V_s = \int_{s}^{\theta} \left[ a(U_r)V_r + b(U_r) \right] dr + \int_{s}^{\theta} \Sigma(U_r) dW_r,
    \end{equation*}
    and thus
    \begin{equation*}
        \left| V_\rho-V_s \right|^{2} \leq C \left\{ \left| \int_{s}^{\rho}  a(U_r)V_r dr \right|^{2} + \left| \int_{s}^{\rho} b(U_r) dr \right|^{2}  + 
        \left| \int_{s}^{\rho} \Sigma(U_r) dW_r \right|^{2} \right \}.
    \end{equation*}
    Using Hölder's inequality, the fact that $a \in C^1$ and 
    Theorem \ref{thm:bounded_moments_true_BM}, for the first term we obtain that
    \begin{align}
        \ds\mathbb{E}\left[  \left| \int_{s}^{\rho}  a(U_r)V_r dr \right|^{2} \right] &\leq 
        \ds\E\big[(\rho-s) \int_{s}^{\rho} |a(U_r)|^2 \left| V_r \right|^{2}  dr\big]
        \leq\ds C\E\big[(t-s) \int_{s}^{t} \max_{0\leq v\leq T}\left| V_v \right|^{2}  dr\big]
        \\
        &=\ds C(t-s) \int_{s}^{t} \E\big[\max_{0\leq v\leq T}\left| V_v \right|^{2}  dr\big]
        \leq  C (t-s)^{2}.
        \end{align}
    For the second term, we obtain that
    \begin{equation*}
        \left| \int_{s}^{\rho} b(U_r) dr \right|^{2} \leq \left( \int_{s}^{t} \left|b(U_r)\right| dr \right)^{2} \leq C^{2}(t-s)^{2}.
    \end{equation*}
    Using the Burkholder-Davis-Gundy inequality and that $\Sigma \in C^1$, for the third term we have that
    \begin{equation*}
        \mathbb{E}\left[ \left| \int_{s}^{\rho} \Sigma(U_r) dW_r \right|^{2} \right] = \mathbb{E} \int_{s}^{\rho} \Sigma^2(U_r) dr   \leq C(t-s).
    \end{equation*}
    Combining everything, we get the result.

    \paragraph*{The $\textbf{[OU]}$-case.}
    Consider the solution of System \eqref{eq:HH_redefined}\textbf{[OU]}, for which we have for any~$s\leq \theta\leq t$
    \begin{equation*}
        V_\theta-V_s = \int\limits_{s}^{\theta} \left[ a(U_r)V_r + b(U_r) + \theta(\mu-Z_r) \right] dr + \int\limits_{s}^{\theta} \sigma dW_r.
    \end{equation*}
    Using Lemma \ref{lemma:boundedMomentsZ} (cf. Appendix \ref{appA}), the result can be proved by applying similar arguments as for the \textbf{[BM]}-case.
\end{proof}

In the following, as done in Section~\ref{sec:fund-thm}, let $X_{t}^{t_{i-1},x}$ and $\overline{X}_{t}^{t_{i-1},x}$ denote the exact and numerical solution of System \eqref{eq:HH_redefined}\textbf{[BM]} or \eqref{eq:HH_redefined}\textbf{[OU]}, respectively, at time $t$, given that they were in point $x$ at time $t_{i-1}$.

We are now in position to prove local consistency for the splitting schemes. 
In the setting under consideration, it is convenient to treat the two components of the process separately. We begin by establishing local consistency for the $V$-component.

\begin{proposition}[Local consistency for $\overline{V}$]\label{thm:consistency_V_local}
      Consider the continuous-time versions~\eqref{eq:splittings_BM_cont} and \eqref{eq:splittings_OU_cont} of the second Lie-Trotter splitting scheme for System \eqref{eq:HH_redefined}\textbf{[BM]} and \eqref{eq:HH_redefined}\textbf{[OU]}, respectively. For any $R>0$, it holds that for any $i=1,\ldots,n$ and any 
      $x\in B_R(0)\times[0,1]^d$, 
    \begin{equation*}
         \mathbb{E}\left[ \; \left| {V}_{t\wedge\rho_R^{t_{i-1}}}^{t_{i-1},x} - \overline{V}_{t\wedge\rho_R^{t_{i-1}}}^{t_{i-1},x} \right|^{2} \; \right]  \leq C_R \Delta^{3}, \quad \forall t\in (t_{i-1},t_i],
    \end{equation*}
    where the constant $C_R>0$ depends on $R$, $T$ and the coefficients. 

    As a consequence, for any $R>0$, the second Lie-Trotter splitting scheme is locally consistent of order $q=3/2$ on $B_R(0)\times[0,1]^d$ for the $V$-component.
\end{proposition}

\begin{proof}
We first treat the \textbf{[BM]}-case and then the \textbf{[OU]}-case.

\paragraph*{The \textbf{[BM]}-case.} Let $1\leq i\leq n$.  In the computations the constants $C$ or $C_R$ may change from line to line, $C$ being independent of $R$. 
For any $t\in (t_{i-1},t_i]$, we have 
    \begin{equation}
        V_{t}^{t_{i-1},x}=v + \int\limits_{t_{i-1}}^{t} a(U_s^{t_{i-1},x}) V_s^{t_{i-1},x} ds + \int\limits_{t_{i-1}}^{t} b(U_s^{t_{i-1},x}) ds + \int\limits_{t_{i-1}}^{t} \Sigma(U_s^{t_{i-1},x}) dW_s, 
        \end{equation}
        and from \eqref{eq:LT2_BM_cont},
        \begin{equation}
        \overline{V}_{t}^{t_{i-1},x}=v + \int\limits_{t_{i-1}}^{t} a(u) \overline{V}_s^{t_{i-1},x} ds + \int\limits_{t_{i-1}}^{t} b(u) ds + \int\limits_{t_{i-1}}^{t} \Sigma(u) dW_s.
    \end{equation}
    Thus, we have that
    \begin{align}
    \label{eq:Vconsist-3pieces}
        & \left|  \overline{V}_{t\wedge\rho_R^{t_{i-1}}}^{t_{i-1},x} - V_{t\wedge\rho_R^{t_{i-1}}}^{t_{i-1},x} \right|^{2} \leq 
        C\left \{  \left|  \int_{t_{i-1}}^{t\wedge\rho_R^{t_{i-1}}} \left( a(u) \overline{V}_s^{t_{i-1},x}-a(U_s^{t_{i-1},x}) V_s^{t_{i-1},x} \right)   ds   \right|^{2}  \right. \notag \\ & \left. \qquad + \left| 
        \int_{t_{i-1}}^{t\wedge\rho_R^{t_{i-1}}} \left( b(u)-b(U_s^{t_{i-1},x}) \right)   ds \right|^{2} + \left| \int_{t_{i-1}}^{t\wedge\rho_R^{t_{i-1}}} \left( \Sigma(u)-\Sigma(U_s^{t_{i-1},x}) \right)   dW_s \right|^{2}  \right\}.
    \end{align}
    Using Hölder's inequality, for the first term we get
    \begin{equation}
         \left| \int_{t_{i-1}}^{t\wedge\rho_R^{t_{i-1}}} \big(a(u) \overline{V}_s^{t_{i-1},x}-a(U_s^{t_{i-1},x}) V_s^{t_{i-1},x}\big)  ds   \right|^{2} \\
          \leq \Delta \int_{t_{i-1}}^{t\wedge\rho_R^{t_{i-1}}} \left| a(u) \overline{V}_s^{t_{i-1},x}-a(U_s^{t_{i-1},x}) V_s^{t_{i-1},x}  \right|^{2} ds. 
    \end{equation}
    Moreover, we have
    \begin{align*}
        &\left| a(u) \overline{V}_s^{t_{i-1},x}-a(U_s^{t_{i-1},x}) V_s^{t_{i-1},x} \right|^2 \\ 
        & \qquad \leq  c\Big(\Big| a(u) \overline{V}_s^{t_{i-1},x}-a(u) V_s^{t_{i-1},x} \Big|^2 + \Big| a(u) V_s^{t_{i-1},x}-a(U_s^{t_{i-1},x}) V_s^{t_{i-1},x} \Big|^2\Big) \\
        & \qquad \leq c|a(u)|^2 \left| \overline{V}_s^{t_{i-1},x}- {V}_s^{t_{i-1},x} \right|^2 + c\left|{V}_s^{t_{i-1},x}\right|^2 \left| a(u)-a(U_s^{t_{i-1},x}) \right|^2. 
    \end{align*}
    In order to control the above quantity for any $t_{i-1}<s\leq t\wedge\rho_R^{t_{i-1}}$ we use several facts. First, $u\in[0,1]^d$ and $a$ is continuous so that $|a(u)|\leq C$.
    Second, ${V}_s^{t_{i-1},x}\in B_R(0)$ so that $|{V}_s^{t_{i-1},x}|\leq R$. Now, from the fact that $a\in C^1$ it is Lipschitz on the hypercube $[0,1]^d$ and using Lemma \ref{lemma:aux_result_new2} we get $| a(u)-a(U_s^{t_{i-1},x})|\leq C_R\Delta$. Thus gathering all these estimates, we get that
    \begin{align}
        \label{eq:consistV-1}
        &\left| \int_{t_{i-1}}^{t\wedge\rho_R^{t_{i-1}}} \big(a(u) \overline{V}_s^{t_{i-1},x}-a(U_s^{t_{i-1},x}) V_s^{t_{i-1},x}\big)  ds   \right|^{2} \\
        & \qquad \leq 
        C_R\left(\Delta^4+\int_{t_{i-1}}^{t\wedge\rho_R^{t_{i-1}}}
        \big| \overline{V}_s^{t_{i-1},x}- {V}_s^{t_{i-1},x} \big|^2ds\right).
    \end{align}        
    Using the fact that $b\in C^1$ and similar arguments as above 
    (in particular $| b(u)-b(U_s^{t_{i-1},x})|\leq C_R\Delta$),
    we can show that
    \begin{equation}
        \label{eq:consistV-2}
        \left|\int_{t_{i-1}}^{t\wedge\rho_R^{t_{i-1}}} \left( b(u)-b(U_s^{t_{i-1},x}) \right)   ds \right|^{2} \leq C_R\Delta^4.
    \end{equation}          
    Using now the Burkholder-Davis-Gundy and Hölder inequalities, and the fact that $\Sigma$ is of class $C^1$ (which entails $|\Sigma(u)-\Sigma(U_s^{t_{i-1},x})|\leq C_R\Delta$), we get
    \begin{equation}
    \label{eq:consistV-3}
        \ds\mathbb{E}\left[ \left| \int_{t_{i-1}}^{t\wedge\rho_R^{t_{i-1}}} \left( \Sigma(u)-\Sigma(U_s^{t_{i-1},x}) \right)  dW_s \right|^{2}  \right] \leq 
         \ds\mathbb{E}\left[  \int_{t_{i-1}}^{t\wedge\rho_R^{t_{i-1}}} \left| \Sigma(u)-\Sigma(U_s^{t_{i-1},x}) \right|^2  ds   \right] 
         \leq  C_R\Delta^{3}.
       \end{equation}
    Taking expectation of \eqref{eq:Vconsist-3pieces}, and using 
    \eqref{eq:consistV-1}, \eqref{eq:consistV-2} and \eqref{eq:consistV-3}, we get
    \begin{align*}
        \ds\E \left[ \Big| \overline{V}_{t\wedge\rho_R^{t_{i-1}}}^{t_{i-1},x} - V_{t\wedge\rho_R^{t_{i-1}}}^{t_{i-1},x} \Big|^2 \right] \leq &
        \ds \ C_R\left(\Delta^3+\E \left[ \int_{t_{i-1}}^{t\wedge\rho_R^{t_{i-1}}}
        \big| \overline{V}_s^{t_{i-1},x}- {V}_s^{t_{i-1},x} \big|^2ds \right] \right)
        \\
         \leq &\ds \ C_R\left(\Delta^3+\E \left[ \int_{t_{i-1}}^{t}
        \big| \overline{V}_{s\wedge\rho_R^{t_{i-1}}}^{t_{i-1},x}- {V}_{s\wedge\rho_R^{t_{i-1}}}^{t_{i-1},x} \big|^2ds \right] \right).
    \end{align*}
    It remains to use Fubini's theorem in order to get
    $$
    \E \left[ \Big| \overline{V}_{t\wedge\rho_R^{t_{i-1}}}^{t_{i-1},x} - V_{t\wedge\rho_R^{t_{i-1}}}^{t_{i-1},x} \Big|^2 \right] \leq 
    C_R\left(\Delta^3+\int_{t_{i-1}}^{t}
        \E \left[ \big| \overline{V}_{s\wedge\rho_R^{t_{i-1}}}^{t_{i-1},x}- {V}_{s\wedge\rho_R^{t_{i-1}}}^{t_{i-1},x} \big|^2 \right] ds\right).
    $$
    Now, noting that both processes $V^{t_{i-1},x}$ and $\overline{V}^{t_{i-1},x}$  stopped at $\rho_R^{t_{i-1}}$ are bounded, the statement follows by applying Grönwall's inequality. 

   \paragraph*{The \textbf{[OU]-case.}}
    For any $t\in(t_{i-1},t_i]$, we have  
    \begin{align*}
        V_{t}^{t_{i-1},x}=v + \int\limits_{t_{i-1}}^{t} a(U_s^{t_{i-1},x}) V_s^{t_{i-1},x} ds + \int\limits_{t_{i-1}}^{t} b(U_s^{t_{i-1},x}) ds + \int\limits_{t_{i-1}}^{t} \theta(\mu-Z_s^{t_{i-1},x}) ds + \int\limits_{t_{i-1}}^{t_i} \sigma dW_s, 
    \end{align*}
    and from \eqref{eq:LT2_OU_cont},
    \begin{align*}
        \overline{V}_{t}^{t_{i-1},x}=v + \int\limits_{t_{i-1}}^{t} a(u) \overline{V}_s^{t_{i-1},x} ds + \int\limits_{t_{i-1}}^{t} b(u) ds +\int\limits_{t_{i-1}}^{t} \theta(\mu-z) ds + \int\limits_{t_{i-1}}^{t} \sigma dW_s.
    \end{align*}
    Similar to before, we thus have that 
    \begin{align*}
        & \left|  \overline{V}_{t\wedge\rho_R^{t_{i-1}}}^{t_{i-1},x} - V_{t\wedge\rho_R^{t_{i-1}}}^{t_{i-1},x} \right|^{2}  \leq\left \{  \left| \int_{t_{i-1}}^{t\wedge\rho_R^{t_{i-1}}} \left( a(u) \overline{V}_s^{t_{i-1},x}-a(U_s^{t_{i-1},x}) V_s^{t_{i-1},x} \right)  ds   \right|^{2p} \right. \notag \\ & \left. \qquad + \left| \int_{t_{i-1}}^{t\wedge\rho_R^{t_{i-1}}} \left( b(u)-b(U_s^{t_{i-1},x}) \right) ds \right|^{2} + \theta^{2} \left| \int_{t_{i-1}}^{t\wedge\rho_R^{t_{i-1}}} (z-Z_s^{t_{i-1},x}) ds \right|^{2}  \right\}.
    \end{align*}
    Consider the third term, for which we obtain from Hölder's inequality and Lemma \ref{lemma:boundedMomentsDifferencesZ} (cf. Appendix~\ref{appA}), that
    \begin{equation*}
        \mathbb{E} \left[\left( \int_{t_{i-1}}^{t\wedge\rho_R^{t_{i-1}}} 
        \left| z-Z_s^{t_{i-1},x} \right|  ds \right)^{2} \right]  \leq \Delta \int_{t_{i-1}}^{t\wedge\rho_R^{t_{i-1}}} \underbrace{\mathbb{E}\left[ \left| z-Z_s^{t_{i-1},x} \right|^{2} \right]}_{\leq C\Delta} ds \leq C\Delta^{3}.
    \end{equation*}
    Thus, the result can be shown analogously to the $\textbf{[BM]}$-case.
\end{proof}

An analogous local consistency result is established for the $U$-component, relying on the corresponding local consistency result for the $V$-component (Proposition~\ref{thm:consistency_V_local}).

\begin{proposition}[Local consistency for $\overline{U}$]\label{thm:consistency_U_local}
    Consider the continuous-time versions~\eqref{eq:splittings_BM_cont} and \eqref{eq:splittings_OU_cont} of the second Lie-Trotter splitting scheme for System \eqref{eq:HH_redefined}\textbf{[BM]} and \eqref{eq:HH_redefined}\textbf{[OU]}, respectively. For any $R>0$, it holds that for any $i=1,\ldots,n$ and any 
      $x\in B_R(0)\times[0,1]^d$, 
    \begin{equation*}
        \mathbb{E}\left[ \; \left| {U}_{t\wedge\rho_R^{t_{i-1}}}^{t_{i-1},x} - \overline{U}_{t\wedge\rho_R^{t_{i-1}}}^{t_{i-1},x} \right|^{2} \; \right]  \leq C_R \Delta^{3}, \quad \forall t\in (t_{i-1},t_i],
    \end{equation*}
    where the constant $C_R>0$ depends on $R$, $T$ and the coefficients. 

    As a consequence, for any $R>0$, the second Lie-Trotter splitting scheme is locally consistent of order $q=3/2$ on $B_R(0)\times[0,1]^d$ for the $U$-component.
\end{proposition}

\begin{proof}
We first treat the  \textbf{[BM]}-case and then the \textbf{[OU]}-case.
\vspace{-0.3cm}
\paragraph*{The \textbf{[BM]-case.}}
For any $t\in (t_{i-1},t_i]$, we have (from \eqref{eq:LT2_BM_cont})
    \begin{align}
         U_{t}^{t_{i-1},x} &= u +\int_{t_{i-1}}^{t} \big[f(V_s^{t_{i-1},x})U^{t_{i-1},x}_s+\balpha(V_s^{t_{i-1},x}) \big] ds, \\ 
         \label{eq:overlineU_LT1}\overline{U}_{t}^{t_{i-1},x} &= u +\int_{t_{i-1}}^{t} \big[f(\overline{V}_t^{t_{i-1},x})\overline{U}^{t_{i-1},x}_s+\balpha(\overline{V}_t^{t_{i-1},x}) \big] ds.
    \end{align}
Thus, we obtain 
\begin{align}
\label{eq:consis-U-1}
\left|\overline{U}_{t\wedge\rho_R^{t_{i-1}}}^{t_{i-1},x}-U_{t\wedge\rho_R^{t_{i-1}}}^{t_{i-1},x}\right|^{2}  & \leq C\left\{ \left| \int_{t_{i-1}}^{t\wedge\rho_R^{t_{i-1}}} \left( f(\overline{V}_{t\wedge\rho_R^{t_{i-1}}}^{t_{i-1},x}) \overline{U}_s^{t_{i-1},x} - f\bigl( V_s^{t_{i-1},x} \bigr) U_s^{t_{i-1},x} \right) ds\right|^{2} \right. \notag \\ 
    & \left.  \qquad   + \left| \int_{t_{i-1}}^{t\wedge\rho_R^{t_{i-1}}} \left( \balpha \bigl( \overline{V}_{t\wedge\rho_R^{t_{i-1}}}^{t_{i-1},x} \bigr) -  \balpha \bigl( V_s^{t_{i-1},x} \bigr) \right) ds \right|^{2}\right\}. 
\end{align}
Using H\"older's inequality, for the first term we get
\begin{align*}
    &
    \left| \int_{t_{i-1}}^{t\wedge\rho_R^{t_{i-1}}} \left( f(\overline{V}_{t\wedge\rho_R^{t_{i-1}}}^{t_{i-1},x}) \overline{U}_s^{t_{i-1},x} - f\bigl( V_s^{t_{i-1},x} \bigr) U_s^{t_{i-1},x} \right) ds\right|^{2}\\
    & \qquad \leq \Delta \int_{t_{i-1}}^{t\wedge\rho_R^{t_{i-1}}} \left|f(\overline{V}_{t\wedge\rho_R^{t_{i-1}}}^{t_{i-1},x}) \overline{U}_s^{t_{i-1},x} - f\bigl( V_s^{t_{i-1},x} \bigr) U_s^{t_{i-1},x}\right|^{2}  ds.
\end{align*}
 Thus, using Proposition \ref{rem:stateSpacePreservation} which entails $|\overline{U}_s^{t_{i-1},x}|\leq \sqrt{d}$, and the fact that  $f$ is of class $C^1$, we  obtain for any $t_{i-1}\leq s\leq \rho_R^{t_{i-1}}$ (remember also that 
 $V_s^{t_{i-1},x},\overline{V}_s^{t_{i-1},x}\in B_R(0)$),
\begin{align*}
    &  \big|\,f(\overline{V}_{t\wedge\rho_R^{t_{i-1}}}^{t_{i-1},x}) \overline{U}_s^{t_{i-1},x} - f\bigl( V_s^{t_{i-1},x} \bigr) U_s^{t_{i-1},x}\,\big|\\
    & \qquad \leq \left|f(\overline{V}_{t\wedge\rho_R^{t_{i-1}}}^{t_{i-1},x}) \overline{U}_s^{t_{i-1},x} -f(V_s^{t_{i-1},x})\overline{U}_s^{t_{i-1},x}\right|  +\left|f(V_s^{t_{i-1},x})\overline{U}_s^{t_{i-1},x}- f\bigl( V_s^{t_{i-1},x} \bigr) {U}_s^{t_{i-1},x}\right|\\
    & \qquad \leq \left|\overline{U}_s^{t_{i-1},x}\right|\left|f(\overline{V}_{t\wedge\rho_R^{t_{i-1}}}^{t_{i-1},x})-f\bigl( {V}_s^{t_{i-1},x} \bigr) \right|+ \left|f(V_s^{t_{i-1},x})\right|\left|\overline{U}_s^{t_{i-1},x}-U_s^{t_{i-1},x}\right| \\
    & \qquad \leq C_R \left\{   \left| \overline{V}_{t\wedge\rho_R^{t_{i-1}}}^{t_{i-1},x} -  {V}_{t\wedge\rho_R^{t_{i-1}}}^{t_{i-1},x} \right| +  \left| {V}_{t\wedge\rho_R^{t_{i-1}}}^{t_{i-1},x} -  {V}_s^{t_{i-1},x} \right| + \left|\overline{U}_s^{t_{i-1},x}-U_s^{t_{i-1},x}\right|\right\}.
\end{align*}
Thus, applying Proposition \ref{thm:consistency_V_local} and Lemma \ref{lemma:boundedMomentsDifferencesV} gives that 
\begin{align*}
    &\mathbb{E} \left[ \Big| \int_{t_{i-1}}^{t\wedge\rho_R^{t_{i-1}}} \Big( f(\overline{V}_{t\wedge\rho_R^{t_{i-1}}}^{t_{i-1},x}) \overline{U}_s^{t_{i-1},x} - f\bigl( V_s^{t_{i-1},x} \bigr) U_s^{t_{i-1},x} \Big) ds\Big|^{2} \right] \\
    & \qquad \leq  C_R \Delta \mathbb{E}\Big[\int_{t_{i-1}}^{t\wedge\rho_R^{t_{i-1}}}  \Big|\overline{V}_{t\wedge\rho_R^{t_{i-1}}}^{t_{i-1},x}-{V}_{t\wedge\rho_R^{t_{i-1}}}^{t_{i-1},x}\Big|^{2}  ds \Big] + C_R \Delta \mathbb{E}\Big[\int_{t_{i-1}}^{t\wedge\rho_R^{t_{i-1}}}  \Big|{V}_{t\wedge\rho_R^{t_{i-1}}}^{t_{i-1},x}-{V}_s^{t_{i-1},x}\Big|^2 ds \Big] \\
    & \qquad \qquad + C_R \; \mathbb{E}\Big[\int_{t_{i-1}}^{t\wedge\rho_R^{t_{i-1}}}  \Big|\overline{U}_s^{t_{i-1},x}-{U}_s^{t_{i-1},x}\Big|^{2p} ds \Big]\\
    & \qquad \leq  C_R \Delta \mathbb{E}\Big[\int_{t_{i-1}}^{t_i}  \Big|\overline{V}_{t\wedge\rho_R^{t_{i-1}}}^{t_{i-1},x}-{V}_{t\wedge\rho_R^{t_{i-1}}}^{t_{i-1},x}\Big|^{2} ds \Big] + C_R \Delta \mathbb{E}\Big[\int_{t_{i-1}}^{t_i}  \Big|{V}_{t\wedge\rho_R^{t_{i-1}}}^{t_{i-1},x}-{V}_s^{t_{i-1},x}\Big|^2 ds \Big]  \\
    & \qquad \qquad + C_R \; \mathbb{E}\Big[\int_{t_{i-1}}^{t}  \Big|\overline{U}_{s\wedge\rho_R^{t_{i-1}}}^{t_{i-1},x}-{U}_{s\wedge\rho_R^{t_{i-1}}}^{t_{i-1},x}\Big|^{2p} ds \Big] 
\end{align*}
Moreover, using Fubini's theorem, we further obtain that
\begin{align*}
    &\mathbb{E} \left[ \Big| \int_{t_{i-1}}^{t\wedge\rho_R^{t_{i-1}}} \Big( f(\overline{V}_{t\wedge\rho_R^{t_{i-1}}}^{t_{i-1},x}) \overline{U}_s^{t_{i-1},x} - f\bigl( V_s^{t_{i-1},x} \bigr) U_s^{t_{i-1},x} \Big) ds\Big|^{2} \right] \\
    & \qquad \leq  C_R \Delta \int_{t_{i-1}}^{t_i}  \mathbb{E}\Big[\Big|\overline{V}_{t\wedge\rho_R^{t_{i-1}}}^{t_{i-1},x}-{V}_{t\wedge\rho_R^{t_{i-1}}}^{t_{i-1},x}\Big|^{2}  \Big] ds  + C_R \Delta \int_{t_{i-1}}^{t_i}  \mathbb{E}\Big[\Big|{V}_{t\wedge\rho_R^{t_{i-1}}}^{t_{i-1},x}-{V}_s^{t_{i-1},x}\Big|^2 \Big] ds  \\
    & \qquad \qquad + C_R \; \int_{t_{i-1}}^{t} \mathbb{E}\Big[ \Big|\overline{U}_{s\wedge\rho_R^{t_{i-1}}}^{t_{i-1},x}-{U}_{s\wedge\rho_R^{t_{i-1}}}^{t_{i-1},x}\Big|^{2p}  \Big] ds \\
    & \qquad \leq C_R\Delta^{3} + C_R \int_{t_{i-1}}^{t} \mathbb{E}\Big[ \Big|\overline{U}_{s\wedge\rho_R^{t_{i-1}}}^{t_{i-1},x}-{U}_{s\wedge\rho_R^{t_{i-1}}}^{t_{i-1},x}\Big|^{2} \Big] ds.
\end{align*}

We now turn to the second term in \eqref{eq:consis-U-1}. In the same fashion as above, using Hölder's inequality and this time the fact that $\balpha$ is $C^1$, we obtain
\begin{align*}
    & 
    \left| \int_{t_{i-1}}^{t\wedge\rho_R^{t_{i-1}}} \left( \balpha \bigl( \overline{V}_{t\wedge\rho_R^{t_{i-1}}}^{t_{i-1},x} \bigr) -  \balpha \bigl( V_s^{t_{i-1},x} \bigr) \right) ds \right|^{2}
    \leq \Delta \int_{t_{i-1}}^{t\wedge\rho_R^{t_{i-1}}} \Big| \balpha(\overline{V}_{t\wedge\rho_R^{t_{i-1}}}^{t_{i-1},x})-\balpha\bigl( V_s^{t_{i-1},x} \bigr) \Big|^{2}  ds\\
    &\qquad \leq 
    C_R\,\Delta \int_{t_{i-1}}^{t\wedge\rho_R^{t_{i-1}}}\left\{ \left| \overline{V}_{t\wedge\rho_R^{t_{i-1}}}^{t_{i-1},x} - V_{t\wedge\rho_R^{t_{i-1}}}^{t_{i-1},x} \right|^2 + \left| {V}_{t\wedge\rho_R^{t_{i-1}}}^{t_{i-1},x} - V_s^{t_{i-1},x} \right|^2\right\}ds.
\end{align*}
Applying again Proposition \ref{thm:consistency_V_local} and Lemma \ref{lemma:boundedMomentsDifferencesV} gives that 
\begin{align*}
    &\mathbb{E}\left[ 
    \Big| \int_{t_{i-1}}^{t\wedge\rho_R^{t_{i-1}}} \Big( \balpha \bigl( \overline{V}_{t\wedge\rho_R^{t_{i-1}}}^{t_{i-1},x} \bigr) -  \balpha \bigl( V_s^{t_{i-1},x} \bigr) \Big) ds \Big|^{2}
    \right] \\ 
    & \qquad \leq  C_R \Delta \int_{t_{i-1}}^{t_i} 
    \mathbb{E}\Big[\Big|\overline{V}_{t\wedge\rho_R^{t_{i-1}}}^{t_{i-1},x}-{V}_{t\wedge\rho_R^{t_{i-1}}}^{t_{i-1},x}\Big|^{2}  \Big] 
    ds 
    + C_R \Delta \int_{t_{i-1}}^{t_i} 
     \mathbb{E}\Big[\Big|{V}_{t\wedge\rho_R^{t_{i-1}}}^{t_{i-1},x}-{V}_s^{t_{i-1},x}\Big|^2 \Big]
    ds \\
    & \qquad \leq C\Delta^{3}.
\end{align*}
Combining everything, for all $t \in (t_{i-1},t_i]$, we obtain
\begin{align*}
    \mathbb{E}\Big[ \; \big| \overline{U}_{t\wedge\rho_R^{t_{i-1}}}^{t_{i-1},x} - {U}_{t\wedge\rho_R^{t_{i-1}}}^{t_{i-1},x} \big|^{2} \; \Big] 
    \leq C_R\left(\Delta^3+\int_{t_{i-1}}^{t}
        \E \Big[ \big| \overline{U}_{s\wedge\rho_R^{t_{i-1}}}^{t_{i-1},x}- {U}_{s\wedge\rho_R^{t_{i-1}}}^{t_{i-1},x} \big|^2 \Big] ds\right).
\end{align*}
Finally, applying Grönwall's inequality, proves the statement for the second Lie-Trotter scheme. 

\paragraph*{The [\textbf{OU}]-case.}

Since the $U$-component of the splitting scheme admits the same expression in the [\textbf{OU}]-case as in the [\textbf{BM}]-case, the result follows analogously.
\end{proof}

We now combine the previous results to conclude that, for any $R>0$, the second Lie-Trotter splitting scheme is locally consistent of order $q=3/2$ on  $B_R(0)\times[0,1]^d$.

\begin{proposition}[Local consistency for $\overline{X}$]\label{thm:consistency_local}
    Consider the continuous-time versions \eqref{eq:splittings_BM_cont} and \eqref{eq:splittings_OU_cont} of the second Lie-Trotter splitting scheme for System \eqref{eq:HH_redefined}\textbf{[BM]} and \eqref{eq:HH_redefined}\textbf{[OU]}, respectively. For any $R>0$, it holds that for $i=1,\ldots,n$ and any $x\in B_R(0) \times [0,1]^d$,  
    \begin{equation*}
        \left( \mathbb{E} \left[ \Big| {X}_{t_i \wedge \rho_R^{t_{i-1}}}^{t_{i-1},x} - \overline{X}_{t_i \wedge \rho_R^{t_{i-1}}}^{t_{i-1},x} \Big|^2 \right] \right)^{1/2} \leq C_R \Delta^{3/2},
    \end{equation*}
    where the constant $C_R>0$ depends on $R$, $T$ and the coefficients. 
\end{proposition}

\begin{proof}
    This result follows directly from Propositions \ref{thm:consistency_V_local} and \ref{thm:consistency_U_local}.
\end{proof}

\begin{remark}\label{rem:add_mult_noise_det}
    In the case of additive noise for System \eqref{eq:HH_redefined}\textbf{[BM]}, that is $\Sigma(u)\equiv \sigma>0$, it follows from the proof of Proposition \ref{thm:consistency_V_local} that local consistency with rate $q=2$ is achieved for the $V$-component as the third term in \eqref{eq:Vconsist-3pieces} disappears. For the $U$-component, however, the rate remains $q=3/2$ even in this setting. By contrast, in the deterministic case, a rate of $q=2$ can be achieved for both components of the system, as Lemma \ref{lemma:boundedMomentsDifferencesV} (which is used in the proof of Proposition \ref{thm:consistency_U_local}) can be modified accordingly. These theoretical findings are in agreement with the convergence behaviour observed in numerical simulations (cf.~Section \ref{sec:num_exp}).
\end{remark}

\begin{remark}
    Note that the order in which the splitting methods are composed affects the order of the local consistency proofs. For example, for the first Lie-Trotter splitting, local consistency must first be established for $\overline{U}$ and then used to prove local consistency for $\overline{V}$. In the case of the Strang splitting, the proof proceeds in three corresponding steps. Nevertheless, each step can be proved in a similar way as in Propositions \ref{thm:consistency_V_local} and \ref{thm:consistency_U_local} above.
\end{remark}

\paragraph*{Mean-square convergence of the splitting methods.}

Proposition \ref{thm:consistency_local}, together with Theorem~\ref{thm:conv-speed-loc} (cf. Section \ref{sec:fund-thm}), immediately implies that
\begin{equation*}
    \left( \mathbb{E}\left[ \Big| X_{t_i}-\overline{X}_{t_i} \Big|^2 \mathds{1}_{t_i < \rho_R} \right]  \right)^{1/2} \leq C \Delta^{1/2},
\end{equation*}
suggesting a mean-square convergence rate of $1/2$ (resp. $1$ for the deterministic case) for the proposed splitting schemes, provided that the processes remain in a compact set. 
Indeed, this rate is observed in the numerical experiments on stochastic Hodgkin-Huxley systems presented in~Section~\ref{sec:num_exp}. 

In the following theorem, we establish (global) mean-square convergence of the splitting~schemes.

\begin{theorem}
    Let $X$ be the solution of System \eqref{eq:HH_redefined}\textbf{[BM]} and \eqref{eq:HH_redefined}\textbf{[OU]}, respectively, and $\overline{X}$ be the corresponding second Lie-Trotter splitting scheme. Then, it holds that
    \begin{equation*}
        \sup_{i=0,\ldots,n}\E\left[ |X_{t_i}-\overline{X}_{t_i}|^2 \right] \xrightarrow[\Delta\downarrow 0]{}0,
    \end{equation*}
    that is, the splitting scheme is mean-square convergent.
\end{theorem}

\begin{proof}
    The statement follows by combining Theorem \ref{thm:bounded_moments_true_BM} (resp. Theorem \ref{thm:existence-and-bounds-OU}), Proposition \ref{thm:boundedMomentsTildeV}, Proposition~\ref{thm:consistency_local} (which implies Theorem \ref{thm:conv-speed-loc}) and Theorem~\ref{thm:conv-nospeed} (cf.~Section \ref{sec:fund-thm}).
\end{proof}

\subsection{Ergodicity}
\label{sec:ergodicity_splitting}

In this section, we prove that the constructed splitting methods are geometrically ergodic, in the sense of \cite{Mattingly2002} and \cite{Ableidinger2017}. According to their analyses, proving this property boils down to verifying that the respective numerical solution satisfies three conditions, a \textit{Lyapunov}, a \textit{smoothness}, and an \textit{irreducibility} condition. Note that, while the latter two are typically satisfied by standard approximation methods, such as Euler-Maruyama, the Lyapunov condition may not be met by standard numerical methods in the locally Lipschitz case.

\paragraph*{Lyapunov condition.}

We start by verifying that the constructed splitting methods satisfy the following property.

\begin{proposition}[Discrete Lyapunov condition]\label{lemma:lyapunov_BM_HH}
     Consider the function $L:\mathbb{R} \times [0,1]^d \to [1,\infty)$ given~by 
    \[
        L(x) = 1 + |v|,
    \]
    which satisfies $L(x)\to \infty$ as $|v| \to \infty$. Further, let  $(\widetilde{X}_{t_i})_{i=0,\ldots,n}$ be a splitting solution of System~\eqref{eq:HH_redefined}\textbf{[BM]} or \eqref{eq:HH_redefined}\textbf{[OU]}, defined through 
    \eqref{eq:splittings_BM} or \eqref{eq:splittings_OU}, respectively. Then, it holds that
    \begin{align*} 
    \mathbb{E}\left[L\bigl(\widetilde{X}_{t_i}\bigr) \bigm| \widetilde{X}_{t_{i-1}}\right] &\leq C_1 L\bigl( \widetilde{X}_{t_{i-1}}\bigr) + C_2, \quad \forall i \in \mathbb{N}, 
    \end{align*}
    where the constants $C_1 \in (0,1)$ and $C_2 \geq 0$ depend solely on the coefficients. 
\end{proposition}

\begin{proof}   
    We first treat the  \textbf{[BM]}-case and then the \textbf{[OU]}-case.
    \vspace{-0.3cm}
    \paragraph*{The \textbf{[BM]-case.}}    
     We start with System \eqref{eq:HH_redefined}\textbf{[BM]} and consider the second Lie-Trotter splitting method \eqref{eq:BM_HH_LT2}, whose $V$-component can be expressed as
    \begin{equation*}
        \widetilde{V}_{t_i} = e^{a\left(\widetilde{U}_{t_{i-1}}\right)\Delta } \widetilde{V}_{t_{i-1}} + \frac{b\bigl(\widetilde{U}_{t_{i-1}}\bigr)}{a\bigl(\widetilde{U}_{t_{i-1}}\bigr)} \left( e^{a\left(\widetilde{U}_{t_{i-1}}\right)\Delta} -1 \right) + \nu_{\textrm{BM}}\bigl(\Delta;\widetilde{U}_{t_{i-1}} \bigr)\xi_{i-1},
    \end{equation*}
    where $\xi_{i-1}\sim \mathcal{N}(0,1)$, by considering the law of the stochastic term and \eqref{eq:nu_variance_BM}.
    Using that $a$ is strictly negative, that $a,b,\Sigma \in C^1$, that the component $\widetilde{U}$ is bounded (cf. Proposition \ref{rem:stateSpacePreservation}) and 
    that $0<(e^x-1)/x<1$ for all $x<0$, we obtain
    \begin{equation*}
        L\bigl( \widetilde{X}_{t_i} \bigr) = 1 + \left|  \widetilde{V}_{t_{i}} \right| \leq 1+ e^{-C_a \Delta}  \left| \widetilde{V}_{t_{i-1}} \right| + C_b \Delta + C_\Sigma \sqrt{\Delta} \left| \xi_{i-1} \right|,
    \end{equation*}
    where $C_a,C_b,C_\Sigma>0$ depend on the coefficients $a$, $b$ and $\Sigma$, respectively.
    Since $\xi_{i-1}$ is standard normal, the distribution of $|\xi_{i-1}|$ is half-normal with expectation equal to $\sqrt{2/\pi}$. Thus,
    \begin{equation*}
         \mathbb{E}\left[L\bigl(\widetilde{X}_{t_i}\bigr) \bigm| \widetilde{X}_{t_{i-1}}\right] \leq \underbrace{e^{-C_a \Delta}}_{=C_1 \in (0,1)} L\bigl( \widetilde{X}_{t_{i-1}} \bigr)   + \underbrace{1+ C_b\Delta + C_\Sigma\sqrt{\Delta}\sqrt{\frac{2}{\pi}}}_{=C_2\geq 0},
    \end{equation*}
   as required. The result follows analogously for the first Lie-Trotter \eqref{eq:BM_HH_LT1} and Strang \eqref{eq:BM_HH_S} schemes. 

\paragraph*{The \textbf{[OU]}-case.}

Now, consider System \eqref{eq:HH_redefined}\textbf{[OU]} and focus on the second Lie-Trotter splitting method \eqref{eq:OU_HH_LT2}, 
whose $V$-component can be expressed as
\begin{equation*}
        \widetilde{V}_{t_i} = e^{a\left(\widetilde{U}_{t_{i-1}}\right)\Delta } \widetilde{V}_{t_{i-1}} + \frac{ \left(b\bigl(\widetilde{U}_{t_{i-1}}\bigr) 
+\theta \bigl(\mu-Z_{t_{i-1}}\bigr) \right)}{a\bigl(\widetilde{U}_{t_{i-1}}\bigr)} \left( e^{a\left(\widetilde{U}_{t_{i-1}}\right)\Delta} -1 \right) + \nu_{\textrm{OU}}\bigl(\Delta;\widetilde{U}_{t_{i-1}} \bigr)\xi_{i-1},
    \end{equation*}
where $\xi_{i-1} \sim \mathcal{N}(0,1)$, by considering the law of the stochastic term and \eqref{eq:nu_variance_OU}. Noting that $\mathbb{E}\left[ |Z_{t_{i-1}}| \right]\leq C(\theta,\mu,\sigma,\mathbb{E}[|Z_0|^2])$ (see Lemma \ref{lemma:boundedMomentsZ} in Appendix \ref{appA}) and, similarly to the previous case, that $\bigl| \nu_{\textrm{OU}}\bigl(\Delta;\widetilde{U}_{t_{i-1}} \bigr) \bigr|\leq \sigma\sqrt{\Delta}$, the result can be shown in the same way as above. Moreover, it follows again analogously for the first Lie-Trotter \eqref{eq:OU_HH_LT1}  
and Strang \eqref{eq:OU_HH_S} 
schemes.
\end{proof}

\begin{remark}
    A notable feature of the proposed splitting schemes, compared to Euler-Maruyama type methods, is that a discrete Lyapunov condition (Proposition \ref{lemma:lyapunov_BM_HH}) can be established directly for the same Lyapunov function as considered for the true process (see \cite{Hoepfner2016}), without any restrictions on the time-step size~$\Delta$. 
\end{remark}

\paragraph*{Smoothness condition.} 

The smoothness condition can be proved in a similar way as in Corollary~7.4 of \cite{Mattingly2002} and Section 6 of \cite{Ableidinger2017}. In particular, one has to verify that there exists a finite number $k$ of time steps such that the transition probability of the respective numerical solution $(\widetilde{X}_{t_i})_{i=0,\ldots,n}$, given by
\begin{equation}\label{eq:transProb}
    \widetilde{P}_{t_i}(x,\mathcal{A}):=\mathbb{P}( \widetilde{X}_{t_i} \in \mathcal{A} | \widetilde{X}_0=x ), \quad  x \in \mathcal{S}, \ \mathcal{A} \in \mathcal{B}(\mathcal{S}),
\end{equation}
where $\mathcal{B}(\mathcal{S})$ denotes the Borel sigma algebra on the state space $\mathcal{S}=\mathbb{R}\times[0,1]^d$, admits a smooth density for $i \geq k$. This property is denoted by $k$-\textit{step hypoellipticity} in \cite{Buckwar2022}.

For all constructed splitting methods, a stochastic term enters the $V$-coordinate at every iteration step (cf. formulas \eqref{eq:sub1_BM_flow} and \eqref{eq:sub1_OU_flow}, respectively). Moreover, the $V$-coordinate enters into all components of the function $\psi_\textrm{D}$, defining the corresponding  $U$-coordinate (cf. formulas \eqref{eq:sub2_BM_flow} and \eqref{eq:psi_gamma}). Thus, it follows from the expressions \eqref{eq:BM_HH_LT1} and \eqref{eq:OU_HH_LT1} of the first Lie-Trotter splitting methods for System~\eqref{eq:HH_redefined}\textbf{[BM]} and \eqref{eq:HH_redefined}\textbf{[OU]}, respectively, that $\widetilde{P}_{t_i}(x,\mathcal{A})$ \eqref{eq:transProb} has a smooth density for $i\geq k=2$, i.e. the first Lie-Trotter schemes are $2$-step hypoelliptic. Moreover, from \eqref{eq:BM_HH_LT2}, \eqref{eq:BM_HH_S} and \eqref{eq:OU_HH_LT2}, \eqref{eq:OU_HH_S}, respectively, it is evident that for the second Lie-Trotter and Strang schemes, we have that $\widetilde{P}_{t_i}(x,\mathcal{A})$ \eqref{eq:transProb} has a smooth density for $i\geq k=1$, i.e. these schemes are even $1$-step hypoelliptic.

\begin{remark}
    The property of $1$-step hypoellipticity of the second Lie-Trotter and Strang schemes may be beneficial in statistical applications, e.g. when using the numerical method to approximate the ($1$-step) transition densities of the process, which are used to derive a likelihood function. Note that the Euler-Maruyama scheme, when applied to System \eqref{eq:HH_redefined}, is only $2$-step hypoelliptic.
\end{remark}

\paragraph*{Irreducibility condition.}  

The irreducibility condition can be also treated as in Corollary~7.4 of \cite{Mattingly2002} and Section 6 of \cite{Ableidinger2017}. In particular, one has to verify that any subset of the state space $\mathcal{A}\in \mathcal{B}(\mathcal{A})$ can be reached from any starting point $x \in \mathcal{S}$ in a finite number of steps, i.e. there exists a finite number $k$ such that $\widetilde{P}_{t_i}(x,\mathcal{A})>0$ for $i\geq k$. 

Note that the (stochastic) function $\psi_{\textrm{BM}}$ in \eqref{eq:sub1_BM_flow} (resp. $\psi_{\textrm{OU}}$ in  \eqref{eq:sub1_OU_flow}), defining the $V$-coordinate of the splitting methods, follows a normal distribution, which may ``send'' it to any point in $\mathbb{R}$ with positive probability at every iteration step. 

Thus, it follows again from the expressions \eqref{eq:BM_HH_LT1} and \eqref{eq:OU_HH_LT1} of the first Lie-Trotter splitting methods that $\widetilde{P}_{t_i}(x,\mathcal{A})>0$ for $i \geq k=2$. Moreover, for the other splitting schemes \eqref{eq:BM_HH_LT2}, \eqref{eq:BM_HH_S} and \eqref{eq:OU_HH_LT2}, \eqref{eq:OU_HH_S}, respectively, we obtain $\widetilde{P}_{t_i}(x,\mathcal{A})>0$ for $i\geq k=1$.

 \section{Numerical experiments: Illustration on the stochastic Hodgkin-Huxley model}\label{sec:num_exp}

In this section, we investigate the stochastic Hodgkin-Huxley model (cf. Section \ref{sec:HHmodel}) and compare the performance of the proposed Lie-Trotter (\texttt{LT}) and Strang (\texttt{S}) splitting methods with that of different explicit Euler-Maruyama type methods. Specifically, we include the standard Euler-Maruyama (\texttt{EM}) method \citep{Kloeden1992}, a tamed variant (\texttt{TEM}) \citep{Hutzenthaler2012}, a diffusion tamed method (\texttt{DTEM}) \citep{Tretyakov2013}, as well as a truncated \texttt{EM} (\texttt{TrEM}) scheme \citep{Hutzenthaler2012_2} and diffusion truncated (\texttt{DTrEM}) method \citep{Hutzenthaler2012_2}. A concise overview of these schemes can be found in Section~7.1 of \cite{Buckwar2022}.

Section \ref{sec:NumExp_rates} presents the empirically observed mean-square convergence rates of all considered numerical methods, while Section \ref{sec:NumExp_dynamics} investigates their ability to preserve the qualitative dynamics of the model as the step size $\Delta$ increases and studies their robustness with respect to different choices of the initial value~$V_0=v_0 \in \mathbb{R}$.

Throughout this section, we focus on System \eqref{eq:HH_redefined}\textbf{[BM]} with $\Sigma(u)\equiv \sigma>0$, and remark that similar observations are obtained for System \eqref{eq:HH_redefined}\textbf{[BM]} with $\Sigma(u)=\sigma \norm{u}^2$ as well as System~\eqref{eq:HH_redefined}\textbf{[OU]} (figures not shown here). All reference paths are generated using the \texttt{TrEM} scheme with a very small step size $\Delta$. We verified that using a different numerical scheme as reference method does not affect the results.

\subsection{Mean-square convergence rates}
\label{sec:NumExp_rates}

To assess the mean-square convergence order of the numerical schemes, we consider the root mean-squared error (RMSE), defined by 
\begin{equation}
\label{eq:RMSE}
    \textrm{RMSE}(\Delta):= \max\limits_{i=0,\ldots,n} \left( \frac{1}{M} \sum\limits_{m=1}^{M} \left| X_{t_i}(\omega_m) - \widetilde{X}_{t_i;\Delta} (\omega_m) \right|^2 \right)^{1/2},
\end{equation}
where $X_{t_i}(\omega_m)$ denotes the $m$-th simulated reference path at time $t_i$, and $\widetilde{X}_{t_i;\Delta} (\omega_m)$ its numerical approximation, obtained using the time step size $\Delta$.

\begin{figure}
	\begin{centering}
		\includegraphics[width=0.49\textwidth]{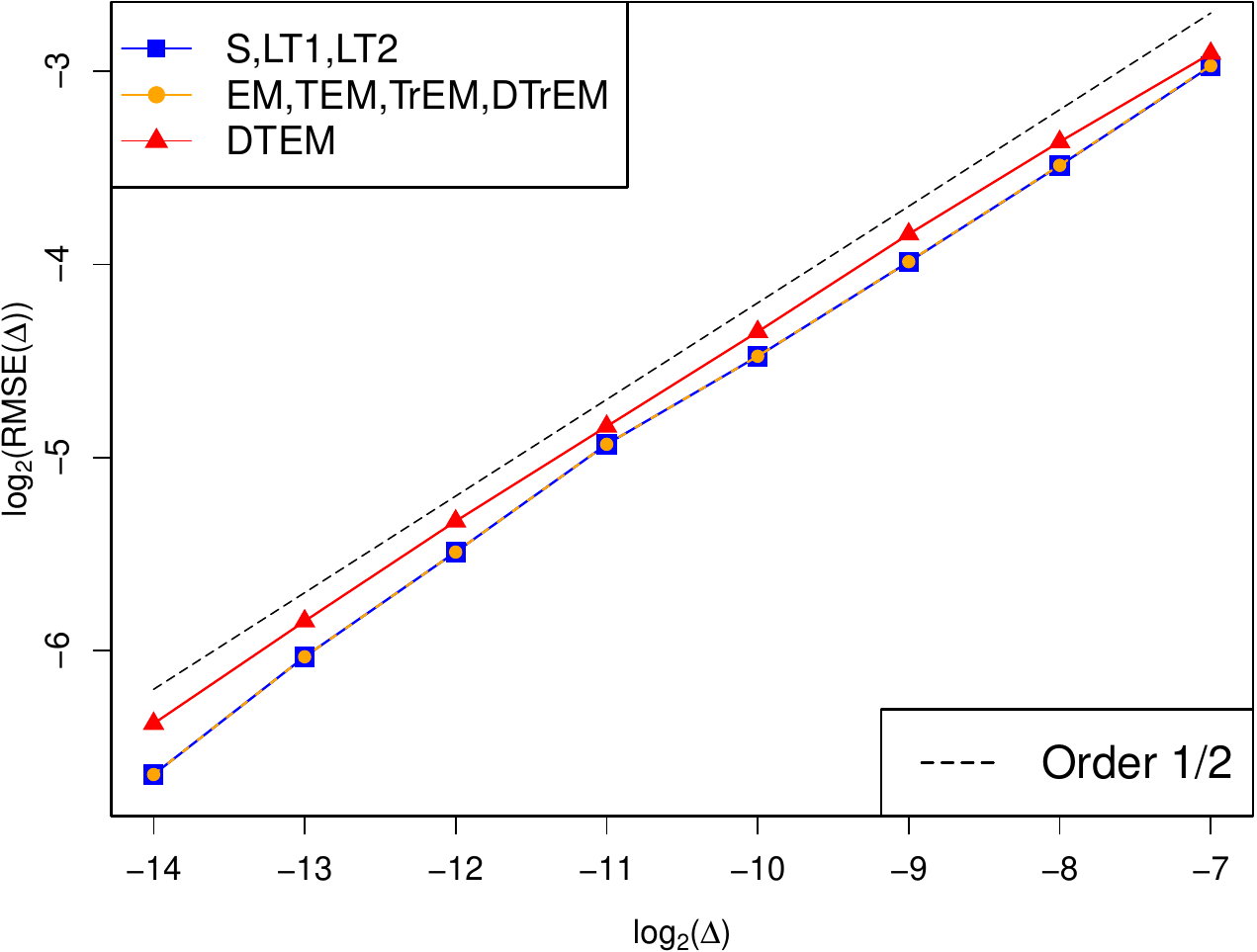}
		\includegraphics[width=0.49\textwidth]{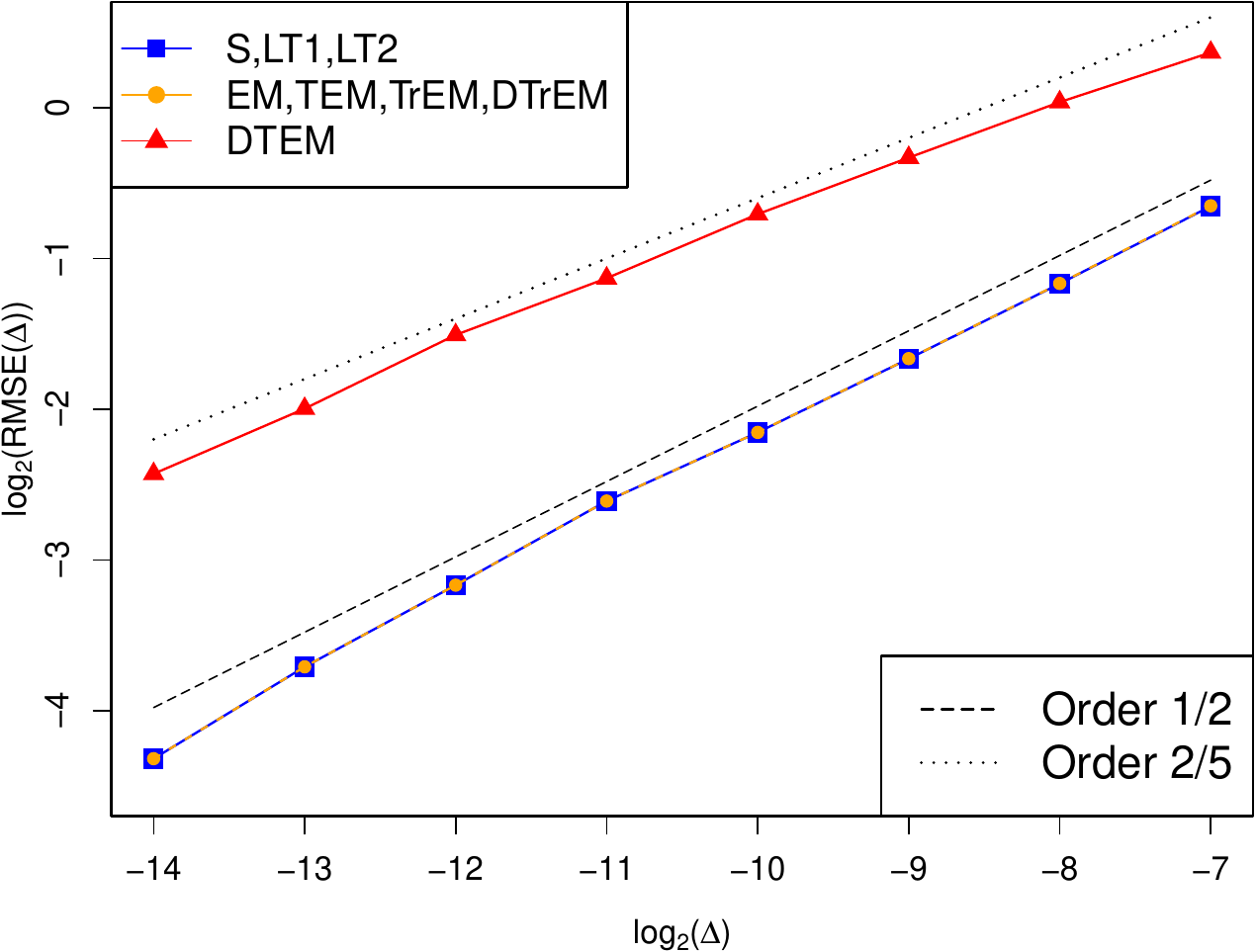}
		\caption{Mean-square convergence rates for the stochastic Hodgkin Huxley model, illustrated via the RMSE \eqref{eq:RMSE}. All parameters are set to $1$, except for $\sigma=5$ in the right panel, $T=1$, $X_0=(0,0,0,0)^T$, and $M=10^3$.  The corresponding reference paths have been obtained with the $\texttt{TrEM}$ scheme using $\Delta=2^{-16}$.}
		\label{fig:convergence_rates}
	\end{centering}
\end{figure}

In Figure \ref{fig:convergence_rates}, we report the RMSEs of the considered numerical methods in log$_2$-scale as functions of $\Delta$. We use $T=1$, $X_0=(0,0,0,0)^T$,  $M=10^3$, and set all model parameters to $1$ (with $\sigma=5$ in the right panel). 

The proposed Lie-Trotter and Strang splitting schemes exhibit a mean-square convergence rate of $1/2$, which agrees with the theoretical local rate $q-1=3/2-1=1/2$ established in Section~\ref{sec:mean_square_convergence_splitting}. Their RMSEs are nearly indistinguishable, and therefore appear as a single line in the figure (blue line). Moreover, their errors closely align with those produced by most Euler-Maruyama type methods, whose RMSEs are likewise nearly identical and plotted together as one line (orange line). 

The only exception is the \textrm{DTEM} method, which yields slightly larger errors (red line). For small noise intensities ($\sigma=1$, left panel) it reaches an empirical rate close to $1/2$. However, for lager noise intensities, its performance degrades. When $\sigma=5$ (right panel), the observed rate drops to approximately $2/5$, and it decreases further for even larger values of $\sigma$ (figures not shown here). This behaviour is consistent with the fact that the $\textrm{DTEM}$ scheme already exhibits a reduced convergence rate relative to other Euler-Maruyama type methods in the one-sided Lipschitz setting (cf. Section~7 of \cite{Buckwar2022}). In contrast, the proposed splitting methods maintain a mean-square convergence rate of $1/2$ regardless of the noise intensity $\sigma$.

Finally, we note that (figures not shown here) a convergence rate of $1/2$ is observed for all splitting methods in both the additive and multiplicative noise cases, whereas our experiments for the deterministic system indicate a rate of $1$ for the Lie-Trotter schemes (for the Strang scheme we even observe an empirical rate between $3/2$ and $2$ in the deterministic case). These observations are also in accordance with the theoretical results, cf. Remark \ref{rem:add_mult_noise_det}.

\subsection{Qualitative model dynamics}
\label{sec:NumExp_dynamics}

In this subsection, we demonstrate that the proposed splitting methods substantially outperform the Euler-Maruyama type schemes in their ability to preserve the qualitative dynamics of the stochastic Hodgkin-Huxley model. Throughout these experiments, we choose parameter values that yield neuronal spiking behaviour: $V_{rest} = -65$, ${g}_K = 36$, ${g}_{Na} = 120$, ${g}_L = 0.3$, $E_K = -77$, $E_{Na} = 55$, $E_L = -61$ and $I = 10$. 

\begin{figure}
	\begin{centering}
		\includegraphics[width=0.95\textwidth]{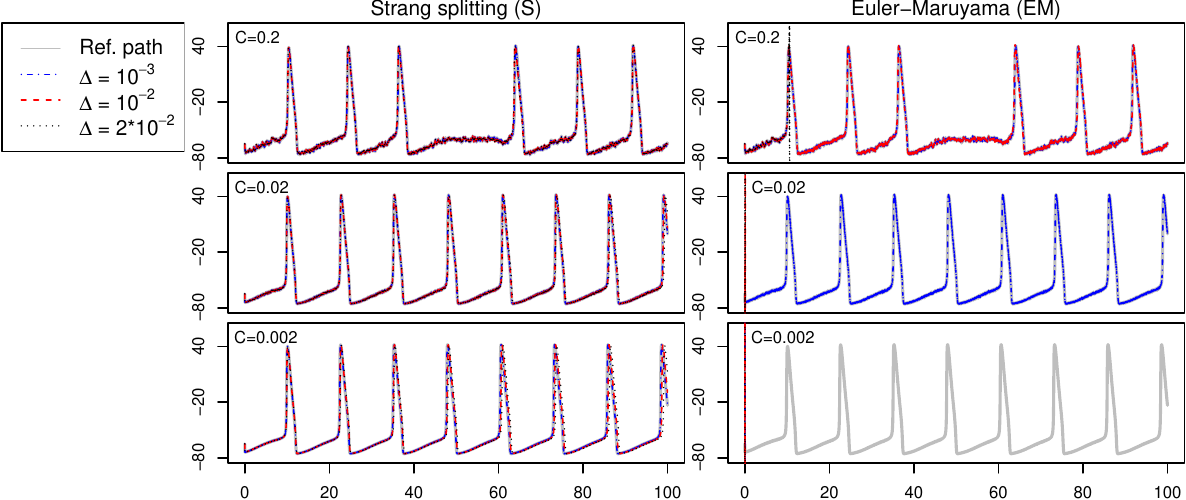}\\
		\vspace{0.2cm}
		\includegraphics[width=0.95\textwidth]{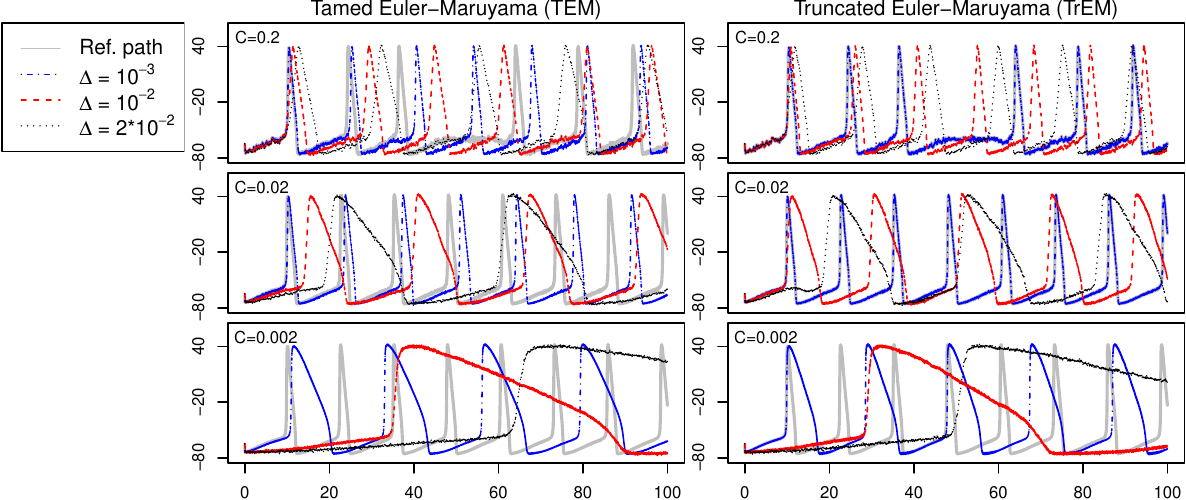}\\
		\vspace{0.2cm}
		\includegraphics[width=0.95\textwidth]{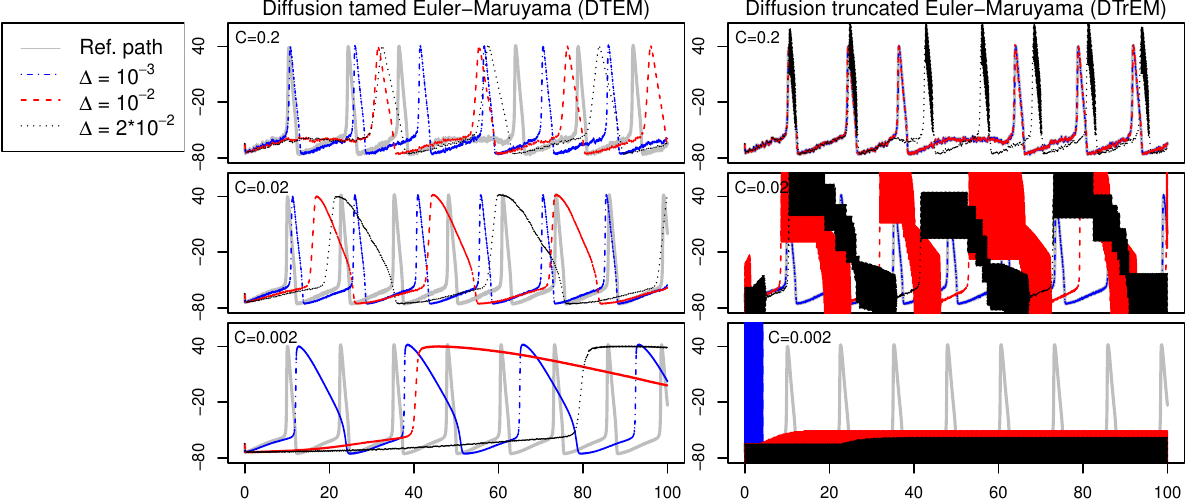}
		\caption{Paths of the $V$-component of the stochastic Hodgkin-Huxley model obtained under different values for the parameter $C$ and the time step $\Delta$. The corresponding reference path has been obtained with the $\texttt{TrEM}$ scheme using $\Delta=10^{-5}$. All paths are obtained under the same random realization.}
		\label{fig:paths}
	\end{centering}
\end{figure}

In Figure \ref{fig:paths}, we report paths of the $V$-component of the model, obtained under different values of the membrane capacitance $C$ (respective top/middle/bottom panels: $C=0.2/0.02/0.002$) and for increasing time steps $\Delta$ (blue dotted-dashed lines: $\Delta=10^{-3}$, red dashed lines: $\Delta=10^{-2}$, black dotted lines: $\Delta=2\cdot 10^{-2}$). Only the splitting methods (figure shown for the Strang scheme, though the  Lie-Trotter schemes behave analogously) produce paths that reliably overlap with the reference path for all values of $C$ and $\Delta$. This demonstrates that the splitting schemes are robust under extreme parameter settings and able to reproduce the correct dynamics of the model even for large values of $\Delta$. Thus, they are both reliable and computationally efficient. 

The standard \texttt{EM} method reproduces the reference path only when $\Delta$ is sufficiently small. However, it becomes unstable and explodes as $\Delta$ increases. For $C=0.002$, it even explodes for all tested step sizes $\Delta$. This instability is related to the fact that the \texttt{EM} method does not preserve the state space of the model, yielding values for the $U$-component outside the unit cube (cf. Remark \ref{remark:euler-maruyama}).

The tamed and truncated schemes ensure that the $U$-components remain inside the unit cube by construction. However, as noted in the literature, taming or truncating modifications on the \texttt{EM} scheme, may lead to the non-preservation of important model dynamics (see, e.g., \cite{Buckwar2022} and \cite{Kelly2022}). This is confirmed by our experiments. In particular, both tamed schemes underestimate the frequency of the neural spikes and fail to reproduce the reference path across all tested values of $C$ and $\Delta$. While their accuracy may improve for sufficiently small $\Delta$, such reductions severely increase computational cost. A slightly better performance is observed for the truncated schemes. The \texttt{TrEM} method accurately matches the reference path for $\Delta=10^{-3}$ when $C=0.2$ or $C=0.02$, but fails when $\Delta$ is increased. The \texttt{DTrEM} method reproduces the desired path in some configurations, e.g. for $C=0.2$ it remains accurate when $\Delta=10^{-2}$. However, for other settings (e.g., $C=0.02/0.002)$, it breaks down and produces spurious oscillations. This is in agreement with the literature, where similar issues have been reported for that scheme (see, e.g., \cite{Tretyakov2013,Buckwar2022,Kelly2022}).

\begin{figure}[t]
	\begin{centering}
		\includegraphics[width=0.95\textwidth]{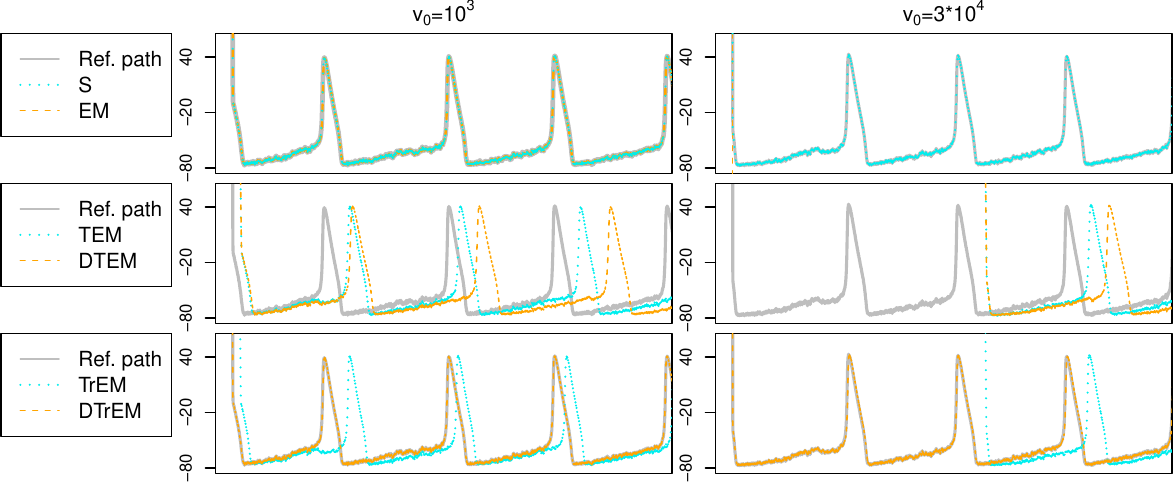}\\
		\caption{Paths of the $V$-component of the stochastic Hodgkin-Huxley model obtained under different values for the initial condition $v_0$, $\Delta=10^{-3}$ and $C=0.2$. The corresponding reference path has been obtained with the $\texttt{TrEM}$ scheme using $\Delta=10^{-6}$. All paths are obtained under the same random realization.}
		\label{fig:v0}
	\end{centering}
\end{figure}

\begin{figure}[!htbp]
	\begin{centering}
		\includegraphics[width=0.95\textwidth]{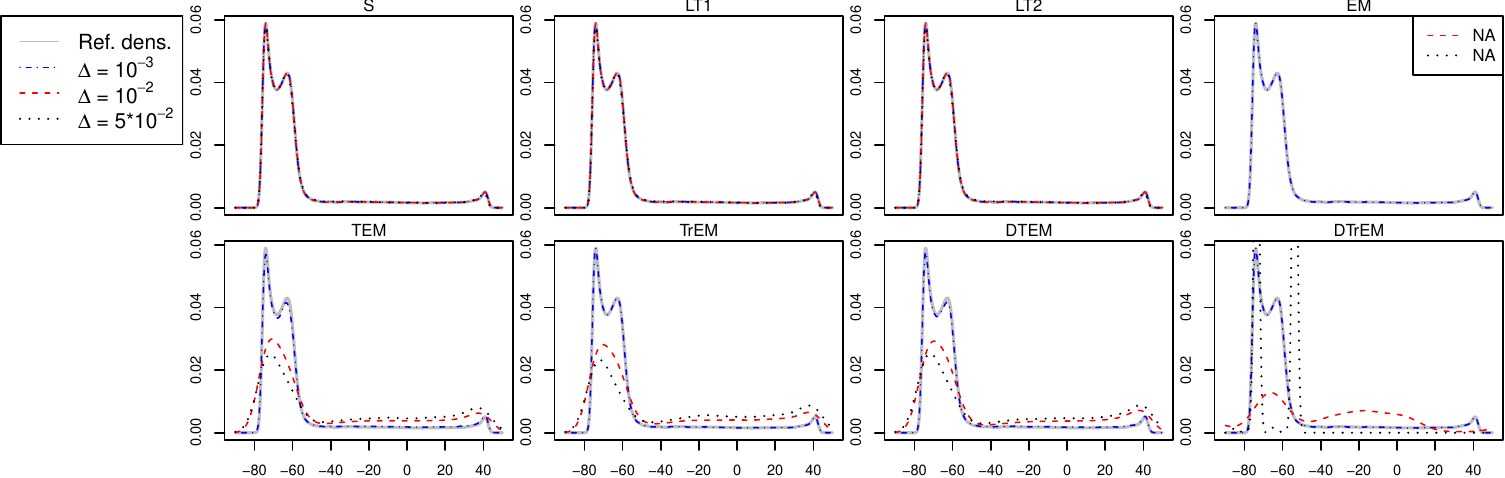}\\
		\caption{Invariant densities of the $V$-component of the stochastic Hodgkin-Huxley model obtained under different values for the time step $\Delta$ and for $C=0.02$. The reference density has been obtained with the $\texttt{TrEM}$ scheme using $\Delta=10^{-5}$.}
		\label{fig:densities}
	\end{centering}
\end{figure}

Figure \ref{fig:v0} examines robustness of the numerical schemes with respect to the initial value. We fix $\Delta=10^{-3}$ and $C=0.02$, and report sample paths of the $V$-component obtained for different values of $v_0$ (left panels: $v_0=10^3$, right panels: $v_0=3\cdot 10^4$). Once again, the splitting methods (figure shown for the Strang scheme, though the  Lie-Trotter schemes behave analogously) reproduce the reference path for both values of $v_0$, being thus also robust under extreme choices of the initial condition. This robustness is important because single paths of the $V$-component may attain arbitrarily large values, yet by ergodicity they should return to the compact invariant set. These observations are thus consistent with the ergodic property established for the splitting schemes. In contrast, all Euler-Maruyama type methods, except \texttt{DTrEM}, fail to reproduce the reference path under large $v_0$. However, as noted above, the \texttt{DTrEM} method exhibits severe issues under other parameter choices.

Finally, in Figure \ref{fig:densities}, we fix $C=0.02$ and compare invariant densities of the $V$-component estimated from long-time simulations obtained under different values of $\Delta$. All splitting methods reproduce the reference invariant density for all tested values of $\Delta$. In contrast, the Euler-Maruyama type methods provide accurate approximations of the invariant density only for $\Delta=10^{-3}$. For larger step sizes, their estimates deviate substantially. Note that, for the standard \texttt{EM} method, invariant densities cannot be computed for $\Delta=10^{-2}$ and $\Delta=2\cdot 10^{-5}$ due to numerical blow-up.

\section{Conclusion}\label{sec:conclusion}

In this work, we introduced a novel localized version of the fundamental theorem of mean-square convergence for stochastic differential equations (SDEs) with locally Lipschitz coefficients. The theorem states that if a numerical scheme for the SDE is locally consistent in the mean-square sense of order $q>1$, then it is locally mean-square convergent with order $q-1$. If, in addition, the $2p$th moments of both the exact and numerical solution are bounded for some $p>1$, then global mean-square convergence follows. We illustrated this general convergence result on a class of Hodgkin-Huxley type SDEs with conditionally linear drift, motivated by stochastic neuronal modeling.

For this class of locally Lipschitz SDEs, we constructed numerical splitting methods based on Lie-Trotter and Strang compositions, exploiting their conditional linear structure. Using the proposed convergence framework, we established mean-square convergence of the splitting schemes. In addition, we proved key structure-preserving properties, including the preservation of the state space and geometric ergodicity. These results provide a rigorous mathematical foundation for the numerical approximation of stochastic variants of Hodgkin-Huxley type systems.

Our numerical experiments support the theoretical findings and demonstrate that the proposed splitting schemes substantially outperform Euler-Maruyama type methods in preserving the qualitative features of neuronal spiking dynamics. In particular, the splitting methods accurately capture oscillation frequencies, amplitudes, and phases, even for comparatively large time steps, resulting in significant computational savings. This efficiency is especially relevant in computational neuroscience, for instance in the simulation of large neuronal networks, or when embedding the numerical methods into simulation-based inference procedures. Similar advantages of splitting methods have been observed for related neuronal models, such as the FitzHugh-Nagumo system; see, e.g.,  \cite{Buckwar2022,Foster2024}. 

Moreover, in our numerical study, we observed that the Lie-Trotter and Strang splittings exhibit similar performance in the considered setting. This contrasts with other contexts where Strang splittings have been observed to be superior to Lie-Trotter methods; see, e.g., \cite{Ableidinger2017,Tubikanec2020,Foster2024}. Given that Strang methods are slightly more computationally demanding due to the required half-step evaluations of subequations, and taking into account the hypoellipticity  of the schemes (see Section~\ref{sec:ergodicity_splitting}), the second Lie–Trotter methods may be the preferred choice.

It is also noteworthy that, due to the conditional linear structure of the system, the convergence and structure-preservation results can be established in a largely analogous manner for different Lie-Trotter and Strang compositions. This stands in contrast to parts of the existing literature, where proofs often depend sensitively on the choice of decomposition order and may be substantially more involved for Strang schemes; see, e.g., \cite{Buckwar2022,Pilipovic2024}. 

The scope of the present work goes beyond the stochastic Hodgkin-Huxley settings considered here. The proposed construction of splitting methods may be adapted to other stochastic variants of System \eqref{eq:HH_redefined}. More generally, the localized fundamental convergence theorem and the associated analytical techniques are applicable to a wide range of locally Lipschitz SDEs and numerical methods, and in particular to splitting schemes for other classes of conditionally linear systems. The continuous-time formulations of the splitting schemes used to establish local consistency and boundedness of moments may also be useful for the analysis of related numerical methods.

We believe that the combination of structure-preserving numerical schemes and a convergence theory tailored to locally Lipschitz coefficients provides a useful contribution to the reliable numerical approximation of complex SDE models arising in a wide range of applications.


\vspace{1.0cm}
\noindent\textbf{Acknowledgements.} 
P.É. and A.M. have been partially supported by the  MATH-AmSud 23-MATH-12 SMILE  project of the MATH-AmSud Program. 


\bibliography{literature}


\newpage
\appendix

\section{Auxiliary results}\label{appA}

\subsection{Proof of Lemma \ref{lem:Uin01} \citep{Hoepfner2012}}

\begin{proof}
Once a path $\omega\in\Omega$ is fixed, the function $t\mapsto V_t(\omega)$ is continuous from time interval $[0,s]$ to~$\R$ for any 
$s<\tau(\omega)$. Besides, from the second line of System \eqref{eq:HH_redefined}, and because we are considering times before explosion time, one has for any $1\leq l\leq d$, 
denoting $u_0^l=U_0^l(\omega)\in[0,1]$,
$$
U^l_t(\omega)=u_0^l+\int_0^t[-(\alpha_l(V_r(\omega))+\beta_l(V_r(\omega)))U^l_r+\alpha_l(V_r(\omega))]dr,\quad 0\leq t\leq s.
$$
In the sequel, to ease readability, we drop the reference to $\omega$. Using variation of constant we have 
$$
\begin{array}{lll}
U^l_t&=&\ds u_0^l e^{-\int_0^t(\alpha_l(V_r)+\beta_l(V_r))dr}
+\int_0^t\alpha_l(V_u)e^{-\int_u^t(\alpha_l(V_r)+\beta_l(V_r))dr}du\\
\\
&=& \ds u_0^l e^{-\int_0^t(\alpha_l(V_r)+\beta_l(V_r))dr}
+\int_0^t\frac{\alpha_l(V_u)}{\alpha_l(V_u)+\beta_l(V_u)}(\alpha_l(V_u)+\beta_l(V_u))e^{-\int_u^t(\alpha_l(V_r)+\beta_l(V_r))dr}du,
\end{array}
$$
from which we deduce, using $0\leq \alpha_l(v)/(\alpha_l(v)+\beta_l(v))\leq 1$, $\forall v\in\R$, that
$$
0\leq U^l_t\leq 
u_0^l e^{-\int_0^t(\alpha_l(V_r)+\beta_l(V_r))dr}
+\int_0^t(\alpha_l(V_u)+\beta_l(V_u))e^{-\int_u^t(\alpha_l(V_r)+\beta_l(V_r))dr}du.
$$
Performing integration for the right-hand side term of the last term of the above inequality we get
$$
0\leq U^l_t\leq 1 + (u_0^l-1)e^{-\int_0^t(\alpha_l(V_r)+\beta_l(V_r))dr}.
$$
Using $u_0^l\in[0,1]$ and the assumptions on $\balpha$ and $\bbeta$ we get the result.
\end{proof}


\subsection{Proof of Lemmas \ref{lem:bounded_moments_true_BM} and \ref{lem:bounded_moments_OU}}

\begin{proof}[Proof of Lemma \ref{lem:bounded_moments_true_BM}]
We set $\Psi^p_k(t)=\E[\max_{0\leq s\leq t\wedge S_k}|V_s|^{2p} ]$.
Let $p\geq 1$ fixed. We denote $\Psi_k=\Psi^p_k$ in order to lighten notations. 
        We have for any $0 \leq t \leq T$ satisfying $t<\tau$,
    \begin{equation*}
        V_t = V_0 +  \int\limits_{0}^{t} a(U_s)V_s ds  + \int\limits_{0}^{t} b(U_s) ds  +  \int\limits_{0}^{t} \Sigma(U_s) dW_s 
    \end{equation*}
    and
    \begin{equation*}
        |V_t|^{2p} \leq C(p) \left\{ \left|V_0\right|^{2p} + \left| \int\limits_{0}^{t} a(U_s)V_s ds \right|^{2p} +\left| \int\limits_{0}^{t} b(U_s) ds \right|^{2p} + \left| \int\limits_{0}^{t} \Sigma(U_s) dW_s \right|^{2p}  \right\}.
    \end{equation*}
    Using Hölder's inequality, Lemma \ref{lem:Uin01} and that $a\in C^1$, we get
    \begin{equation*}
        \left| \int\limits_{0}^{t} a(U_s)V_s ds \right|^{2p} \leq \left(  \int\limits_{0}^{t} \underbrace{\left| a(U_s) \right|}_{\leq C} \left| V_s \right| ds  \right)^{2p} \leq C^{2p} t^{2p-1} \int\limits_{0}^{t} \left| V_s \right|^{2p} ds \leq C  \int\limits_{0}^{t} \left| V_s \right|^{2p} ds.
    \end{equation*}
    Moreover, since $b\in C^1$, we have that
    \begin{equation*}
       \left| \int\limits_{0}^{t} b(U_s) ds \right|^{2p} \leq \left(  \int\limits_{0}^{t} \underbrace{\left| b(U_s) \right|}_{\leq C } ds  \right)^{2p} \leq Ct^{2p} \leq C. 
    \end{equation*}
    Using the Burkholder-Davis-Gundy inequality, Hölder's inequality, Lemma \ref{lem:Uin01} and that $\Sigma \in C^1$, we conclude that for any stopping time $S<\tau$ (and with $S\leq T$):
    \begin{align*}
        \mathbb{E}\left[  \max_{0\leq r\leq S}\left| \int\limits_{0}^{r} \Sigma(U_s) dW_s \right|^{2p}   \right] 
        &\leq C(p) \mathbb{E}\left[  \left(  \int\limits_{0}^{S} \underbrace{\left|  \Sigma(U_s)  \right|^2}_{\leq C} ds   \right)^p \right] 
        \leq C\E[ S^{p}]. 
    \end{align*}
    We now use the above inequalities at (stopping) time  $t\wedge S_k$ for $0\leq t\leq T$. Note that~$t\wedge S_k <~\tau$ and $t\wedge S_k\leq T$. We have
\begin{align*}
\Psi_k(t)&\leq  \ds K
\left( \E[|V_0|^{2p}] +1  + \E \left[ \int_0^{t\wedge S_k}|V_s|^{2p}ds \right] \right)\\
&\leq \ds K \left(\E[|V_0|^{2p}] + 1  + \E \left[ \int_0^{t}\max_{0\leq r\leq s\wedge S_k}|V_r|^{2p}ds \right] \right)\\
&\leq \ds K\left( \E[|V_0|^{2p}]+1  + \int_0^{t}\E\left[\max_{0\leq r\leq s\wedge S_k}|V_r|^{2p}\right]ds  \right),
\end{align*}
where the constant $K$ depends on $p,T, \mathbf{b}$ and $\bsigma$.

Note that at the third inequality we have used Fubini's theorem. Besides, we know that 
\begin{equation*}
    \Psi_k(s)=\E\big[\max_{0\leq r\leq s\wedge S_k}|V_r|^{2p}\big]
\end{equation*}
takes a finite value (as it is bounded by $k^{2p}$). Thus, we have 
\begin{equation*}
    \Psi_k(t)\leq K\left( \E[|V_0|^{2p}] + 1+\int_0^t\Psi_k(s)ds\right), \quad 0\leq t\leq T,
\end{equation*}
and applying Grönwall's inequality, we get that
    \begin{equation*}
          \Psi_k(t) \leq  K\left( \mathbb{E}\left[ |V_0|^{2p} \right] + 1 \right)  e^{KT} \leq
        C \left(  1+ \mathbb{E}\left[ \left| V_0 \right|^{2p} \right] \right),
    \end{equation*}
where the constant $C$ depends on $p,T, \mathbf{b}$ and $\bsigma$.
\end{proof}

\begin{proof}[Proof of Lemma \ref{lem:bounded_moments_OU}]
    From the first and third lines of System \eqref{eq:HH_redefined}$\textbf{[OU]}$ we have
    \begin{equation*}
    \begin{aligned}
    |V_t|^{2p} \leq C(p) \Biggl\{ &\, |V_0|^{2p}
    + \left| \int_{0}^{t} a(U_s)V_s \, ds \right|^{2p}
    + \left| \int_{0}^{t} b(U_s) \, ds \right|^{2p} 
    + \left| \int_0^t \theta \mu \, ds \right|^{2p} \\
    & \quad + \left| \int_0^t \theta Z_s \, ds \right|^{2p}
    + \sigma^{2p} |W_t|^{2p} \Biggr\}
    \end{aligned}
    \end{equation*}
    and
    \begin{equation*}
        |Z_t|^{2p} \leq C(p) \left\{ \left|Z_0\right|^{2p} 
        +\left|\int_0^t\theta\mu ds \right|^{2p}
        +\left|\int_0^t\theta Z_sds \right|^{2p}
        +\sigma^{2p}\left|W_t\right|^{2p}  \right\}.
    \end{equation*}
    Thus, remembering that we have set $S_k=\inf\{ t\geq 0:\,|X_t|>k \}$ one gets as in the proof of Lemma~\ref{lem:bounded_moments_true_BM}, using in particular~\eqref{eq:Uin01OU}, for any 
    $0\leq t\leq T$,
    $$
    \E\left[\max_{0\leq t\wedge S_k}|V_s|^{2p} \right]
    \leq 
     K\left\{ \E\left[|V_0|^{2p}\right]+1  + \int_0^{t}\E\left[\max_{0\leq r\leq s\wedge S_k}|V_r|^{2p}\right]ds 
     +\int_0^{t}\E\left[\max_{0\leq r\leq s\wedge S_k}|Z_r|^{2p}\right]ds 
     \right\}
    $$
and
$$
    \E\left[\max_{0\leq t\wedge S_k}|Z_s|^{2p} \right]
    \leq 
     K_Z\left\{ \E\left[|Z_0|^{2p}\right]+1 
     +\int_0^{t}\E\left[\max_{0\leq r\leq s\wedge S_k}|Z_r|^{2p}\right]ds 
     \right\},
    $$
     where the constants $K$ and $K_Z$ depend on $p,T,\mathbf{b}_{OU},\bsigma_{OU}$. Thus, setting
     $$\Psi_k(t)=\E\left[ \max_{0\leq s\leq t\wedge S_k}|V_s|^{2p}
     + \max_{0\leq t\wedge S_k}|Z_s|^{2p}\right]$$
     one gets
     \begin{equation*}
    \Psi_k(t)\leq (K\vee K_Z)\left( \E[|V_0|^{2p}] + \E[|Z_0|^{2p}] + 2+\int_0^t\Psi_k(s)ds\right), \quad 0\leq t\leq T,
\end{equation*}
     and again by Grönwall's lemma,
     \begin{equation}
     \label{eq:cont-Psi-OU}
     \Psi_k(t)\leq C(1+\E[|V_0|^{2p}] + \E[|Z_0|^{2p}]),
     \end{equation}
     where the constant $C$ depends on $p,T,\mathbf{b}_{OU},\bsigma_{OU}$.
\end{proof}

\subsection{Preliminary results for the process $Z$ in the \textbf{[OU]}-case.}\label{sec:preliminaryResults}

The following two lemmas are used in the proofs of Section \ref{sec:Splitting_properties} when considering System \eqref{eq:HH_redefined}\textbf{[OU]}.

\begin{lemma}\label{lemma:boundedMomentsZ}
    Let the process $Z$ be defined via \eqref{eq:OU}. Then, 
    $\forall p \geq 1$, it holds that 
    \begin{equation*}
        \mathbb{E}\left[ \left| Z_t \right|^{2p} \right] \leq C(1+\mathbb{E}[|Z_0|^{2p}]), \quad t>0,
    \end{equation*}
    where the constant $C>0$ depends on $p,\theta,\mu,\sigma$ (but is independent of $T$). As a consequence, 
    \begin{equation*}
        \mathbb{E}\left[ \left| Z_t \right| \right] \leq C(1+\mathbb{E}[|Z_0|^2])^{1/2}, \quad t>0,
    \end{equation*}
    where the constant $C>0$ depends on $\theta,\mu,\sigma$ (but is independent of $T$).
\end{lemma}

\begin{proof}
The Ornstein-Uhlenbeck process, defined via \eqref{eq:OU}, is for all $t>0$ given by
\begin{equation*}
    Z_t=e^{-\theta t}Z_0+\mu\theta \underbrace{\int\limits_{0}^{t} e^{-\theta (t-s)} ds}_{=\frac{1}{\theta}\left(1-e^{-\theta t}\right) \leq \frac{1}{\theta} } +\sigma \int\limits_{0}^{t} e^{-\theta(t-s)} dW_s.
\end{equation*} 
Considering the $2p$-th moment and using that $e^{-x}\leq 1$ for all $x\geq 0$ yields
\begin{equation*}
    \mathbb{E}\left[ \left|Z_t\right|^{2p} \right] \leq C(p)\left\{ \mathbb{E}[|Z_0|]^{2p} +|\mu|^{2p} + \sigma^{2p} \mathbb{E}\left[ \left| \int\limits_{0}^{t} e^{-\theta(t-s)} dW_s \right|^{2p} \right] \right\}.
\end{equation*}
Applying the Burkholder-Davis-Gundy inequality and using again that $e^{-x}\leq 1$ for $x\geq 0$ gives that
\begin{equation*}
    \mathbb{E}\Bigg[ \Bigg| \int\limits_{0}^{t} e^{-\theta(t-s)} dW_s \Bigg|^{2p} \Bigg] \leq C(p) \mathbb{E}\Bigg[ \Bigg( \int\limits_{0}^{t} e^{-2\theta(t-s)} ds \Bigg)^p \Bigg] \leq \frac{C(p)}{(2\theta)^p},
\end{equation*}
which concludes the first result. The second result follows from Jensen's inequality. 
\end{proof}

\begin{lemma}\label{lemma:boundedMomentsDifferencesZ}
    Let the process $Z$ be defined via \eqref{eq:OU}. 
    Then, for all $s,t\geq0$ with $s\leq t$, it holds that
    \begin{equation*}
        \mathbb{E}\left[ \left| Z_t-Z_s \right|^{2} \right] \leq C\left[(t-s)+(t-s)^{2}\right],
    \end{equation*}
    where the constant $C>0$ depends on $\theta,\mu,\sigma,\mathbb{E}[|Z_0|^{2}]$. 
\end{lemma}

\begin{proof}
    For the Ornstein-Uhlenbeck process, defined via \eqref{eq:OU}, it holds for all $s,t$ with $s\leq t$ that
    \begin{equation}\label{eq:DiffZ}
        Z_t-Z_s=\left( e^{-\theta(t-s)}-1 \right) Z_s+\mu\theta \int\limits_{s}^{t} e^{-\theta (t-r)}dr+\sigma\int\limits_{s}^{t} e^{-\theta(t-r)} dW_r.
    \end{equation}
    Using that $e^{-x}\leq 1$ for all $x\geq0$, and that $|e^{-x}-e^{-y}|\leq|x-y|$ for all $x,y\geq 0$, yields
    \begin{equation*}
        \mathbb{E}\left[ \left| Z_t-Z_s \right|^{2} \right] \leq C\left\{ \theta^{2}(t-s)^{2} \mathbb{E}\left[ \left| Z_s \right|^{2} \right] + |\mu|^{2}\theta^{2}(t-s)^{2} + \sigma^{2} \mathbb{E}\Bigg[ \Bigg( \int\limits_{s}^{t} e^{-\theta(t-r)} dW_r \Bigg)^2 \Bigg]  \right\}.
    \end{equation*}
    Now, from Lemma \ref{lemma:boundedMomentsZ} we obtain that
        $\mathbb{E}\left[ \left| Z_s \right|^{2} \right] \leq  C(\theta,\mu,\sigma,\mathbb{E}[|Z_0|^{2}])$. 
    Moreover, applying It\^o's isometry  
    and using again that $e^{-x}\leq 1$ for all $x\geq0$, gives that
     \begin{equation*}
       \mathbb{E}\Bigg[ \Bigg( \int\limits_{s}^{t} e^{-\theta(t-r)} dW_r \Bigg)^{2} \Bigg] = \mathbb{E}\Bigg[ \int\limits_{s}^{t} e^{-2\theta(t-r)} dr  \Bigg] =  \int\limits_{s}^{t} e^{-2\theta(t-r)} dr \leq \int\limits_{s}^{t} 1 \ dr = t-s,
    \end{equation*}
    which concludes the statement.
    %
\end{proof}


\end{document}